\documentclass[12pt]{article}
\usepackage{amsmath,amssymb,mathtools}
\usepackage{amsthm}
\usepackage{amscd}
\usepackage{epsfig}
\usepackage{graphics}
\usepackage{graphicx}
\usepackage{colordvi}
 \usepackage[bookmarks=false]{hyperref}
\usepackage{enumitem}
\usepackage{mathrsfs} 
\usepackage{dsfont}
\usepackage{mparhack}
\usepackage{tikz}
\usepackage{xcolor,lipsum}
\usepackage[displaymath, mathlines]{lineno}
\modulolinenumbers[5]
\nolinenumbers
\allowdisplaybreaks[4]

\hypersetup{
breaklinks,
colorlinks,
linkcolor=blue!75!black,
urlcolor=blue!75!black,
citecolor=blue!75!black}

\theoremstyle{plain}

\theoremstyle{definition}

\theoremstyle{remark}

\numberwithin{equation}{section}

\newcommand{\1}{\mathds 1}

\newcommand{\F}{\mathscr F}

\newcommand{\R}{{\mathbb{R}}}

\newcommand{\vep}{{\varepsilon}}

\renewcommand{\to}{\longrightarrow}

\newcommand{\lmt}{\longmapsto}

\newcommand{\sgn}{{\rm sgn}}

\newcommand{\bs}{\boldsymbol}
\newcommand{\ms}{\mathscr}
\renewcommand{\P}{{\mathbb P}}

\newcommand{\E}{{\mathbb  E}}

\newcommand{\dxi}{{\widehat{\xi}}}

\oddsidemargin
-.35in \evensidemargin -.5in\marginparwidth 0.4in

\begin{document}

\newtheoremstyle{slantthm}{10pt}{10pt}{\slshape}{}{\bfseries}{}{.5em}{\thmname{#1}\thmnumber{ #2}\thmnote{ (#3)}.}
\newtheoremstyle{slantthmp}{10pt}{10pt}{\slshape}{}{\bfseries}{}{.5em}{\thmname{#1}\thmnumber{ #2}\thmnote{ (#3)}.}
\newtheoremstyle{slantrmk}{10pt}{10pt}{\rmfamily}{}{\bfseries}{}{.5em}{\thmname{#1}\thmnumber{ #2}\thmnote{ (#3)}.}

\theoremstyle{slantthm}
\newtheorem{thm}{Theorem}[section]
\newtheorem{prop}[thm]{Proposition}
\newtheorem{lem}[thm]{Lemma}
\newtheorem{cor}[thm]{Corollary}
\newtheorem{ass}[thm]{Assumption}
\newtheorem{defi}[thm]{Definition}
\newtheorem{prob}[thm]{Problem}
\newtheorem{disc}[thm]{Discussion}
\newtheorem*{nota}{Notation}
\newtheorem*{conj}{Conjecture}
\theoremstyle{slantrmk}
\newtheorem{rmk}[thm]{Remark}
\newtheorem{eg}[thm]{Example}
\newtheorem{step}{Step}
\newtheorem{claim}{Claim}
\newtheorem{que}[thm]{Question}
\theoremstyle{plain}
\newtheorem{thmm}{Theorem}[section]

\title{{\bf Wright-Fisher diffusions for evolutionary games with death-birth updating}}

\author{Yu-Ting Chen\footnote{Department of Mathematics, University of Tennessee, Knoxville, TN, US}}

\maketitle

\vspace{-1cm}

\abstract{
We investigate spatial evolutionary games with death-birth updating in large finite populations. Within growing spatial structures subject to appropriate conditions, the density processes of a fixed type are proven to converge to the Wright-Fisher diffusions  with drift. In addition, convergence in the Wasserstein distance of the laws of their occupation measures holds. The proofs of these results develop along an equivalence between the laws of the evolutionary games and certain voter models and rely on the analogous results of voter models on large finite sets by convergences of the Radon-Nikodym derivative processes. As another application of this equivalence of laws, we show that in a general, large population of size $N$, for which the stationary probabilities of the corresponding voting kernel are comparable to uniform probabilities, a first-derivative test among the major methods for these evolutionary games is applicable at least up to weak selection strengths in the usual biological sense (that is, selection strengths of the order $\mathcal O(1/N)$). 
\\ 


}
{
\tableofcontents}

\section{Introduction and main results}\label{sec:intro}
The goal of this paper is to investigate diffusion approximations of the interacting particle systems which are known in the biological literature as evolutionary games with death-birth updating. In the Supplementary Information \cite[SI]{Ohtsuki_2006} of their seminal work on evolutionary games, Ohtsuki et al. analyze the  density processes of a fixed type in the evolutionary games with death-birth updating on random regular graphs. They find that these processes approximate the Wright-Fisher diffusions with drift in the limit of large population size, and the key argument there follows the physics method of pair approximation. This method goes back to Matsuda et al.
\cite{Matsuda_1992} for the Lotka-Volterra model and has since been applied extensively to spatial models in biology. 

In this work, we present a mathematical proof of the diffusion approximation in \cite[SI]{Ohtsuki_2006}.
The proof follows the viewpoint in \cite{CDP,Chen_2013}, where the evolutionary games are regarded as perturbations of certain  reference voter models 
and is built on the assumption that the diffusion approximation of the evolutionary games on large finite sets
holds in the special case of voter models. This assumption is supported by the results in \cite{CCC,CC_16}. There it is  proven that the diffusion approximation of voter models on large finite sets requires only  mild conditions of the underlying spatial structures and that the Wright-Fisher diffusions appear as
the universal limiting processes. 
The approach in this paper thereby develops along an equivalence between the probability laws of the evolutionary games on finite sets  and the reference voter models.

\subsection{The evolutionary games and voter models}
Throughout this paper
we consider evolutionary games on finite sets to be defined as follows.
On a finite set $E$ with size $N\geq 2$, each of the sites is occupied by an individual with one of the two types, $1$ and $0$. Individuals engage in pairwise interaction, and payoffs from this interaction follow a given payoff matrix $\Pi=\big(\Pi(\sigma,\tau)\big)_{\sigma,\tau\in \{1,0\}}$  
with real entries. Whenever an individual with type $\sigma$ and an individual with type $\tau$ interact, the individual with type $\sigma$ receives payoff $\Pi(\sigma,\tau)$. With respect to a given transition probability $q$ on $E$, the total payoff of the individual at site $x$ is given by the following weighted average provided that the population configuration is $\xi\in \{1,0\}^E$:
\begin{align}\label{weightedpay}
\sum_{y\in E}q(x,y)\Pi\big(\xi(x),\xi(y)\big).
\end{align}
An individual's total payoff enters its fitness (that is, reproductive rate), and the fitness is given by a convex combination of baseline fitness 1 and the total payoff. Here,
{\bf selection strength} $w$ is the constant weight applied to the total payoff of every individual throughout time. It is understood to be sufficiently small, relative to the entries of the payoff matrix $\Pi$, to ensure that all the fitness values are positive. 

In the above evolutionary game, players in the population are updated indefinitely according to the following rule: At the unit rate, the individual at $x$ is chosen to die. Then the individuals at all the other sites compete for reproduction to occupy the vacant site $x$ in a random fashion; the 
probability of successful reproduction of the parent at $y$, $y\neq x$, is proportional to the following product:
\begin{align}\label{fitness}
q(x,y)\cdot \mbox{(fitness of the individual at $y$)}.
\end{align}
Here and throughout this paper, we require that $q$ 
have trace (that is, $q(x,x)\equiv 0$) and
be irreducible and reversible.
See Equation~(\ref{def:Lw}) for the Markov generator of the evolutionary game.
For the purpose of this introduction,
we remark that in the above scenario,  the entire population fixates at either 
 the all-$1$ state or the all-$0$ state after a sufficiently large amount of time as a result of the assumed irreducibility of $q$. 
In addition, in certain biological contexts (cf. \cite{Nowak_2010}), significant interest in including mutation in evolutionary game dynamics exists. We will only consider models without mutations unless otherwise mentioned until Section~\ref{sec:poisson}.

In the special case of zero selection strength, the  evolutionary game introduced above simplifies to a reference voter model with {\bf voting kernel} $q$. A voter model is an oversimplified model for death and birth of species in biological systems and can be regarded as a generalization of the Moran process from population genetics \cite{Moran} on a structured population. Here, the underlying spatial structure is defined in the natural way by the nonzero entries of $q$. The canonical example is the case where $q$ is the transition kernel of a random walk on a finite, connected, simple graph. In this case, the individuals chosen to die are replaced by the children of their neighbors. 

The study of
voter models allows for several classical approaches of interacting particle systems  to start with (cf. \cite{L:IPS}), including  attractiveness and a nice duality by coalescing Markov chains driven by voting kernels both in the sense of the Feynman-Kac representation (cf. Section~\ref{sec:FK}) and in the pathwise sense
of identity by descent in population genetics (cf. \cite[Section III.6]{L:IPS} and \cite{Granovsky_1995,Mal_cot_1975, Rousset}). 
By contrast,
the game transition probabilities at positive selection strengths show configuration-dependent asymmetry arising from the differences in individuals' payoffs. In this case,
attractiveness is absent,
and exact evaluations of basic quantities 
in population genetics (e.g. absorbing probabilities and expected times to absorption) become difficult. 
See \cite{Chen_2015} for additional properties of the evolutionary games arising from the lack of attractiveness.

\subsection{Pair approximation for the evolutionary games}\label{sec:1.2}
The primary focus of this paper is an approximation method for the evolutionary games with death-birth updating in ~\cite[SI]{Ohtsuki_2006}. With the goal of quantifying the absorbing probabilities of the evolutionary games under weak selection, the analysis in \cite[SI]{Ohtsuki_2006} invokes the corresponding density processes and conditional density processes. Here, weak selection
is usually understood in the biological literature as requiring $w\leq \mathcal O(1/N)$. In addition, with respect to a voting kernel $q$ with stationary distribution $\pi$, the {\bf density} of $\sigma$'s in $\xi\in \{1,0\}^E$ is given by the following weighted average:
\begin{align}\label{def:p1}
p_\sigma(\xi)=\sum_{x\in E}\pi(x)\1_{\{\sigma\}}\big(\xi(x)\big),
\end{align}
and, with $p_{\tau\sigma}(\xi)$ defined by the weighted average
\begin{align}\label{def:p01}
p_{\tau \sigma}(\xi)=\sum_{x,y\in E}\pi(x)q(x,y)\1_{\{\tau\}}\big(\xi(x)\big)\1_{\{\sigma\}}\big(\xi(y)\big),
\end{align}
the conditional densities are defined by the ratios
\[
p_{\tau|\sigma}(\xi)=\frac{p_{\tau\sigma}(\xi)}{p_\sigma (\xi)}
\]
($0/0=0$ by convention).

The analysis in \cite[SI]{Ohtsuki_2006} 
provides diffusion approximations of the absorbing probabilities of the evolutionary game $(\xi_t)$ by means of
the same probabilities of the one-dimensional process $p_1(\xi_t)$, where the underlying spatial structure is assumed to be a large random regular graph of degree $k\geq 3$. 
As we will discuss in more detail below, the implication of pair approximation is nontrivial and is a key step to make further analysis possible in \cite[SI]{Ohtsuki_2006}. It leads to the property that the two processes $p_1(\xi_t)$ and $p_{1|0}(\xi_t)$ form a closed system. Moreover, the two processes decouple in the limit of large population size, whereas the density process $p_1(\xi_t)$ approximates a self-consistent Wright-Fisher diffusion with drift coefficient and squared noise coefficient given by 
\begin{align}\label{coefficients}
w\cdot \frac{k-2}{k^2(k-1)}p_1(\xi)[1-p_1(\xi)][\alpha p_1(\xi)+\beta]\quad\mbox{and}\quad
\frac{2(k-2)}{N(k-1)}p_1(\xi)[1-p_1(\xi)],
\end{align}
respectively. Here, the constants $\alpha$ and $\beta$ entering the drift coefficient are given by the following equations:
\begin{align}
\begin{split}\label{alphabeta}
\alpha=&(k+1)(k-2)[\Pi(1,1)-\Pi(1,0)-\Pi(0,1)+\Pi(0,0)],\\
\beta=&(k+1)\Pi(1,1)+(k^2-k-1)\Pi(1,0)-\Pi(0,1)-(k^2-1)\Pi(0,0).
\end{split}
\end{align}
See \cite[Eq. (18) in SI]{Ohtsuki_2006} for the coefficients in Equation~(\ref{coefficients}). (The differences between the coefficients in \cite[Eq. (18) in SI]{Ohtsuki_2006} and those in Equation~(\ref{coefficients}) are only attributable to the definition of total payoffs of individuals and the choice of time scales in this paper and will be explained in Remark~\ref{rmk:O}.) 
Notice that in Equation~(\ref{coefficients}), only the drift coefficient depends on the game payoffs, and the approximate diffusion process is not mean-field but incorporates the underlying spatial structure only by the simple parameter $k$.

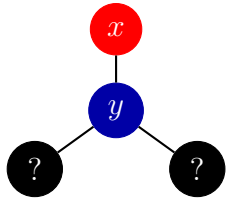
\begin{figure}
\vspace{.5cm}
\begin{center}
\begin{tikzpicture}[scale=0.6]
\path (3,1.8) node(a) [circle, draw=red, fill= red, text=white] {$x$}
(3,0) node(b) [circle, draw=blue!65!black, fill=blue!65!black, text=white] {$y$}
 (1.2,-1.3) node(d) [circle, draw=black, fill=black, text=white] {$?$}
(4.8,-1.3) node(h) [circle, draw=black, fill=black, text=white]{$?$};
\draw[-,thick] (node cs:name=a) -- (node cs:name=b);
\draw[-,thick]
 (node cs:name=d)--(node cs:name=b)
(node cs:name=h)--(node cs:name=b)
;
\end{tikzpicture}
\end{center}
\caption{\rm Site $x$ is occupied by a focal individual with type $0$. Site $y$ is occupied by an individual with type $1$. The number of types among the neighbors of $y$, excluding the one at $x$, are left to be estimated.
}
\label{fig:1}
\end{figure}

If we understand
the application of pair approximation in \cite[SI]{Ohtsuki_2006} correctly, then it can be summarized as two major mechanisms to be discussed below (they are adapted to the setup in this paper). In particular, they both involve reductions of
 the local frequencies
\begin{align}\label{def:ptau}
p_\tau(y,\xi)=\sum_{z\in E}q(y,z)\1_{\{\tau\}}\big(\xi(z)\big)=\frac{\#\{z;z\sim y,\xi(z)=\tau\}}{k}
\end{align}
of individuals with type $\tau$ to the conditional  densities $p_{\tau|\sigma}(\xi)$, where $y$'s are sites occupied by
 individuals with type $\sigma$ and $z\sim y$ means that $z$ and $y$ are neighbors to each other. (The second equality above follows since the voting weight between a site and any of its neighboring sites is $1/k$ on a $k$-regular graph.)

Now, we condition on the event that an individual randomly chosen from the entire population, called a \emph{focal} individual, is an individual with type $0$ located at site $x$.
Then the first mechanism states that
 the types of its neighbors, whose numbers can be quantified by $k$ multiplies of the local frequencies $p_1(x,\xi)$ and $p_0(x,\xi)$ as in (\ref{def:ptau}), are
i.i.d. Bernoulli distributed. Moreover, the probability of finding an individual with type $1$ is $p_{1|0}(\xi)$. Next, 
recall that the fitness of a neighbor, say at site $y$ and with type $\sigma$,
of the focal individual is by definition a convex combination of baseline fitness $1$ and the total payoff that it receives. It can be written as follows: 
\begin{align}\label{approx:fsigma0}
f_\sigma(y)=(1-w)+w\Bigg(\frac{1}{k}\Pi(\sigma,0)+p_1(y,\xi)\Pi(\sigma,1)+\Big(p_0(y,\xi)-\frac{1}{k}\Big)\Pi(\sigma,0)\Bigg).
\end{align}
Here in Equation~(\ref{approx:fsigma0}), the first payoff $\Pi(\sigma,0)$ on the right-hand side results from the interaction between the focal individual with type $0$ and the neighbor at $y$ with type $\sigma$ under consideration.
Then the second mechanism states that the fitness in Equation~(\ref{approx:fsigma0}) satisfies the following approximate equality:
\begin{align}\label{approx:fsigma}
f_\sigma(y) \simeq  (1-w)+w\Bigg(\frac{1}{k}\Pi(\sigma,0)+\frac{(k-1)p_{1|\sigma}(\xi)}{k}\Pi(\sigma,1)+\frac{(k-1)p_{0|\sigma}(\xi)}{k}\Pi(\sigma,0)\Bigg),
\end{align}
where the numbers of 
 types of the remaining $k-1$ neighbors of the individual at $y$ are now estimated by the conditional densities $p_{1|\sigma}(\xi)$ and $p_{0|\sigma}(\xi)$. See Fig~\ref{fig:1}. Similar hypotheses are in force if the focal individual is conditioned to be an individual with type $1$.
The argument in \cite[SI]{Ohtsuki_2006} further
uses the locally  
tree-like property of a large random $k$-regular graph (cf. \cite{McKay_1981}) so that 
the neighbors of the individual at site $y$, excluding the focal individual, can be neglected when fitnesses of the other neighbors of the focal individual are calculated.

The above application of pair approximation in \cite[SI]{Ohtsuki_2006} is closely related to the standard probabilistic technique of characterizing the scaling limits of stochastic processes by the corresponding martingale problems, which requires the closure of the dynamical equations under consideration.
In a general population, however, a typical statistic of the evolutionary game depends on state of the entire evolving population and the number of equations required to close its dynamics appears to grow with the population size. This fact should make clear a nontrivial mathematical issue underlying the two `quasi-mean-field' hypotheses discussed above. Yet it is not clear to us how to verify the hypotheses.

Before the present work, mathematical proofs are provided to support arguably the most important finding in \cite[SI]{Ohtsuki_2006} implied by the above diffusion approximation. That finding uses explicit solutions of the absorbing probabilities of the approximate diffusion processes with the coefficients defined in Equation~(\ref{coefficients}), and considers games of the generalized prisoner's dilemma with payoff matrices given as follows:
\begin{align}\label{eq:payoff0}
\Pi=
\bordermatrix{
&  1 & 0 \cr
1 & b-c & -c \cr
0 & b & 0 \cr},\quad b,c\in \R.
\end{align} 
Here, $b$ and $c$ are interpreted as benefit and cost, respectively, when they are strictly positive.
This finding in \cite[SI]{Ohtsuki_2006} states that for $k\geq 3$, the degree $k$ of a large random $k$-regular graph approximates a critical value concerning whether the emergence of cost-benefit effective game interactions
can improve the survival of individuals with type $1$: If $b>ck$, the survival probability of individuals with type $1$ is strictly larger than the same probability under the reference voter model. If $b<ck$, the strict inequality between the probabilities is reversed.

The work of Cox, Durrett and Perkins in \cite{CDP} obtains 
the rescaled limits of general voter model perturbations on any integer lattice of dimension $d\geq 3$ and proves related deep results. There these results are used to study long-term behaviors of the interacting particle systems. In particular, it is proven in \cite{CDP} that on any integer lattice of dimension $d\geq 3$, the graph degree $2d$ is exactly the critical value for the evolutionary game.
More precisely, the critical values in \cite{CDP} are defined in terms of the fixation of types in finite regions after a large amount of time, instead of the global fixation of types after a large amount of time. (To accommodate the transient nature of these infinite lattices,
this definition is necessary.) Other progress relating to mathematical proofs of the prediction in \cite[SI]{Ohtsuki_2006} has been within the scope of finite, simple, regular graphs and considers a first-derivative test that is used to compare absorbing probabilities at all \emph{arbitrary small} selection strengths 
by signs of their derivatives at zero selection strength (e.g. \cite{Chen_2013}). In contrast to the method in \cite[SI]{Ohtsuki_2006},
the major investigation along that first-derivative test  focuses on exact evaluations of the derivatives in finite populations (see Section~\ref{intro:3} for more details). These evaluations under all initial conditions are now complete in \cite{Chen_2015} by the duality between voter models and coalescing Markov chains.
In particular,  within the spatial structures of $k$-regular graphs, \cite{Chen_2013,Chen_2015}
recover the critical value $k$ predicted in \cite[SI]{Ohtsuki_2006} in the limit of large population size.

\subsection{Weak convergence of the game density processes}\label{intro:1.3}
The first main result  of this paper is a general theorem for diffusion approximations of the game density processes. See Theorem~\ref{thm:main2}.
An application of the theorem leads to a mathematical proof of the prediction  in \cite[SI]{Ohtsuki_2006} discussed above  when payoff matrices as in Equation~(\ref{eq:payoff0}) are in use: 
After a time change by a suitable constant multiple of $N$, the approximate diffusion process defined by the coefficients in Equation~(\ref{alphabeta}) and  subject to a selection strength as a constant multiple of $1/N$ coincides with a limiting diffusion process obtained in this paper. See Theorem~\ref{thm:O} and Remark~\ref{rmk:O}.

The proof of Theorem~\ref{thm:main2} does not invoke pathwise duality for the evolutionary game dynamics as in \cite{CDP}, in which certain branching coalescing Markov chains are used as the dual processes. By contrast, we proceed with the fact 
that the laws of the evolutionary games are equivalent to the laws of the reference voter models (Section~\ref{sec:poisson}). In this way, we can view the game density processes  in terms of the voter models, even with the limit of large population size, if the corresponding Radon-Nikodym derivative processes satisfy appropriate tightness properties. 
Then the proof of Theorem~\ref{thm:main2} turns to and relies heavily on the main result in  \cite{CCC} that diffusion approximations of the voter density processes hold on large spatial structures where the underlying voting kernels are subject to appropriate, but mild, mixing conditions; an extension in \cite{CC_16} is used when mutation is present.
In these cases, the limiting voter density processes are given by the Wright-Fisher diffusions, which originally arise from the Moran processes, namely, the voter models on complete graphs (cf. \cite{EK:MP}).
(See also the pioneering works \cite{MT,CDP} for diffusion approximations of voter models, which are for voter models defined on integer lattices and give limits as solutions to stochastic PDEs.) Moreover, the spatial structures are encoded by the time scales in these diffusion approximations and are not present in the coefficients of the limiting diffusions.
 See Theorem~\ref{thm:WF} for a restatement of these results in \cite{CCC,CC_16}. By Girsanov's theorem, characterizing subsequential limits of the game density processes is thus reduced to characterizing the covariations between the limiting voter density process and subsequential limits of the Radon-Nikodym derivative processes (Theorem~\ref{thm:main2} 2$^\circ$) and Theorem~\ref{thm:main2-1}). A certain spatial homogeneity property of the voting kernels is enough to close the 
covariations by the limiting voter density process. 

Let us give two remarks for the present method. 
First, it allows for the possibility of explicitly  characterizing the limiting Radon-Nikodym derivative process; it can take the form of  a Dol\'eans-Dade exponential martingale  explicitly defined in terms of the limiting voter density process 
(Theorem~\ref{thm:main2} 3$^\circ$)). We stress that the Radon-Nikodym derivative processes under consideration are used to change the laws of the voter models to the laws of the evolutionary games, not just to relate the laws of their density processes. Second, the only reason why we restrict our attention to the particular payoff matrices in Equation~(\ref{eq:payoff0}) is because we do not know how to close the covariations between subsequential limits of the Radon-Nikodym derivative processes and the limiting voter density process explicitly otherwise except on very special graphs. On the other hand, 
 the game density processes under general payoff matrices are tight if the voting kernels are subject to appropriate conditions, 
 and any subsequential limit is a continuous semimartingale with a Wright-Fisher martingale part (Theorem~\ref{thm:main2} 1$^\circ$)). This result proves the presence of a Wright-Fisher noise coefficient in the approximate diffusion process obtained in \cite[SI]{Ohtsuki_2006}
when individuals play games according to a general payoff matrix.

\subsection{Occupation measures of the game density processes}
As an application of the diffusion approximation of the game density processes,
we investigate the use of the absorbing probabilities of the limiting diffusions as approximate solutions for the absorbing probabilities of the evolutionary games in \cite[SI]{Ohtsuki_2006}. A similar method is  used in \cite{Sui_2015} to approximate the expected times to absorption of the evolutionary games. For these two approximations, the reader may recall the fact that 
the
weak convergence of absorbing processes does not guarantee
the weak convergence of their times to absorption in general.  
By proving a stronger tightness property of the Radon-Nikodym derivative processes at selection strengths of order $ \mathcal O(1/N)$, we show that Oliveira's result on the convergence in the Wasserstein distance of order $1$ of times to absorption in \cite{Oliveira_2012,Oliveira_2013} and the convergence of absorbing probabilities in \cite{CCC} under voter models carry to the corresponding convergences under the evolutionary games. These are included in  
the second main result, Theorem~\ref{thm:main3}, where the major theme is
around convergences of occupation measures of the game density processes. See
Cox and Perkins~\cite{Cox_2004} for a closely related result of voter models on integer lattices.

\subsection{The game absorbing probabilities}\label{intro:3}
The third main result of this paper, Theorem~\ref{thm:main1}, proves that in a large finite population, the first-derivative test discussed by the end of Section~\ref{sec:1.2} for the comparison of the game absorbing probabilities and the voter absorbing probabilities
is applicable at all selection strengths at least up to $\mathcal O(1/N)$.
This result does not require particular forms of payoff matrices. 
Theorem~\ref{thm:main1} is a long overdue result motivated by a seminar inquiry from Omer Angel several years ago when the author was a Ph.D. student. An answer to Angel's inquiry can be used to quantify the scope of the first-derivative test in terms of the strength of selection. Here in this paper, it reinforces the comparison of the game absorbing probabilities and the voter absorbing probabilities by diffusion approximation. 
Indeed, the limiting diffusions in Theorem~\ref{thm:main2} can capture
game interactions among individuals only if the selection strengths are comparable to nonzero constant multiples of $1/N$.

To find selection strengths eligible for the first-derivative test, one could use some power series of the game absorbing probabilities in selection strength, which are obtained in \cite[Proposition~3.2]{Chen_2013}. Coefficients in the series are represented as explicit functionals of the voter models. In particular, an exact computation of the first-order coefficients is possible
by the duality between voter models and coalescing Markov chains and calculations of the coalescing Markov chains. This suggests similar arguments for
all the higher-order coefficients, and then finding appropriate bounds for them is turned to. Here we obtain the bound $\mathcal O(1/N)$ for the eligible selection strengths by the equivalence of laws, since this bound seems to pose technical difficulties
for that method by the power series in \cite{Chen_2013}. After all, the dual presentations for the higher-order coefficients appear highly intricate.

\paragraph{\bf Organization of the paper} In Section~\ref{sec:poisson}, we discuss the dynamics of the evolutionary games with death-birth updating in more detail. In Section~\ref{sec:radon}, we prove some a-priori bounds for the Radon-Nikodym derivative processes between the laws of the evolutionary games and the laws of the reference voter models. Section~\ref{sec:weak} investigates convergences of the game density processes. We reinforce this result to a convergence in the Wasserstein distance of occupation measures of the game density processes in Section~\ref{sec:wass}. In Section~\ref{sec:expansion}, we present the proof that the first-derivative test discussed above is applicable for all selection strengths up to $\mathcal O(1/N)$  for suitable voting kernels. In Section~\ref{sec:flip}, we calculate first-order expansions of some covariation processes of the evolutionary games in selection strength. Section~\ref{sec:FK} gives a brief account of the Feynman-Kac duality between voter models and coalescing Markov chains. Section~\ref{sec:FK} is followed by a list of frequent notations.  

\paragraph{\bf Acknowledgements} 
Partial supports from the Center of Mathematical Sciences and Applications, the John-Templeton Foundation, and  a grant from B~Wu and Eric Larson during the author's previous positions at Harvard University are gratefully acknowledged. The author would like to thank Prof. Martin A. Nowak for bringing the attention of the paper by Sui et al.~\cite{Sui_2015}.

\section{Stochastic integral equations for the evolutionary games}\label{sec:poisson}
In this section, we describe the Markovian dynamics of an evolutionary game with death-birth updating in more detail and give a construction of the evolutionary game by Poisson calculus. 
We write $S= \{1,0\}$ from now on and recall that voting kernels are assumed to have zero traces and be irreducible and reversible throughout this paper. 

First let us specify the generator of an evolutionary game defined by   
a voting kernel $(E,q)$ and a payoff matrix $\Pi=\big(\Pi(\sigma,\tau)\big)_{\sigma,\tau\in S}$ in the presence of mutation. In this case, the fitness of an individual at $x\in E$ under population configuration $\xi\in S^E$ is given by  
\begin{align}\label{def:fitness}
f^w(x,\xi)= (1-w)+w\sum_{y\in E}q(x,y)\Pi\big(\xi(x),\xi(y)\big).
\end{align}
Here and throughout the rest of this paper, we assume that selection strengths $w$ satisfy the constraint $w\in [0,\overline{w}]$, where
\begin{align}\label{def:w0}
\overline{w}=  \Big(2+2\max_{\sigma,\tau\in S}|\Pi(\sigma,\tau)|\Big)^{-1}.
\end{align}
Hence, $f^w(x,\xi)>0$ for all these $w$'s.
We also define the population configurations $\xi^x$ and $\xi^{x|\sigma}$ as the ones obtained from $\xi$ by changing only the type at $x$, with the type $\xi(x)$ at $x$ changed to 
\begin{align}\label{def:hat}
\dxi(x)=1-\xi(x)
\end{align}
for $\xi^x$ and changed to $\sigma\in S$ for $\xi^{x|\sigma}$. 
Then given a mutation measure $\mu$ on $S$, the generator of the evolutionary game is defined as follows:
\begin{align}\label{def:Lw}
\begin{split}
&\mathsf L^{w,\mu} F(\xi)=\sum_{x\in E}c^w(x,\xi)\big(F(\xi^x)-F(\xi)\big)+\sum_{x\in E}\int_{S}\big(F(\xi^{x|\sigma})-F(\xi)\big)d\mu(\sigma)
\end{split}
\end{align}
for $F:S\to \R$,
where $c^w(x,\xi)$'s are defined as follows:
\begin{align}
q^w(x,y,\xi)&= \frac{q(x,y)f^w(y,\xi)}{\sum_{z\in E}q(x,z)f^w(z,\xi)},\label{def:qw}\\
c^w(x,\xi)&= \sum_{y\in E}q^w(x,y,\xi)\big(\xi(x)\dxi(y)+\dxi(x)\xi(y)\big).\label{def:cw}
\end{align}
Notice that the function $q^w$ defined by (\ref{def:qw}) reduces to the voting kernel $q$ if $w=0$; in this case, $\mathsf L^{0,\mu}$ is the generator  of an $(E,q,\mu)$-voter model.

Now
we recall a coupling of the $(E,q,\mu)$-voter model by stochastic integral equations, which has been used in, for example, \cite{MT} and \cite[Lemma~2.1]{CDP_2000}. 
We introduce the following independent $(\F_t)$-Poisson processes:
\begin{align}
\begin{split}\label{rates}
\Lambda_t(x,y)& \quad\mbox{with rate}\quad  \E[\Lambda_1(x,y)]=q(x,y) \quad\mbox{and}\\
 \Lambda_t^{\sigma}(x)& \quad \mbox{with rate}\quad \E[\Lambda_1^{\sigma}(x)]=\mu(\sigma),\quad x,y\in E,\;\sigma\in S,
 \end{split}
\end{align}
which are defined on a complete filtered probability space $\big(\Omega,\F,(\F_t),\P\big)$. The filtration $(\F_t)$ is assumed to satisfy the usual conditions, and we set $\F_\infty=\bigvee_{t\geq 0}\F_t$. Then given an initial condition $\xi\in S^E$, 
an $(E,q,\mu)$-voter model $(\xi_t)$ can be defined as the pathwise unique $S^E$-valued solution, with c\`adl\`ag paths, of the following system of stochastic integral equations:\vspace{-.2cm} 
\begin{align}
\begin{split}
\xi_t(x)&=\xi(x)+\sum_{y\in E}\int_0^t \big(\xi_{s-}(y)-\xi_{s-}(x)\big)d\Lambda_s(x,y)\\
&\hspace{2cm}+\int_0^t\dxi_{s-}(x)d\Lambda_s^{1}(x)-\int_0^t\xi_{s-}(x)d\Lambda_s^{0}(x),\quad x\in E.\label{eq:voter}
\end{split}
\end{align}

We can use the system in (\ref{eq:voter}) to couple the above evolutionary game  in the following way. We introduce the $(\F_t,\P)$-martingale 
\begin{align}
\begin{split}\label{def:Dwxy}
D_t^w(x,y)= &\exp\Bigg\{
\int_0^t \log \frac{q^w(x,y,\xi_{s-})}{q(x,y)}d\Lambda_s(x,y)-\int_0^t \big(q^w(x,y,\xi_s)-q(x,y)\big)ds\Bigg\}
\end{split}
\end{align} 
to change the intensity of $\Lambda(x,y)$ under $\P$ whenever $q(x,y)>0$, and set $D^w_t(x,y)\equiv 1$ otherwise (see \cite[page 473]{Revuz_2005} for the fact that $D^w(x,y)$ defines an $(\F_t,\P)$-martingale). A global change of intensities is done through the following $(\F_t,\P)$-martingale:
\begin{align}
\label{def:D}
D_t^w\stackrel{\rm def}{=} &\prod_{(x,y)\in E\times E}D_t^w(x,y).
\end{align}
Define a probability measure $\P^w$ on $(\Omega,\F_\infty)$, with expectation $\E^w$, by
\begin{align}\label{def:Pw}
d\P^w|_{\F_t}=D^w_td\P|_{\F_t}.
\end{align}
Then the process $(\xi_t)$ satisfying (\ref{eq:voter}) defines an evolutionary game under $\P^w$, and its generator is given by $\mathsf L^{w,\mu}$; recall the definition of $c^w(x,y,\xi)$ in (\ref{def:cw}). 
In more detail, it follows from Girsanov's theorem 
(cf. \cite[Theorem~III.3.11 and Theorem~III.3.17]{JS}) that for any $x,y\in E$, the jump process $\Lambda(x,y)$ under $\P^w$ is an $(\F_t)$-doubly-stochastic Poisson process with an $(\F_t)$-predictable intensity 
$\big(q^w(x,y,\xi_{t-});t\geq 0\big)$
in the sense of S. Watanabe's characterization. That is, it holds that
\[
\E^w\left[\int_0^\infty C_td\Lambda_t(x,y)\right]=\E^w\left[\int_0^\infty C_tq^w(x,y,\xi_{t})dt\right]
\]
for all nonnegative $(\F_t)$-predictable processes $(C_t)$. Notice that under $\P^w$, $\Lambda^{\sigma}(x)$ remains an $(\F_t)$-Poisson process with rate $\mu(\sigma)$. 

The following proposition gives a summary 
of the above construction.

\begin{prop}\label{prop:SDE}
For any $w\in [0,\overline{w}]$ and initial condition $\xi\in S^E$, the pathwise unique solution $(\xi_t)$ of the system (\ref{eq:voter}) under $\P^w$ is a jump Markov process with its generator given by $\mathsf L^{w,\mu}$. 
\end{prop}

In the sequel, we write $\P^w_\xi$ and $\E^w_\xi$ whenever the solution to (\ref{eq:voter}) is subject to the initial condition $\xi\in S^E$.
The notations $\P^w_\lambda$  and $\E_\lambda^w$, for $\lambda$ being a probability measure on $S^E$, are understood similarly. We drop the superscripts $w$ in these notations if $w=0$ and there is no risk of confusion.

\section{The Radon-Nikodym derivative processes}\label{sec:radon}
In this section, we prove some a-priori bounds for the $(\F_t,\P)$-martingales $(D^w_t)$ defined by (\ref{def:D}). These bounds will play a crucial role in Section~\ref{sec:weak} and Section~\ref{sec:wass} for limit theorems of the game density processes.

Recall that, for each fixed $w\in [0,\overline{w}]$,  $D^w$ is a Dol\'eans-Dade exponential martingale:
\begin{align}
D^w_t=\mathcal  E (L^w)_t\stackrel{\rm def}{=}\exp\left(L^w_t-\frac{1}{2}\langle (L^w)^c,(L^w)^c\rangle_t\right)\prod_{s:s\leq t}(1+\Delta L^w_s)\exp\left(-\Delta L^w_s\right).\label{DM}
\end{align}
Here, the stochastic logarithm $L^w$ of $D^w$ is defined
with respect to the compensated $(\F_t,\P)$-Poisson processes: 
\begin{align}\label{comp}
\widehat{\Lambda}_t(x,y)\equiv \Lambda_t(x,y)-q(x,y)t,\quad x,y\in E,
\end{align}
as the following $(\F_t,\P)$-martingale:
\begin{align}
L_t^w= &\sum_{x,y\in E}\int_0^t\left(\frac{q^w(x,y,\xi_{s-})}{q(x,y)}-1\right)d\widehat{\Lambda}_s(x,y),\label{eq:Mw1}
\end{align}
which has a zero continuous part $(L^w)^c\equiv 0$. In (\ref{eq:Mw1}) and what follows, we use the convention that $0/0=0$.
The equation (\ref{DM}) implies that $D^w$ is the pathwise unique solution to the linear equation    
\begin{align}\label{D:si}
D_t^w=1+\int_0^t D^w_{s-}dL^w_s=1+\sum_{x,y\in E}\int_0^t D^w_{s-}\left(\frac{q^w(x,y,\xi_{s-})}{q(x,y)}-1\right)d\widehat{\Lambda}_s(x,y),
\end{align}
where the last equality follows from (\ref{eq:Mw1}). See \cite[Theorem I.4.61]{JS} for these properties of $D^w$.

In the sequel, $\ms P(U)$ denotes the set of probability measures defined on a Polish space $U$. Also,
recall that $\pi$ denotes the unique stationary distribution of a voting kernel $q$.

\begin{prop}\label{prop:Dbdd}
For every $a\in [1,\infty)$, there is a positive constant $C_{\ref{Dbdd0}}$ depending only on $(\Pi,a)$ such that for all $w\in [0,\overline{w}]$ and $\lambda\in \ms P(S^E)$, 
\begin{align}\label{Dbdd0}
(D_t^w)^a\exp\left(-C_{\ref{Dbdd0}}w^2\pi_{\min}^{-1}\sum_{\ell=1}^4\int_0^t W_\ell(\xi_s)ds\right)\quad\mbox{is an $(\F_t,\P_\lambda)$- supermartingale,}
\end{align}
where $\pi_{\min}=\min_{x\in E}\pi(x)$ and $W_\ell(\xi)$'s are weighted two-point density functions defined by
\begin{align}
\label{def:Pell}
W_\ell(\xi)=& \sum_{x,y\in E}\pi(x)q^{\ell}(x,y)\xi(x)\dxi(y),\quad \ell\geq 1.
\end{align}
In particular, there is a positive constant $C_{\ref{cond:ui}}$ depending only on $(\Pi,a)$ such that
for all $w\in[0, \overline{w}]$, $\lambda\in \ms P(S^E)$ and $(\F_t)$-stopping times $T'$, we have
\begin{align}\label{cond:ui}
\E_\lambda[(D^w_{T'})^a]\leq \E_\lambda\left[\exp\left( C_{\ref{cond:ui}}w^2\pi^{-1}_{\min}\sum_{\ell=1}^4\int_0^{T'} W_\ell(\xi_t)dt \right)\right]^{1/2}.
\end{align} 
\end{prop}
\begin{proof} 
Fix $w\in [0,\overline{w}]$. Note that we have
\begin{align}\label{Dbdd}
\sup_{s\in [0,t]}\E_\xi\big[(D_s^w)^a\big]<\infty,\quad \forall\;a\in [1,\infty),\; t\in (0,\infty),\;\xi\in S^E,
\end{align}
which follows from 
the fact that $q^w(x,y,\xi),q(x,y)$ and $\big|\log \big(q^w(x,y,\xi)/q(x,y)\big)\big|$ are uniformly bounded in $x,y,\xi$ by the choice of the maximal selection strength $\overline{w}$ in (\ref{def:w0}).

To obtain the required supermartingale property in (\ref{Dbdd0}), we work with the stochastic integral equation in (\ref{D:si}) satisfied by $D^w$. By the chain rule for Stieltjes integrals \cite[Proposition~0.4.6]{Revuz_2005} and (\ref{D:si}), we have
\begin{align}
(D_t^w)^a=&1+\sum_{x,y\in E}\int_0^t a(D^w_{s-})^{a-1}\cdot D_{s-}^w\left(\frac{q^w(x,y,\xi_{s-})}{q(x,y)}-1\right)d\widehat{\Lambda}_s(x,y)\notag\\
&+\sum_{s:0< s\leq t}\big((D_s^w)^a-(D_{s-}^w)^a-a(D^w_{s-})^{a-1}\Delta D_s^w\big)\notag\\
\begin{split}\label{Daeq}
=&1+\sum_{x,y\in E}\int_0^t a(D^w_{s-})^{a}\left(\frac{q^w(x,y,\xi_{s-})}{q(x,y)}-1\right)d\widehat{\Lambda}_s(x,y)\\
&+\sum_{x,y\in E}\int_0^t(D_{s-}^w)^a
\left[\left(\frac{q^w(x,y,\xi_{s-})}{q(x,y)}\right)^a-1-a\left(\frac{q^w(x,y,\xi_{s-})}{q(x,y)}-1\right)\right]d \Lambda_s(x,y),
\end{split}
\end{align}
where the last equality follows since $D^w_s/D^w_{s-}=q^w(x,y,\xi_{s-})/q(x,y)$ if $\Delta \Lambda_s(x,y)>0$. The second term in (\ref{Daeq}) is a martingale by (\ref{Dbdd}) and the fact that $q^w(x,y,\xi)/q(x,y)$ are uniformly bounded in $x,y,\xi$ (see (\ref{def:qw})). 

Now we handle the integrands in the last sum in (\ref{Daeq}). First, if $q(x,y)>0$, it follows from the definition (\ref{def:qw}) of $q^w$ that $q^w/q$ satisfies the following series expansion in $w$:
\begin{align}
\frac{q^w(x,y,\xi)}{q(x,y)}=&\frac{1-wB(y,\xi)}{1-wA(x,\xi)}\notag\\
=&1+\sum_{i=1}^\infty w^iA(x,\xi)^{i-1}[A(x,\xi)-B(y,\xi)]\label{taylor}\\
\begin{split}
=1+w[A(x,\xi)-B(y,\xi)]
+w^2R^w(x,y,\xi),\label{eq:qwexp0}
\end{split}
\end{align}
where $A,B,R^w$ are functions defined by
\begin{align}
A(x,\xi)&= 1-\sum_{z\in E}q(x,z)\sum_{z'\in E}q(z,z')\Pi\big(\xi(z),\xi(z')\big),\label{def:A}\\
B(y,\xi)&= 1-\sum_{z\in E}q(y,z)\Pi\big(\xi(y),\xi(z)\big)\label{def:B},\\
R^w(x,y,\xi)&= \frac{A(x,\xi)[A(x,\xi)-B(y,\xi)]}{1-wA(x,\xi)}.\label{def:Rw}
\end{align}
Second, observe that  the following inequality is satisfied:
\begin{align}
\sum_{x,y\in E}\pi(x)q(x,y)|A(x,\xi)-B(y,\xi)|
\leq C_{\ref{Dbdd1}}\sum_{\ell=1}^4W_\ell(\xi),\label{Dbdd1}
\end{align}
where the constant $C_{\ref{Dbdd1}}\in (0,\infty)$ depends only on $(\Pi,a)$ and $W_\ell(\xi)$'s are defined by (\ref{def:Pell}). 
To see (\ref{Dbdd1}), recall the reversibility of $q$ and observe that whenever $A(x,\xi)-B(y,\xi)\neq 0$ 
for $x,y$ such that $q(x,y)>0$, we must have 
$\Pi\big(\xi(z),\xi(z')\big)\neq \Pi\big(\xi(y),\xi(z'')\big)$
for some $z,z',z''$ such that $q(x,z)q(z,z')>0$ and $q(y,z'')>0$. In this case, we have either (1) $1\in \{\xi(z),\xi(z')\}$ and $0\in \{\xi(y),\xi(z'')\}$ or (2) $0\in \{\xi(z),\xi(z')\}  $   
 and $1\in \{\xi(y),\xi(z'')\}$. 
Then using (\ref{taylor}),
we consider the first-order Taylor expansions around $0$ of the two functions
 $w\mapsto  (q^w/q)^a-1$ and $w\mapsto a(q^w/q-1)$ and see that both of them take the same form as follows: 
\[
wa(A-B)+\mathcal O(w^2).
\] 
These Taylor expansions give (\ref{Dbdd1}) since
 the derivative of $w\mapsto q^w/q$ at zero of any order is bounded by $|A-B|$ up to a multiplicative constant depending only on $\Pi$ by (\ref{taylor}).
Third, by (\ref{Dbdd1}), we obtain
 that, for all $\xi\in S^E$, 
\begin{align}
&\sum_{x,y\in E}\left|\left(\frac{q^w(x,y,\xi)}{q(x,y)}\right)^a-1-a\left(\frac{q^w(x,y,\xi)}{q(x,y)}-1\right)\right|q(x,y)\leq C_{\ref{Dbdd2}}w^2\pi_{\min}^{-1}\sum_{\ell=1}^4W_\ell(\xi),\label{Dbdd2}
\end{align}
where the constant $C_{\ref{Dbdd2}}\in (0,\infty)$ depends only on $(\Pi,a)$.

We are ready to prove the required supermartingale property in (\ref{Dbdd0}) with the choice $C_{\ref{Dbdd0}}=C_{\ref{Dbdd2}}$. Write $A_t$ for the continuous process
$\int_0^t C_{\ref{Dbdd0}}w^2\pi_{\min}^{-1}\sum_{\ell=1}^4W_\ell(\xi_s)ds$. Then by integration by parts (cf. \cite[Proposition~0.4.5]{Revuz_2005}) and (\ref{Daeq}), we get 
\begin{align*}
(D_t^w)^ae^{-A_t}
=&1+\int_0^t (D^w_s)^ae^{-A_s}\left(-C_{\ref{Dbdd0}}w^2\pi^{-1}_{\min}\sum_{\ell=1}^4W_\ell(\xi_s)\right)ds\\
&+\int_0^t(D_{s-}^w)^a e^{-A_s}\sum_{x,y\in E}
\left[\left(\frac{q^w(x,y,\xi_{s-})}{q(x,y)}\right)^a-1-a\left(\frac{q^w(x,y,\xi_{s-})}{q(x,y)}-1\right)\right]q(x,y)ds\\
&+\sum_{x,y\in E}\int_0^t a(D^w_{s-})^{a}e^{-A_s}\left(\frac{q^w(x,y,\xi_{s-})}{q(x,y)}-1\right)d\widehat{\Lambda}_s(x,y)\\
&+\sum_{x,y\in E}\int_0^t(D_{s-}^w)^a e^{-A_s}
\left[\left(\frac{q^w(x,y,\xi_{s-})}{q(x,y)}\right)^a-1-a\left(\frac{q^w(x,y,\xi_{s-})}{q(x,y)}-1\right)\right]d \widehat{\Lambda}_s(x,y),
\end{align*}
where the sum of the two Riemann-integral terms is nonpositive by (\ref{Dbdd2}) and the choice that $C_{\ref{Dbdd0}}=C_{\ref{Dbdd2}}$, and the last two sums 
are both finite sums of $(\F_t,\P)$-martingales. The foregoing equality is enough for (\ref{Dbdd0}).

The second assertion of the proposition is a simple application of the first assertion. We use the supermartingale in (\ref{cond:ui}) with $a$ replaced by $2a$ and get from the Cauchy-Schwarz inequality that 
\begin{align*}
\E_\lambda[(D^w_{T'})^a]\leq &\E_\lambda\left[\left((D^w_{T'})^a\exp\left(-\frac{C_{\ref{Dbdd0}}(2a)}{2}w^2\pi_{\min}^{-1}\sum_{\ell=1}^4\int_0^{T'}W_\ell(\xi_s)ds\right)\right)^2\right]^{1/2}\\
&\times \E_\lambda\left[\exp\left(\frac{C_{\ref{Dbdd0}}(2a)}{2}w^2\pi_{\min}^{-1}\sum_{\ell=1}^4\int_0^{T'}W_\ell(\xi_s)ds\right)^2\right]^{1/2}
\end{align*}
The required inequality follows from the foregoing inequality and  the optional stopping theorem \cite[Theorem~II.3.3]{Revuz_2005} (this leads to the choice $C_{\ref{cond:ui}}=C_{\ref{Dbdd0}}(2a)$). The proof is complete.
\end{proof}

In the rest of this section, we turn to the predictable covariation between $D^w$ and the density process 
\begin{align}\label{def:Y}
Y_t\stackrel{\rm def}{=} p_1(\xi_t)
\end{align}
as well as their own predictable quadratic variations, where the function $p_1(\xi)$ is defined by (\ref{def:p1}). Recall that $\widehat{\Lambda}_t(x,y)$ denote the compensated $(\F_t,\P)$-Poisson processes defined by (\ref{comp}) and $\widehat{\Lambda}^\sigma_t(x)$ are similarly defined.
With the stochastic integral equation satisfied by $D^w$ already given in (\ref{D:si}), the other process
$Y$ satisfies the following equation by (\ref{eq:voter}) and the reversibility of $q$: 
\begin{align}
 Y_t
=&\, Y_0
+\int_0^t[\mu(1)(1-Y_s)-\mu(0)Y_s] ds+M_t,\label{eq:Y}
\end{align}
where $M$ is an $(\F_t,\P)$-martingale defined by
\begin{align}
\begin{split}\label{eq:M}
M_t=&\sum_{x,y\in E}\pi(x)\int_0^t [\xi_{s-}(y)-\xi_{s-}(x)]d\widehat{\Lambda}_s(x,y)\\
&\hspace{2cm}+\sum_{x\in E}\pi(x)\int_0^t \widehat{\xi}_{s-}(x)d\widehat{\Lambda}^1_s(x)-\sum_{x\in E}\pi(x)\int_0^t \xi_{s-}(x)d\widehat{\Lambda}^0_s(x).
\end{split}
\end{align}

\begin{lem}\label{lem:qv}
Fix $w\in [0,\overline{w}]$. Then under $\P$, we have  
\begin{align}
\begin{split}
\langle M,M\rangle_t=&\int_0^t\sum_{x,y\in E}\nu(x,y)\big[\dxi_s(x)\xi_s(y)+\xi_s(x)\dxi_s(y)\big]ds\\
&+\int_0^t\sum_{x\in E}\pi(x)^2\big[\,\dxi_s(x)\mu(1)+ \xi_s(x)\mu(0)\big]ds,\label{MM}
\end{split}\\
\label{Dp}
\langle M,D^w\rangle_t=&\;w\int_0^t D^w_s\overline{D}(\xi_s)ds
+w^2\int_0^t D^w_sR^w_1(\xi_s)ds,\\
\begin{split}
\langle D^w,D^w\rangle_t=&\;w^2\int_0^t (D^w_s)^2\sum_{x,y\in E}q(x,y)[A(x,\xi_s)-B(y,\xi_s)]^2ds\\
&+w^3\int_0^t (D^w_s)^2R^w_2(\xi_s)ds,\label{DD}
\end{split}
\end{align}
where
\begin{align}
\nu(x,y)= &\pi(x)^2q(x,y)\1_{x\neq y}=\pi(x)^2q(x,y),\quad x,y\in E,\label{def:nu}
\end{align}
and, for $A,B,R^w$ defined by (\ref{def:A}), (\ref{def:B}) and (\ref{def:Rw}), the functions $\overline{D},R^w_1,R^w_2$ in (\ref{Dp}) and (\ref{DD}) are defined by 
\begin{align}
\overline{D}(\xi)&=\sum_{x,y\in E}\pi(x)q(x,y)[\xi(y)-\xi(x)][A(x,\xi)-B(y,\xi)],\label{def:Dbar1}\\
R^w_1(\xi)&=\sum_{x,y\in E}\pi(x)q(x,y)[\xi(y)-\xi(x)]R^w(x,y,\xi),\label{def:R1w}\\
R^w_2(\xi)&=\sum_{x,y\in E}q(x,y)\big\{2[A(x,\xi)-B(y,\xi)]R^w(x,y,\xi)+wR^w(x,y,\xi)^2\big\}.\notag
\end{align}
\end{lem}
\begin{proof}
Recall that the rates of the driving Poisson processes $\Lambda(x,y)$ and $\Lambda^\sigma(x)$ under $\P$ are given by (\ref{rates}).
Hence,
by (\ref{D:si})  and (\ref{eq:M}), we have
\begin{align}
\begin{split}
\langle M,M\rangle_t=&\int_0^t\sum_{x,y\in E}\pi(x)^2q(x,y)[\xi_s(y)-\xi_s(x)]^2ds\\
&+\int_0^t\sum_{x\in E}\pi(x)^2[\, \dxi_s(x)\mu(1)+ \xi_s(x)\mu(0)]ds,
\end{split}\\
\langle M,D^w\rangle_t
=&\sum_{x,y\in E }\pi(x)q(x,y)
\int_0^tD^w_s[\xi_{s}(y)-\xi_{s}(x)]\left(\frac{q^w(x,y,\xi_{s})}{q(x,y)}-1\right)ds,\label{MD1}\\
\langle D^w,D^w\rangle_t=&\sum_{x,y\in E }q(x,y)\int_0^t (D^w_s)^2\left(\frac{q^w(x,y,\xi_s)}{q(x,y)}-1\right)^2ds.\label{DD1}
\end{align}
  The first equation above gives (\ref{MM}), upon using the notation in (\ref{def:nu}) and the equality \begin{align}\label{xixy}
[\xi(y)-\xi(x)]^2=\dxi(x)\xi(y)+\xi(x)\dxi(y).
\end{align} 
For (\ref{Dp}) and (\ref{DD}), we apply the Taylor expansion (\ref{eq:qwexp0}) of $q^w/q$ in $w$ to (\ref{MD1}) and (\ref{DD1}). 
\end{proof}

 The following lemma shows some moment bounds for $\langle D^w,D^w\rangle$ and $\langle M,D^w\rangle$ under $\P$.

\begin{lem}\label{lem:qvbdd} 
For all $a\in [1,\infty)$, we can find positive constants $C_{\ref{tight_bdd1}}$ and $C_{\ref{tight_bdd2}}$ depending only on $(\Pi,a)$ such that for all $\lambda\in \ms P(S^{E})$,
\begin{align}\label{tight_bdd1}
\begin{split}
&\E_\lambda\big[\langle D^{w},D^{w}\rangle_t^a\big]\leq  C_{\ref{tight_bdd1}}\sum_{\ell=1}^4
\E_{\lambda}\left[\exp\left(C_{\ref{tight_bdd1}}w^2 \pi_{\min}^{-1}
\int_0^t W_\ell(\xi_{s})ds\right)\right]^{1/2}\\
&\hspace{4cm}\times \E_{\lambda}\left[\left(w^2\pi_{\min}^{-1}
\int_0^t W_\ell(\xi_{s})ds\right)^{2a}\right]^{1/2},
\end{split}\\
\begin{split}\label{tight_bdd2}
&\E_\lambda\big[{\rm Var}\big(\langle M,D^{w}\rangle\big)_t^a\big]
\leq C_{\ref{tight_bdd2}}\sum_{\ell=1}^4\E_{\lambda}\left[\exp\left(C_{\ref{tight_bdd2}}w^2\pi^{-1}_{\min}
\int_0^t W_\ell(\xi_{s})ds\right)\right]^{1/2}\\
&\hspace{4cm}\times\E_{\lambda}\left[\left(w
\int_0^t W_\ell(\xi_{s})ds\right)^{2a}\right]^{1/2},
\end{split}
\end{align}
where ${\rm Var}(A)$ denotes the total variation process for $A$.
\end{lem}
\begin{proof}
By (\ref{taylor}), (\ref{Dbdd1}) and (\ref{DD1}), we obtain 
 \begin{align}
\E_\lambda\big[\langle D^w,D^w\rangle_t^a\big]
\leq &C_{\ref{tight_bdd1-1}}\sum_{\ell=1}^4\E_\lambda\left[\left(
\int_0^t \big(D_s^w\big)^2
w^2 \pi_{\min}^{-1}W_\ell(\xi_{s})ds\right)^a\right]\label{tight_bdd1-1}\\
\leq &C_{\ref{tight_bdd1-1}}\sum_{\ell=1}^4
\E_\lambda[(D_t^w)^{4a}]^{1/2}\times
\E_{\lambda}\left[\left(w^2\pi_{\min}^{-1}
\int_0^t W_\ell(\xi_{s})ds\right)^{2a}\right]^{1/2}\label{Dqv}\\
\begin{split}\label{Dqv1}
\leq &C_{\ref{Dqv1}}\sum_{\ell=1}^4
\left(\sum_{\ell'=1}^4\E_\lambda\left[\exp\left(C_{\ref{Dqv1}}w^2\pi^{-1}_{\min}\int_0^t W_{\ell'}(\xi_s)ds\right)\right]^{1/2}\right)\\
&\times \E_{\lambda}\left[\left(w^2\pi_{\min}^{-1}
\int_0^t W_\ell(\xi_{s})ds\right)^{2a}\right]^{1/2}.
\end{split}
\end{align}
Here, the positive constants $C_{\ref{tight_bdd1-1}}$ and $C_{\ref{Dqv1}}$ depend only on $(\Pi,a)$, (\ref{Dqv}) follows from the Cauchy-Schwarz inequality and Doob's strong $L^p$-inequality \cite[Theorem~II.1.7]{Revuz_2005}
since $(D^{w})^{4a}$ is a submartingale,
and (\ref{Dqv1}) follows from (\ref{cond:ui}) and some elementary inequalities. The inequality (\ref{tight_bdd1}) 
is then implied by (\ref{Dqv1}). 

The proof of (\ref{tight_bdd2}) is similar. We use (\ref{taylor}), (\ref{Dbdd1}) and (\ref{MD1}) and get
\begin{align}\label{tight_bdd1-2}
\E_\lambda\big[{\rm Var}\big(\langle M,D^{w}\rangle\big)_t^a\big]\leq C_{\ref{tight_bdd1-2}}\sum_{\ell=1}^4\E_\lambda\left[\left(\int_0^t D^w_swW_\ell(\xi_s)ds\right)^a\right].
\end{align}
This leads to (\ref{tight_bdd2}) upon applying the same arguments as those for (\ref{Dqv}) and (\ref{Dqv1}). The proof is complete.
\end{proof}

Equation~(\ref{MM}) and the inequalities in Lemma~\ref{lem:qvbdd} show that the voter potential functions $\int_0^\cdot W_\ell(\xi_s)ds$ 
play a key role in bounding the covariations considered in Lemma~\ref{lem:qv}.
We will study these functions in Section~\ref{sec:tight}.

\section{Weak convergence of the game density processes}\label{sec:weak}
Our goal in this section is to study the density processes of $1$'s in the evolutionary games. 
Let a sequence of voting kernels $(E_n,q^{(n)})$ and a sequence of mutation measures $\mu_n$ defined on $S$ be given, where $N_n=\#E_n $ increases to infinity. 
To apply the method of equivalence of laws outlined in Section~\ref{intro:1.3}, we consider the following vector semimartingale under $\P^{(n)}$ for each $n\in \Bbb N$:
\begin{align}\label{def:Zn}
Z^{(n)}=\big(Y^{(n)}_t,M^{(n)}_t,D^{(n)}_t\big)= \big(Y_{\gamma_n t},M_{\gamma_n t}, D^{w_n}_{\gamma_nt}\big),
\end{align}
where the constants $\gamma_n$ and $w_n$ will be chosen later on such that $\gamma_n$ tends to infinity and $w_n$ tends to zero, respectively. 
Here in (\ref{def:Zn}), for each $n$, $(Y,M,D)$ under $\P^{(n)}$ consists of the processes considered in Section~\ref{sec:poisson} and Section~\ref{sec:radon}
 with respect to the $(E_n,q^{(n)},\mu_n)$-voter model (recall (\ref{def:D}), (\ref{def:Y}) and (\ref{eq:M})).  Notice that the vector semimartingale $Z^{(n)}$ is adapted to the filtration
\begin{align}\label{def:Fn} 
\F^{(n)}_t= \sigma(\xi_{\gamma_n s};s\leq t),\quad 0\leq t<\infty,
\end{align}
where $(\xi_t)$ is understood to be the $(E_n,q^{(n)},\mu_n)$-voter model.

Similar to the above notations, objects defined with respect to a triplet $(E_n,q^{(n)},\mu_n)$ will carry
either subscripts `$n$' or superscripts `$(n)$' whenever necessary. Those where references to $n$ are not made are defined under general voter models.

The arguments in the sequel will use the Feynman-Kac duality between voter models and coalescing Markov chains, which we discuss briefly here and in more detail in Section~\ref{sec:FK} (see also \cite[Section~6]{CC_16} and \cite{Granovsky_1995}). For a triplet $(E,q,\mu)$, the dual process is a system of coalescing $q$-Markov chains $\{B^x;x\in E\}$ on $E$ so that $B^x$'s move along sites of $E$ as rate-$1$ $q$-Markov chains independently before meeting and together afterwards; particular dual functions are given by
\begin{align}\label{def:H}
H(\xi;x,y)=\big[\xi(x)-\overline{\mu}(1)\big]\big[\dxi(y)-\overline{\mu}(0)\big],\quad x,y\in E,
\end{align}
where 
\begin{align}\label{def:mubar}
\overline{\mu}(\sigma)=\mu(\sigma)/\mu(\1)\quad\mbox{with the convention that }0/0=0
\end{align}
and $\mu(\1)$ is the total mass of $\mu$. 
Then the Feynman-Kac duality between the $(E,q,\mu)$-voter model and the coalescing system $\{B^x\}$ gives the following equation:
\begin{align}
\begin{split}
\E_\xi[H(\xi_t;x,y)]=&\E\left[H(\xi;B^{x}_t,B^y_t)\exp\left(-\mu(\1)\int_0^t |B^{\{x,y\}}_s|ds\right)\right]\\
&-\mu(\1)\overline{\mu}(1)\overline{\mu}(0)\E\left[\int_0^t\1_{\{B^x_s=B^y_s\}}\exp\left(-\mu(\1)\int_0^s |B^{\{x,y\}}_r|dr\right)ds\right],
\end{split}\label{MG}
\end{align}
where $B^{\{x,y\}}=\{B^x,B^y\}$ and $|\{x,y\}|$ is the number of distinct points in $\{x,y\}$ (see (\ref{eq:gh2}) for the generator equation of (\ref{MG})). 
It can be shown that by (\ref{MG}), for all $\xi\in S^E$ and $x,y\in E$,
\begin{align}\label{mombdd}
\begin{split}
\big|\E_\xi[\xi_t(x)\dxi_t(y)]-\E[\xi(B^x_t)\dxi(B^y_t)]\big|
\leq &C_{\ref{mombdd}}\big(1-e^{-\mu(\1)t}\big)\P(M_{x,y}>t)\\
&\hspace{1cm}+C_{\ref{mombdd}}\mu(\1)\int_0^t \P(M_{x,y}>s)ds,
\end{split}
\end{align}
where $C_{\ref{mombdd}}$ is a universal constant and $M_{x,y}$ is the first time that $B^x$ and $B^y$ meet. An alternative proof of 
(\ref{mombdd}) by the pathwise duality between voter models and coalescing Markov chains can be found in \cite[Proposition~3.1]{CC_16}.

\subsection{Main theorem}\label{sec:setup}
Let us state four assumptions for the main theorem, Theorem~\ref{thm:main2}, of Section~\ref{sec:weak} to be stated later on.

\begin{ass}[Uniformity in stationary distributions]\label{ass:stat}
The stationary distributions $\pi^{(n)}$'s of the voting kernels $q^{(n)}$ are comparable to uniform distributions in the sense that they satisfy
\begin{align}
\label{diag}
0<\liminf_{n\to\infty}N_n\pi^{(n)}_{\min}\leq \limsup_{n\to\infty}N_n\pi^{(n)}_{\max}<\infty,
\end{align}
where $\pi^{(n)}_{\max}=\max_{x\in E_n}\pi^{(n)}(x)$ and $\pi^{(n)}_{\min}=\min_{x\in E_n}\pi^{(n)}(x)$. 
\qed 
\end{ass}

\begin{ass}[Weak convergence of voter models]\label{ass:vm}
We can choose a sequence of constants $\gamma_n$ growing to infinity such that
the time-changed density processes $Y^{(n)}$ of $1$'s in the $(E_n,q^{(n)},\mu_n)$-voter models defined in (\ref{def:Zn}) satisfy: 
\begin{align}\label{conv:WF}
\big(Y^{(n)},\P^{(n)}_{\lambda_n}\big)\xrightarrow[n\to\infty]{{(\rm d)}} \big(Y,\P^{(\infty)}\big)
\end{align}
for some $\lambda_n\in \ms P(S^{E_n})$. 
Here,  $\xrightarrow[n\to\infty]{{(\rm d )}}$ denotes convergence in distribution, and under $\P^{(\infty)}$,  $Y$ is a Wright-Fisher diffusion obeying the following equation:
\begin{align}\label{eq:WF}
dY_t=[\mu(1)(1-Y_t)-\mu(0) Y_t]dt+\sqrt{Y_t(1-Y_t)}dB_t,
\end{align}
where $\mu=\big(\mu(1),\mu(0)\big)\in\R^2_+$ is a constant vector  and $B$ is a standard Brownian motion. 

In the case that $\sup_n \gamma_n/N_n=\infty$, we also require that (\ref{conv:WF}) apply 
with respect to the same sequence $\{\gamma_n\}$, 
when 
mutation measures are zero and
the initial laws are given by Bernoulli product measures $\beta_u$ with constant densities $\beta_u\{\xi\in S^{E_n};\xi(x)=1\}\equiv u$ for all 
$u\in (0,1)$.
\qed 
\end{ass}

Assumption~\ref{ass:vm} holds if we impose mild mixing conditions on $(E_n,q^{(n)},\mu_n)$. This is  the content of \cite[Theorem~2.2]{CCC} and a particular consequence of \cite[Theorem~4.1]{CC_16}, which are restated below as Theorem~\ref{thm:WF}. Here and in what follows, $\mathbf g_n$ denotes the difference between $1$ and the second largest eigenvalue of $q^{(n)}$, and
\[
\mathbf t^{(n)}_{\rm mix}=\inf\Big\{t\geq 0;\max_{x\in E}\|e^{tq^{(n)}}(x,\,\cdot\,)-\pi^{(n)}\|_{\rm TV}\leq \frac{1}{2e}\Big\}
\]
stands for the mixing time of the rate-$1$ $(E_n,q^{(n)})$-chains, where $\|\lambda\|_{\rm TV}$ is the total variation norm of a signed measure $\lambda$.

\begin{thm}[\cite{CCC,CC_16}]\label{thm:WF}
Let $(E_n,q^{(n)},\mu_n)$ with $N_n\nearrow\infty$, mutation measures $\mu_n$ defined on $S$, and $\lambda_n\in \ms P(S^{E_n})$ be given such that all of the following three properties are satisfied: 
\begin{enumerate}
\item [\rm (i)] $\displaystyle \lim_{n\to\infty}\sum_{x\in E_n}\pi^{(n)}(x)^2=0$, 
\item [\rm (ii)]$\displaystyle \lim_{n\to\infty}\gamma_n\mu_n=\mu$,
\item [\rm (iii)] the sequence
$\{\displaystyle \lambda_n(p_1(\xi)\in \,\cdot\,)\}$ converges weakly to $\widetilde{\lambda}_\infty$
as probability measures on $[0,1]$, 
\end{enumerate}
and at least one of the following two conditions applies: 
\begin{enumerate}
\item [\rm (iv-1)] $\displaystyle \lim_{n\to\infty}\frac{\mathbf t^{(n)}_{\rm mix}}{\gamma_n}=0$,
\item [\rm (iv-2)] $\displaystyle \lim_{n\to\infty}\frac{\log(e\vee \gamma_n \pi^{(n)}_{\max})}{\mathbf g_n\gamma_n}=0$,
\end{enumerate}
with respect to the constant time scales
\begin{align}\label{gammaforvm}
\gamma_n=\sum_{x,y\in E_n}\pi^{(n)}(x)\pi^{(n)}(y)\E^{(n)}[M_{x,y}].
\end{align}
Then (\ref{conv:WF}) holds. 
\end{thm}

For Theorem~\ref{thm:WF}, (iv-1) and (iv-2) are its major conditions. Condition (iv-1) has the informal interpretation that on the time scale $\gamma_n$, any two independent $(E_n,q^{(n)})$-Markov chains starting at $x\neq y$ reach
stationarity very soon without meeting. A similar interpretation applies to (iv-2) if one
recalls that inverse spectral gaps are interpreted as relaxation times to stationarity \cite[Section~3.4]{AF_MC}. See \cite{Keilson_1979,Aldous_1982} for general results of such notions in the classical theory of Markov chains.
In addition, notice that, in Theorem~\ref{thm:WF}, condition (i) is implied by the fact that there is almost uniformity in stationarity (\ref{diag}). 
Condition (ii) of Theorem~\ref{thm:WF} follows from (\ref{eq:Y}) and (\ref{conv:WF}), which can be seen by solving elementary differential equations.

The next assumption concerns spatial structures defined by voting kernels.

\begin{ass}[Spatial homogeneity]\label{ass:homo}
For fixed $L\in \Bbb N$, we can choose a sequence of constants $\gamma_n$ growing to infinity such that the following $L$th-order {\bf spatial homogeneity condition} holds: for constants $R_0=1$, $R_1=0,R_2\cdots, R_L\in \R_+$, 
\begin{align}
\begin{split}
\lim_{n\to\infty}\gamma_n\nu_n(\1)\pi^{(n)}\big\{x\in E_n;q^{(n),\ell}(x,x)\neq R_\ell\big\}=0,\quad \forall\; 0\leq\ell \leq L,
\label{cond:hom}
\end{split}
\end{align} 
where $\nu_n$ is a measure on $E_n\times E_n$ defined by $\nu_n(x,y)=\pi^{(n)}(x)^2q^{(n)}(x,y)$ as in (\ref{def:nu}) and $\nu_n(\1)$ is the total mass of $\nu_n$. \qed 
\end{ass}

We have $R_1=0$ in Assumption~\ref{ass:homo} since voting kernels are assumed to have zero traces.

Assumption~\ref{ass:homo} corresponds to the local convergence of spatial structures in the sense of \cite{McKay_1981, Benjamini_2001}. For example,  if $\gamma_n=\Theta(N_n)$, that is
\begin{align}\label{gamman_nun}
 C_{\ref{gamman_nun}}^{-1}N_n\leq \gamma_n\leq C_{\ref{gamman_nun}}N_n 
\end{align}
for some constant $C_{\ref{gamman_nun}}\in (0,\infty)$ independent of $n$, and $\pi^{(n)}$'s are comparable to uniform distributions in the sense of (\ref{diag}),
then (\ref{cond:hom}) is equivalent to
\begin{align}\label{cond:hom1}
\lim_{n\to\infty}\pi^{(n)}\{x\in E_n;q^{(n),\ell}(x,x)\neq R_\ell\}=0,\quad \forall\;0\leq \ell\leq L.
\end{align}
See \cite[Section~8]{CCC} and Theorem~\ref{thm:O}
 for examples of (\ref{gamman_nun}) where $\gamma_n$ are given by (\ref{gammaforvm}).

The last assumption specifies the choice of selections strengths.

\begin{ass}[Weak selection]\label{ass:w}
We choose a sequence of selection strengths $w_n\in  [0,\overline{w}]$ satisfying:
\begin{align} 
& w_\infty=
\displaystyle \lim_{n\to\infty}\frac{w_n}{\nu_n(\1)}\in [0,\infty),
\label{selection2}
\end{align}
where $\overline{w}$ is defined by (\ref{def:w0}). \qed 
\end{ass}

Below is the main result of Section~\ref{sec:weak} for the vector semimartingales $Z^{(n)}$ defined in (\ref{def:Zn}). We equip spaces of Polish-space-valued c\`adl\`ag functions with Skorokhod's $J_1$-topology.

\begin{thm}[Main theorem]\label{thm:main2} 
Suppose that 
\begin{enumerate}
\item [\rm (i)] {\rm Assumption~\ref{ass:stat}} holds,
\item [\rm (ii)] {\rm Assumption~\ref{ass:vm}} holds, and
\item [\rm (iii)] a sequence of selection strengths $w_n$ satisfying {\rm Assumption~\ref{ass:w}} is given.
\end{enumerate}
Then we have the following results.

\begin{enumerate}
\item [\rm  1$^\circ$)] The sequence of laws of
  $Z^{(n)}=(Y^{(n)},M^{(n)},D^{(n)})$ under $\P^{(n)}_{\lambda_n}$ is $C$-tight. Any subsequential limit, say along $(Y^{(n_k)},M^{(n_k)},D^{(n_k)})$ under $\P^{(n_k)}_{\lambda_{n_k}}$, is the law of a continuous vector semimartingale $(Y,M,D)$ under $\P^{(\infty)}$ such that the last two components define a vector martingale with respect to the filtration generated by $(Y,M,D)$. In addition, the sequence of laws of $(Y^{(n_k)},M^{(n_k)})$ under $\P^{(n_k),w_{n_k}}_{\lambda_{n_k}}$ converges to the law of $(Y,M)$ under $D\cdot \P^{(\infty)}$. \\

\item [\rm 2$^\circ$)] If, moreover, 
$q^{(n)}$ are symmetric kernels, 
{\rm Assumption~\ref{ass:homo}} with $L=2$ with respect to the same sequence $\{\gamma_n\}$ chosen in {\rm (ii)} applies, 
and the payoff matrix $\Pi$ is given by (\ref{eq:payoff0}), then any subsequential limit $(Y,M,D)$ under $\P^{(\infty)}$ satisfies the covariation equations:
\begin{align}\label{YD:covar}
\langle Y,D\rangle_t=\langle M,D\rangle_t=w_\infty K_1(b,c)\int_0^t D_sY_1(1-Y_s)ds\quad\mbox{ under }\P^{(\infty)},
\end{align}
where $w_\infty$ is defined by (\ref{selection2}) and, with respect to $R_\ell$ chosen in (\ref{cond:hom}), $K_1(b,c)$ is defined by 
\begin{align}
K_1(b,c)=\frac{b(R_2+R_1)-c(R_1+R_0)}{2}.\label{def:K1}
\end{align}

\item [\rm 3$^\circ$)] If the assumptions of {\rm 2$^\circ$)} apply and the stronger {\rm Assumption~\ref{ass:homo}} with $L=3$ is valid as well, then the sequence of laws of  $(Y^{(n)},M^{(n)},D^{(n)})$ under $\P^{(n)}_{\lambda_n}$ converges weakly towards the law of a vector semimartingale $(Y,M,D)$ under $\P^{(\infty)}$. The triplet $(Y,M,D)$ under $\P^{(\infty)}$
can be characterized as a solution to the following system of stochastic differential equations:
\begin{align}\label{SDE}
\left\{
\begin{array}{ll}
dY_t=[\mu(1)(1-Y_t) -\mu(0)Y_t]dt+\sqrt{Y_t(1-Y_t)}dW^{1}_t,\\
\vspace{-.2cm}\\
dM_t=\sqrt{Y_t(1-Y_t)}dW^{1}_t,\\
\vspace{-.2cm}\\
dD_t= w_\infty D_t\sqrt{Y_t(1-Y_t)}\big[K_1(b,c)dW^{1}_t+\sqrt{K_2(b,c)- K_1(b,c)^2}dW^{2}_t\big].
\end{array}
\right.
\end{align}
Here, $(W^{1},W^{2})$ is a two-dimensional standard Brownian motion, $K_1(b,c)$ is given by (\ref{def:K1}), and $K_2(b,c)$ is defined by
\begin{align}
K_2(b,c)=&\frac{b^2(R_3+R_2)-2bc(R_2+R_1)+c^2(R_1+R_0)}{2}.\label{def:K2}
\end{align}
\end{enumerate}
\end{thm}

The proof of Theorem~\ref{thm:main2} is given in Section~\ref{sec:tight} and Section~\ref{sec:id}. See Proposition~\ref{prop:tight} for Theorem~\ref{thm:main2} 1$^\circ$) and Proposition~\ref{prop:id} for Theorem~\ref{thm:main2} 2$^\circ$) and 3$^\circ$). These propositions give more detailed results.

The following theorem is a straightforward application of Theorem~\ref{thm:main2} 1$^\circ$) and 2$^\circ$), Girsanov's theorem \cite[Theorem VIII.1.7]{Revuz_2005}, and the Yamada-Watanabe theorem for pathwise uniqueness in stochastic differential equations~\cite[Theorem IX.3.5]{Revuz_2005}.

\begin{thm}[Diffusions for evolutionary games with death-birth updating]\label{thm:main2-1}
Let the assumptions of {\rm Theorem~\ref{thm:main2} {\rm 2$^\circ$)}} be in force, and recall the constants $w_\infty$ and $K_1(b,c)$ defined by (\ref{selection2}) and (\ref{def:K1}), respectively. Then we have the following. 
\begin{enumerate}
\item [\rm 1$^\circ$)] The sequence of laws of $(Y^{(n)},\P^{(n),w_n}_{\lambda_n})$ converges weakly to the law of a Wright-Fisher diffusion $Y$ with initial law  $\ms L (Y_0)=\widetilde{\lambda}_\infty$ under $\P_{\widetilde{\lambda}_\infty}^{(\infty),w_\infty}$, where $\P_{\widetilde{\lambda}_\infty}^{(\infty),w_\infty}$ can be defined as $D\cdot \P^{(\infty)}$ for any subsequential weak limit $D$ of $D^{(n)}$. \smallskip

\item [\rm 2$^\circ$)] The Wright-Fisher diffusion  $Y$ in 
1$^\circ$) obeys the following equation: 
\begin{align}\label{SDE:WF}
dY_t=[w_\infty K_1(b,c)Y_t(1-Y_t)+\mu(1)(1-Y_t) -\mu(0) Y_t]dt+\sqrt{Y_t(1-Y_t)}dW_t
\end{align}
with respect to a standard Brownian motion $W$.
\end{enumerate}
\end{thm}

\subsection{Example: evolutionary games on large random regular graphs}\label{sec:app}
We fix $k\geq 3$ and consider a sequence of random $k$-regular graphs $G_n$ on $N_n$ vertices with $N_n\nearrow\infty$. (For definiteness, we assume that $G_n$'s are given by the uniform models.) One basic property of $\{G_n\}$ states that the second eigenvalues of the adjacency matrices of $G_n$ are bounded away from the largest ones, namely $k$, in the limit of infinite volume (see \cite{Friedman_2008, Bordenave_2015}). Hence, $G_n$'s are connected for all large $n$.

The following proposition can be used to verify Assumption~\ref{ass:homo} with $L= 3$, which is one of the conditions for Theorem~\ref{thm:main2} $3^\circ)$.

\begin{prop}\label{prop:gap}
Let $\mathbf g(G)$ denote the spectral gap of a random walk on a finite connected unweighted graph $G$. Recall that $M_{x,y}$ denotes the first meeting time of two independent rate-$1$ random walks on $G$ starting from $x$ and $ y$. 
Then
\begin{align}\label{mbdd}
\max_{x,y\in G}\E[M_{x,y}]\leq \max_{y\in G}\frac{2\sum_{x:x\neq y}\deg (x)}{\mathbf g(G)\deg (y)}.
\end{align}
\end{prop}
\begin{proof}
Let $H_{x,y}$ denote the first hitting time of $y$ by a rate-$1$ random walk on $G$ starting from $x$.
By \cite[Proposition~14.5, Lemma 3.15, Lemma~3.17]{AF_MC}, we have
\begin{align*}
\max_{x,y\in G}\E[M_{x,y}]\leq \max_{x,y\in G}\E[H_{x,y}]\leq \max_{y\in G}2\sum_{x\in G}\pi(x)\E[H_{x,y}]\leq \max_{y\in G}\frac{2\big(1-\pi(y)\big)}{\mathbf g(G)\pi(y)},
\end{align*}
which is enough for the required inequality in  (\ref{mbdd}) since $\pi(y)\equiv \deg(y)/\sum_x\deg(x)$.
\end{proof}

The following theorem obtains diffusion approximations of the game density processes on large random regular graphs when payoff matrices are given by (\ref{eq:payoff0}).

\begin{thm}
\label{thm:O}
For fixed $k\geq 3$, consider a sequence of random $k$-regular graphs $G_n$ with $G_n$ carrying $N_n$ vertices and $N_n\nearrow\infty$. 
Set $\gamma_n=N_n^{-2}\sum_{x,y\in E_n}\E^{(n)}[M_{x,y}]$ and then choose $\{w_n\}$ according to  Assumption~\ref{ass:w} and mutation measures $\mu_n$ on $S$ which satisfy {\rm Theorem~\ref{thm:WF}} {\rm (ii)}.
Finally, assume that {\rm Theorem~\ref{thm:WF}} {\rm (iii)} holds for some $\lambda_n\in \ms P(S^{E_n})$. 
Then the conclusion of {\rm Theorem~\ref{thm:main2}} {\rm 3$^\circ$)} holds, and the constants $K_1(b,c)$ and $K_2(b,c)$ are now given by
\begin{align}K_1(b,c)=\frac{bk^{-1}-c}{2}\quad \mbox{and}\quad  K_2(b,c)=\frac{b^2k^{-1}-2bck^{-1}+c^2}{2}.\label{specialK}
\end{align}
In particular, the limiting Wright-Fisher diffusion $Y$ under $\P^{(\infty),w_\infty}_{\widetilde{\lambda}_\infty}$ in Theorem~\ref{thm:main2-1} simplifies to the following stochastic differential equation:
\begin{align}\label{SDE:O}
dY_t=\left(\frac{w_\infty (b-ck)}{2k}  Y_t(1-Y_t)+\mu(1)(1-Y_t)-\mu(0)Y_t\right) dt+\sqrt{Y_t(1-Y_t)}dB_t,
\end{align}
where $B$ is a standard Brownian motion. 
\end{thm}

\begin{rmk}\label{rmk:O}
(1) To convert the diffusion process defined by the coefficients in 
(\ref{alphabeta}) to the diffusion process defined by \cite[Eq. (18) in SI]{Ohtsuki_2006}, the reader may notice that, on a $k$-regular graph with $N$ vertices, the generator of the evolutionary game considered in \cite[SI]{Ohtsuki_2006} is given by $N^{-1}\mathsf L^{w,0}$, where $\mathsf L^{w,0}$ is defined by (\ref{def:Lw});
compare \cite[Eq. (11) and (12) in SI]{Ohtsuki_2006} to (\ref{def:Lw}), (\ref{def:cw}) and (\ref{eq:qwexp0}). Hence, speeding up its time scale by the constant factor $N$ recovers the evolutionary game considered in this paper. Also, we have an additional multiplicative factor of $k^{-1}$ in the drift coefficient in (\ref{coefficients}) 
 since total payoffs of individuals are defined by the weighted averages in (\ref{weightedpay}), where $q(x,y)$ are equal to $k^{-1}$ for all pairs of vertices $x,y$ adjacent to each other.
\smallskip

\noindent (2) Assume that $\mu_n=0$ for all $n$. Given a payoff matrix $\Pi$ taking the form (\ref{eq:payoff0}), the constants $\alpha,\beta$ defined by (\ref{alphabeta}) simplify to $\alpha=0$ and $\beta=k(b-kc)$. If we speed up time by applying the constant time change $[N(k-1)]/[2(k-2)]$ to (\ref{coefficients}), the diffusion process predicted in \cite[SI]{Ohtsuki_2006} has a Wright-Fisher noise coefficient as in (\ref{SDE:O}). In addition, by setting selection strength $w$  in (\ref{coefficients}) to be $w_\infty/N$, we recover the drift term in (\ref{SDE:O}).  
\qed 
\end{rmk}

\begin{que}\label{que:O}
Is the prediction in \cite[SI]{Ohtsuki_2006} precise to the degree that 
\[
\gamma_n= N_n^{-2}\sum_{x,y\in E_n}\E^{(n)}[M_{x,y}]\sim \frac{N_n(k-1)}{2(k-2)}, \quad\mbox{as }n\to\infty?
\]
\qed 
\end{que}

\paragraph{\bf Proof of Theorem~\ref{thm:O}}
Note that $\gamma_n=\Theta(N_n)$ since, for example, \cite[(3.21)]{CCC} shows
\begin{align}\label{gamman:lbd}
\gamma_n\geq N_n\left(\frac{N_n-1}{2N_n}\right)^2
\end{align}
and Proposition~\ref{prop:gap} applies by the fact that the spectral gaps $\mathbf g(G_n)$ are bounded away from zero. We have used the aforementioned property of random regular graphs proven in \cite{Friedman_2008, Bordenave_2015}.

To see that the conclusion of Theorem~\ref{thm:main2} 3$^\circ$) holds, it is enough to verify Assumption~\ref{ass:vm} and Assumption~\ref{ass:homo} with $L=3$ since $q^{(n)}(x,y)\equiv 1/k$ for $x\sim y$, $\pi^{(n)}(x)\equiv N_n^{-1}$, and we have chosen $\{w_n\}$ according to Assumption~\ref{ass:w}. For Assumption~\ref{ass:vm}, 
Theorem~\ref{thm:WF} holds with the present choice of $\gamma_n$. Indeed, 
(i) of Theorem~\ref{thm:WF} obviously holds and its (ii)--(iii) are valid by the choice of $\gamma_n$, $\mu_n$, and $\lambda_n\in \ms P(S^{E_n})$. We also know that 
$\mathbf g(G_n)$ are bounded away from zero, so that condition (iv-2) of Theorem~\ref{thm:WF} holds. To satisfy Assumption~\ref{ass:homo} with $L=3$, notice that  $\nu_n(\1)=1/N_n$ and we have seen that $\gamma_n=\Theta(N_n)$. Then it is enough to check (\ref{cond:hom1}). But this condition follows from the well-known locally tree-like property of random regular graphs (cf. \cite{McKay_1981}), which yields the following exact values of $R_1,R_2,R_3$ in particular:
\begin{align}\label{R:prob}
R_1=R_3=0\quad\mbox{ and }\quad R_2=k^{-1} .
\end{align} 
Hence, the conclusion of Theorem~\ref{thm:main2} 3$^\circ$) holds. The constants $K_1(b,c)$ and $K_2(b,c)$ now take the forms in (\ref{specialK}), and the diffusion process in (\ref{SDE:WF}) simplifies to the one in (\ref{SDE:O}). 
\qed \\

In Section~\ref{sec:wass}, we will continue this discussion in the context where mutations are absent and prove diffusion approximations of the game absorbing probabilities. See Corollary~\ref{cor:O} for the precise statement.

\subsection{Proof of the main theorem: tightness}\label{sec:tight}
In this section, we prove tightness of the sequence of laws of the vector semimartingales $Z^{(n)}$ defined in (\ref{def:Zn}) and related tightness properties. Before that, we handle predictable covariations between $M^{(n)}$ and $D^{(n)}$ and their own predictable quadratic variations
by Lemma~\ref{lem:qv} and Lemma~\ref{lem:qvbdd}.

To simplify notation, we introduce  two discrete-time $(E,q)$-Markov chains $(X_\ell)$ and $(Y_\ell)$ which satisfy the following three properties: (1) $X_0=Y_0$; (2) they are
independent of the system of $(E,q)$-coalescing chains $\{B^x;x\in E\}$ (defined at the beginning of Section~\ref{sec:weak}); (3)
they are independent if conditioned on $X_0$. Notice that we can write the two-point density functions $W_\ell$ defined by (\ref{def:Pell}) as $W_\ell(\xi)=\E_\pi[\xi(X_0)\dxi(X_\ell)]$, where $X_0$ starts from stationarity under $\E_\pi$.

\begin{prop}\label{prop:expbdd}
Suppose that {\rm Assumption~\ref{ass:stat}} and {\rm Assumption~\ref{ass:vm}} are in force.\\

\noindent {\rm 1$^\circ$)} For all $a\in (0,\infty)$ and $\ell\in \Bbb N$, it holds that
\begin{align}\label{exp:Well}
&\sup_{n\in \Bbb N}\sup_{\lambda\in \ms P(S^{E_n})}\E^{(n)}_\lambda\left[\exp\left\{a\gamma_n\nu_n(\1)\int_0^t W_\ell(\xi_{\gamma_n s})ds\right\}\right]<\infty,\quad \forall\;t\in (0,\infty),\\
&\lim_{\theta\searrow  0+} \sup_{n\in \Bbb N}\sup_{\lambda\in \ms P(S^{E_n})}\E^{(n)}_\lambda\left[\exp\left\{a\gamma_n\nu_n(\1)\int_0^\theta W_\ell(\xi_{\gamma_n s})ds\right\}\label{exp:Well-vanish}
\right]=1.
\end{align}

\noindent {\rm 2$^\circ$)} If $\mu_n= 0$ for all $n$, then for every $\ell\in \Bbb N$ we can find $a\in (0,\infty)$ small enough such that the inequality in (\ref{exp:Well}) with $t=\infty$ holds.
\end{prop}
\begin{proof}
1$^\circ$) We proceed with the following steps to prove (\ref{exp:Well}) and (\ref{exp:Well-vanish}), which start with three claims. \\

\noindent \emph{Step 1.} We claim that
\begin{align}\label{numom}
\forall\;\ell\geq 1,\quad \sup_{n\in \Bbb N}\nu_n(\1)\E^{(n)}[M_{X_0,X_\ell}]\leq \ell \left(\sup_{n\in \Bbb N}\nu_n(\1)\E^{(n)}[M_{X_0,X_1}]\right)<\infty.
\end{align}
To see (\ref{numom}) for $\ell=1$,
recall that $\pi^{(n)}$'s are comparable to uniform distributions
 by Assumption~\ref{ass:stat} and we have \begin{align}
\sum_{x,y\in E_n}\pi^{(n)}(x)^2q^{(n)}(x,y)\E^{(n)}_\pi[M_{x,y}]
= \frac{1-\sum_{x\in E_n}\pi^{(n)}(x)^2}{2}
\label{ell=1-0}
\end{align}
(cf. \cite[(3.17)]{CCC}). These two facts imply that
\begin{align}
\sup_{n\in \Bbb N}\nu_n(\1)\E^{(n)}_\pi[M_{X_0,X_1}]\leq & \sup_{n\in \Bbb N}\left(\frac{\pi^{(n)}_{\max}}{\pi^{(n)}_{\min}}\right)\Bigg( \sum_{x,y\in E_n}\pi^{(n)}(x)^2q^{(n)}(x,y)\E^{(n)}_\pi[M_{x,y}]\Bigg)<\infty\label{ell=1}
\end{align}
and so the inequality in (\ref{numom}) with $\ell=1$ follows.

To obtain (\ref{numom}) for $\ell\geq 2$, first notice that the strong Markov property of $(B^x,B^y)$ at its first epoch time gives, for all $x\neq y$ and $F:E\times E\to \R$ with $F(u,u)\equiv 0$,
\begin{align}
\begin{split}\label{F_recur}
&\E^{(n)}[F(B^x_t,B^y_t)]=F(x,y)e^{-2t}\\
&\hspace{1cm}+\int_0^t e^{-2(t-s)}\sum_{z\in E_n}\left(q^{(n)}(x,z)\E^{(n)}[F(B^z_s,B^y_s)]+q^{(n)}(y,z)\E^{(n)}[F(B^x_s,B^z_s)]\right)ds.
\end{split}
\end{align}
In particular, if we take $F(u,v)=\1_{\{u\neq v\}}$ so that
$\E^{(n)}[F(B^x_t,B^y_t)]=\P^{(n)}(M_{x,y}>t)$, then (\ref{F_recur}) implies that
\begin{align}\int_0^t\P^{(n)}_\pi(M_{X_0,X_\ell-1}>s)ds
=&\Big(\frac{1-e^{-2t}}{2}\Big)\P^{(n)}_\pi(X_0\neq X_{\ell-1})\notag\\
&+\int_0^t\left(\frac{1-e^{-2(t-s)}}{2}\right)\P^{(n)}_\pi(M_{Y_1,X_{\ell-1}}>s,X_0\neq  X_{\ell-1})ds\notag\\
&+\int_0^t\left(\frac{1-e^{-2(t-s)}}{2}\right)\P^{(n)}_\pi(M_{X_0,X_{\ell}}>s,X_0\neq  X_{\ell-1})ds\notag\\
\begin{split}\label{eq:meetingtime}
=&
\Big(\frac{1-e^{-2t}}{2}\Big)\P^{(n)}_\pi(X_0\neq X_{\ell-1})
+\int_0^t(1-e^{-2(t-s)})\P^{(n)}_\pi(M_{X_0,X_{\ell}}>s)ds\\
&-\int_0^t\big(1-e^{-2(t-s)}\big)\P^{(n)}_\pi(M_{X_0,X_{\ell}}>s,X_0= X_{\ell-1})ds
\end{split}
\end{align}
by the reversibility of $q^{(n)}$.
Passing $t$ to infinity for both sides of (\ref{eq:meetingtime}) and using the integrability of each meeting time $M_{x,y}$, we deduce the following inequality:
\[
\E^{(n)}_{\pi}[M_{X_0,X_{\ell}}]\leq \E^{(n)}_{\pi}[M_{X_0,X_{\ell-1}}]+\E^{(n)}_{\pi}[M_{X_0,X_{1}}], 
\]
which is enough for (\ref{numom}) for all $\ell\geq 2$ by (\ref{ell=1}) and iteration.\\

\noindent \emph{Step 2.} We claim the following uniform continuity: 
\begin{align}\label{Wellvep1}
\forall\;\ell\geq 1\;
\forall\;\vep\in (0,1)\;\exists\;\delta\in (0,1),\;\;\sup_{n\in \Bbb N}\gamma_n\nu_n(\1)\int_0^{\delta}\P^{(n)}(M_{X_0,X_\ell}>\gamma_n s)ds\leq \vep.
\end{align}
It suffices to consider the case $\sup_n \gamma_n/N_n=\infty$ thanks to the fact that $\nu_n(\1)=\Theta(N_n^{-1})$ by Assumption~\ref{ass:stat}. Now
we use the part in Assumption~\ref{ass:vm} stating that (\ref{conv:WF}) holds for all initial laws as Bernoulli product measures with constant densities in the absence of mutation. Then it follows from  \cite[Theorem~4.1]{CCC} that 
\begin{align}\label{meetingtime}
\limsup_{n\to\infty}\gamma_n\nu_n(\1)\int_0^t \P^{(n)}(M_{X_0,X_1}>\gamma_n s)ds\leq C_{\ref{meetingtime}}(1-e^{-t}),\quad t\geq 0,
\end{align}
where $C_{\ref{meetingtime}}$ depends only on $\limsup\pi^{(n)}_{\max}/\pi^{(n)}_{\min}$ (this limit supremum is finite by Assumption~\ref{ass:stat}). The uniform continuity in (\ref{Wellvep1}) for $\ell=1$ then follows from (\ref{meetingtime}). The proof for general $\ell\geq 2$ can be 
obtained by iterating (\ref{eq:meetingtime}) and using (\ref{meetingtime}) 
since 
\[
\nu_n(\1)\gamma_n\int_0^{t}e^{-\gamma_n (t-s)}ds\leq \nu_n(\1)\quad\mbox{for every }t\geq 0.
\]

\noindent \emph{Step 3.} We claim the following uniform continuity similar to the one in (\ref{Wellvep1}): 
\begin{align}\label{Wellvep}
\forall\;\ell\geq 1\;\forall\;\vep\in (0,1)\;\exists\;\delta \in (0,1),\;\;\sup_{n\in \Bbb N}\sup_{\lambda\in \ms P(S^{E_n})}\E_\lambda^{(n)}\left[\gamma_n\nu_n(\1)\int_0^{\delta}W_\ell(\xi_{\gamma_ns})ds\right]\leq \vep.
\end{align}
We use the following consequence of (\ref{mombdd}):
\begin{align}\label{int:bdd}
&\left|\int_0^t\E^{(n)}_\xi[\xi_{ s}(x)\dxi_{s}(y)]ds-\int_0^{t}\E^{(n)}[\xi(B^x_s)\dxi(B^y_s)]ds\right|\notag\\
&\hspace{2cm}\leq  C_{\ref{mombdd}}\int_0^{t}(1-e^{-\mu_n(\1)s})\P^{(n)}(M_{x,y}>s)ds+C_{\ref{mombdd}}\mu_n(\1)\E^{(n)}[M_{x,y}]\cdot t.
\end{align}
By (\ref{int:bdd}), we see that, for all $\lambda\in \ms P(S^{E_n})$,
\begin{align}
\begin{split}
&\E_\lambda^{(n)}\left[\gamma_n\nu_n(\1)\int_0^t W_\ell(\xi_{\gamma_n s})ds\right]\\
&\leq (1+C_{\ref{mombdd}})\gamma_n\nu_n(\1)\int_0^t\P^{(n)}(M_{X_0,X_\ell}>\gamma_n s)ds+C_{\ref{mombdd}}\gamma_{n}\mu_{n}(\1)\cdot \nu_n(\1)\E^{(n)}[M_{X_0,X_\ell}]t.\label{bdd:Well-0}
\end{split}
\end{align}
By (\ref{bdd:Well-0})
and the inequality $\sup_n \gamma_n\mu_n(\1)<\infty$ (implied by Assumption~\ref{ass:vm}), the first two claims in (\ref{numom}) and (\ref{Wellvep1}) are enough for (\ref{Wellvep}).\\

\noindent \emph{Step 4.}
Observe that for any $m\geq 1$,
the Markov property of voter models implies that
\begin{align}
\begin{split}\label{Markovit0}
&\sup_{\lambda\in \ms P(S^{E_n})}\E^{(n)}_\lambda\left[\exp\left\{a\gamma_n\nu_n(\1)\int_0^t W_\ell(\xi_{\gamma_n s})ds\right\}\right]\\
&\hspace{4cm}\leq \left(\sup_{\lambda\in \ms P(S^{E_n})}\E^{(n)}_\lambda\left[\exp\left\{a\gamma_n\nu_n(\1)\int_0^{t/m} W_\ell(\xi_{\gamma_n s})ds\right\}\right]\right)^m
\end{split}
\end{align}
and
\begin{align}
&\sup_{\lambda\in \ms P(S^{E_n})}\E_\lambda^{(n)}\left[\left(\gamma_n\nu_n(\1)\int_0^{t}W_\ell(\xi_{\gamma_ns})ds\right)^m\right]\notag\\
=&m!\cdot \sup_{\lambda\in \ms P(S^{E_n})}\E_\lambda^{(n)}\left[\big(\gamma_n\nu_n(\1)\big)^m\int_0^{t} ds_1\int_{s_1}^{t} ds_2\cdots \int_{s_{m-1}}^{ t} ds_m\prod_{i=1}^mW_\ell(\xi_{\gamma_n s_i})\right]\notag\\
\leq &m!\cdot \left(\sup_{\lambda\in \ms P(S^{E_n})}\E_\lambda^{(n)}\left[\gamma_n\nu_n(\1)\int_0^{t}W_\ell(\xi_{\gamma_ns})ds\right]\right)^m.
\label{Markovit}
\end{align}
For the proof of (\ref{exp:Well}) with fixed $t\in (0,\infty)$ and $\ell\geq 1$, 
we choose $\delta$ according to the uniform continuity in (\ref{Wellvep}) with $\vep=1/(2a)$ and then $m$ large such that $t/m\leq \delta$. By (\ref{Markovit0}) for the first inequality below and (\ref{Markovit}) for the second, we have: 
\begin{align*}
&\sup_{n\in \Bbb N}\sup_{\lambda\in \ms P(S^{E_n})}\E^{(n)}_\lambda\left[\exp\left\{a\gamma_n\nu_n(\1)\int_0^t W_\ell(\xi_{\gamma_n s})ds\right\}\right]\\
\leq &\sup_{n\in \Bbb N}\left(\sum_{m'=0}^\infty \frac{a^{m'}}{(m')!}\sup_{\lambda\in \ms P(S^{E_n})}\E_\lambda^{(n)}\left[\left(\gamma_n\nu_n(\1)\int_0^{t/m}W_\ell(\xi_{\gamma_ns})ds\right)^{m'}\right] \right)^m\\
\leq &\left(\sup_{n\in \Bbb N}\sum_{m'=0}^\infty a^{m'}\left(\sup_{\lambda\in \ms P(S^{E_n})}\E_\lambda^{(n)}\left[\gamma_n\nu_n(\1)\int_0^{t/m }W_\ell(\xi_{\gamma_ns})ds\right]\right)^{m'}\right)^m
\leq \left(\sum_{m'=0}^\infty \frac{1}{2^{m'}}\right)^m<\infty. 
\end{align*}
The proof of (\ref{exp:Well-vanish}) follows similarly if we argue as above with $t$ replaced by $\theta$ and $m$ set to be $1$. 
We have proved {\rm 1$^\circ$)}. \\

\noindent {\rm 2$^\circ$)} The proof of 2$^\circ$) follows almost the same line as the proof of (\ref{exp:Well-vanish}) except that we do not need to handle the second term on the right-hand side of (\ref{bdd:Well-0}), which is due to mutation. In more detail, now we consider
\begin{align*}
&\sup_{n\in \Bbb N}\sup_{\lambda\in \ms P(S^{E_n})}\E^{(n)}_\lambda\left[\exp\left\{a\gamma_n\nu_n(\1)\int_0^\infty  W_\ell(\xi_{\gamma_n s})ds\right\}\right]\\
\leq &\sup_{n\in \Bbb N}\sum_{m'=0}^\infty a^{m'}\left(\sup_{\lambda\in \ms P(S^{E_n})}\E_\lambda^{(n)}\left[\gamma_n\nu_n(\1)\int_0^{\infty }W_\ell(\xi_{\gamma_ns})ds\right]\right)^{m'}\\
\leq &\sum_{m'=0}^\infty a^{m'}\left((1+C_{\ref{mombdd}})\ell \sup_{n\in \Bbb N}\nu_n(\1)\E^{(n)}[M_{X_0,X_1}]\right)^{m'},
\end{align*}
where the last inequality follows from (\ref{numom}) and
(\ref{bdd:Well-0}), with $t$ sent to infinity in (\ref{bdd:Well-0}). 
By (\ref{numom}), we can choose $a>0$ small enough such that the last infinite series is finite. This proves 2$^\circ)$.
\end{proof}

For any vector martingale $A$, we write $\langle A,A\rangle$ for the matrix of predictable covariations between components of $A$. If $A$ is a vector semimartingale, then $[A,A]$ denotes the matrix of covariations between its components.  
The following proposition proves Theorem~\ref{thm:main2} 1$^\circ$).

\begin{prop}\label{prop:tight}
If conditions {\rm (i)--(iii)} of  {\rm Theorem~\ref{thm:main2}} are in force, then the following holds.\\

\noindent {\rm 1$^\circ$)} The sequence of laws of
$\big(Z^{(n)},\big\langle (M^{(n)},D^{(n)}),(M^{(n)},D^{(n)})\big\rangle\big)$ under $\P^{(n)}_{\lambda_n}$ is $C$-tight. \\

\noindent {\rm 2$^\circ$)} Suppose that, by choosing a subsequence if necessary, the sequence of laws of $Z^{(n)}$ under $\P^{(n)}_{\lambda_n}$ converges weakly to the law of $Z=(Y,M,D)$ under $\P^{(\infty)}$,  then $Z$ is a continuous vector semimartingale, $M$ is the martingale part of $Y$, $(M,D)$ is a vector martingale with respect to the filtration generated by $(Y,M,D)$, and we have the following convergence:
\begin{align}\label{conv:MD}
\begin{split}
&\Big(Z^{(n)},\big \langle (M^{(n)},D^{(n)}),(M^{(n)},D^{(n)})\big\rangle,\big [ (M^{(n)},D^{(n)}),(M^{(n)},D^{(n)})\big]\Big)\\
&\hspace{4cm}\xrightarrow[n\to\infty]{\rm (d)}\Big(Z,\big [ (M,D),(M,D)\big],\big [ (M,D),(M,D)\big]\Big).
\end{split}
\end{align}

\noindent {\rm 3$^\circ$)} In the context of {\rm 2$^\circ$)}, the sequence of laws of $(Y^{(n)},M^{(n)})$ under $\P^{(n),w_{n}}_{\lambda_{n}}$ converges weakly to the law of $(Y,M)$ under $D\cdot \P^{(\infty)}$.
\end{prop}
\begin{proof}
1$^\circ$) By \cite[Proposition 3.2.4]{EK:MP}, it is enough to prove that all the sequences of laws of components of the multi-dimensional processes under consideration are $C$-tight.\\ 

\noindent \emph{$C$-tightness of the sequence $\big\{\ms L\big(Y^{(n)}\big)\big\}$.} This follows readily from (\ref{conv:WF}) in  Assumption~\ref{ass:vm}. \\

\noindent \emph{$C$-tightness of the sequence $\big\{\ms L\big(M^{(n)}\big)\big\}$.} Recall that Assumption~\ref{ass:vm} implies that $\gamma_n\mu_n\to \mu$. Then it follows from the decomposition (\ref{eq:Y}) of $Y^{(n)}$ and (\ref{conv:WF}) in  Assumption~\ref{ass:vm} that $M^{(n)}$ are martingales uniformly bounded on compacts and converge in distribution to $M$. In particular, the required $C$-tightness follows.\\

\noindent \emph{$C$-tightness of the sequence $\big\{\ms L\big(\langle M^{(n)},M^{(n)}\rangle\big)\big\}$.}
By \cite[Proposition~III.3.26]{JS}, it is enough to 
obtain tightness of the sequence under consideration. Then by \cite[Theorem~VI.4.5]{JS}, we need to verify the compact containment condition 
\begin{align}
&\forall\;\vep,t>0\;\exists\, K>0\mbox{ such that }\sup_{n\in \Bbb N}\P^{(n)}_{\lambda_n}\big(\big\langle M^{(n)},M^{(n)}\big\rangle_t\geq K\big)\leq \vep,\label{cond1}
\end{align}
and Aldous's condition :
\begin{align}
&\forall\;\vep,K>0,\; \lim_{\theta\to 0+}\limsup_{n\to\infty}\sup_{\stackrel{\scriptstyle S,T\in \ms T(n,K)}{\scriptstyle  S\leq T\leq S+\theta}}\P^{(n)}_{\lambda_n}\Big(\big\langle M^{(n)},M^{(n)}\big\rangle_T-\big\langle M^{(n)},M^{(n)}\big\rangle_S\geq \vep\Big)=0.\label{cond2}
\end{align}
For (\ref{cond1}), note that (\ref{MM}) implies
\[
\langle M^{(n)},M^{(n)}\rangle_t\leq 2\left(\frac{\pi^{(n)}_{\max}}{\pi^{(n)}_{\min}}\right)\left( \gamma_n\nu_n(\1)\int_0^t W_1(\xi_{\gamma_ns})ds\right)+\pi^{(n)}_{\max}\gamma_n \mu_n(\1)t,
\]
where the function $W_1$ is defined by (\ref{def:Pell}).
By Proposition~\ref{prop:expbdd} $1^\circ)$, the foregoing inequality and  the validity of condition (ii) of Theorem~\ref{thm:WF}, we deduce that $\langle M^{(n)},M^{(n)}\rangle$ are $L^p$-bounded on compacts for every $p\in [1,\infty)$. This is enough for (\ref{cond1}).

Next, we verify (\ref{cond2}).
For $J,n,K\geq 1$, define
\[
w_J(\alpha,\theta)=\sup_{0\leq t\leq t+\theta\leq J}\sup_{a,b\in [t,t+\theta]}|\alpha(a)-\alpha(b)|
\]
for c\`adl\`ag functions $\alpha:[0,\infty)\to \R$,
and $\ms T(n,K)$ to be the set of all $(\F_t^{(n)})$-stopping times bounded by $K$.
Recall that $M^{(n)}$'s are uniformly bounded on compacts. Then for all $\theta\in(0,1]$ and $S,T\in \ms T(n,K)$
satisfying $0\leq S\leq T\leq S+ \theta$, it follows from the martingale characterization of $\langle M^{(n)},M^{(n)}\rangle$ (cf. \cite[Theorem~I.4.2]{JS}) and the optional stopping theorem \cite[Theorem~II.3.3]{Revuz_2005} that
\begin{align*}
\E^{(n)}_{\lambda_n}\big[\langle M^{(n)},M^{(n)}\rangle_T-\langle M^{(n)},M^{(n)}\rangle_S\big]=&\E_{\lambda_n}^{(n)}\big[(M^{(n)}_T)^2-(M^{(n)}_S)^2\big]\\
=&\E_{\lambda_n}^{(n)}\big[(M^{(n)}_T-M^{(n)}_S)^2\big]\leq \E_{\lambda_n}^{(n)}\big[w_{K+1}(M^{(n)},\theta)^2\big].
\end{align*}
By \cite[Proposition~VI.3.26]{JS}, dominated convergence and the convergence in distribution of $M^{(n)}$ towards the continuous process $M$, we see that the foregoing inequality implies
\begin{align}\label{limitMM}
\lim_{\theta\to 0+}\limsup_{n\to\infty}\sup_{\stackrel{\scriptstyle S,T\in \ms T(n,K)}{\scriptstyle  S\leq T\leq S+\theta}}\E^{(n)}_{\lambda_n}\big[\langle M^{(n)},M^{(n)}\rangle_T-\langle M^{(n)},M^{(n)}\rangle_S\big]=0.
\end{align}
Aldous's condition in (\ref{cond2}) is then satisfied by Chebyshev's inequality and (\ref{limitMM}). The required $C$-tightness follows. 
\\

\noindent \emph{$C$-tightness of the sequence $\big\{\ms L\big(\langle D^{(n)},D^{(n)}\rangle \big)\big\}$.}
By  \cite[Proposition~VI.3.26]{JS}, it is enough to verify tightness of the sequence. For this, we verify the compact containment condition and Aldous's condition for $\langle D^{(n)},D^{(n)}\rangle$ again, that is, analogues of (\ref{cond1}) and (\ref{cond2}) for $\langle D^{(n)},D^{(n)}\rangle$.
First, for the compact containment condition, we have
\begin{align}\label{wc}
w_n^2(\pi_{\min}^{(n)})^{-1}\leq C_{\ref{wc}}\nu_n(\1)
\end{align}
by Assumption~\ref{ass:stat} and Assumption~\ref{ass:w}
so that (\ref{exp:Well}) is applicable to the moment bounds in (\ref{tight_bdd1}). 
Next, for Aldous's condition, we take $\theta\in (0,1]$ and obtain from the strong Markov property of voter models that
\begin{align}
&\sup_{n\in \Bbb N}\sup_{\stackrel{\scriptstyle S,T\in \ms T(n,K)}{\scriptstyle  S\leq T\leq S+\theta}}\E^{(n)}_{\lambda_n}\big[\langle D^{(n)},D^{(n)}\rangle_T-\langle D^{(n)},D^{(n)}\rangle_S\big]\notag\\
\begin{split}\label{Dtight}
\leq &\sup_{n\in \Bbb N}\sup_{\lambda\in \ms P(S^{E_n})}C_{\ref{Dtight}}\sum_{\ell=1}^4
\E^{(n)}_{\lambda}\left[\exp\left(C_{\ref{Dtight}}\gamma_n\nu_n(\1)
\int_0^\theta W_\ell(\xi_{s})ds\right)\right]^{1/2}\\
&\hspace{4cm}\times \E_{\lambda}^{(n)}\left[\left(\gamma_n\nu_n(\1)
\int_0^\theta  W_\ell(\xi_{s})ds\right)^{2a}\right]^{1/2}
\xrightarrow[\theta\to 0+]{}0
\end{split}
\end{align}
for some constant $C_{\ref{Dtight}}$ depending only on $(\Pi,a)$.
Here in (\ref{Dtight}), 
the inequality follows from (\ref{tight_bdd1}) and (\ref{wc}), and the convergence follows from (\ref{exp:Well}) and (\ref{exp:Well-vanish}). \\

\noindent \emph{$C$-tightness of the sequence $\big\{\ms L\big(\langle M^{(n)},D^{(n)}\rangle \big)\big\}$.} The proof is similar to the previous one by verifying conditions analogous to (\ref{cond1}) and (\ref{cond2}) for $\langle M^{(n)},D^{(n)}\rangle$; this uses the moment bound in (\ref{tight_bdd2}) for ${\rm Var}\big(\langle M^{(n)},D^{(n)}\rangle\big)$ now. We omit the details.
In fact, we can get the $C$-tightness of the sequences of laws of $\langle M^{(n)}\pm D^{(n)},M^{(n)}\pm D^{(n)}\rangle$.\\

\noindent \emph{$C$-tightness of the sequence $\big\{\ms L\big(D^{(n)} \big)\big\}$.}
By the $C$-tightness of the sequence $\big\{\ms L\big( \langle D^{(n)},D^{(n)}\rangle \big)\big\}$ proven above, the sequence  $\big\{\ms L\big(D^{(n)}\big)\big\}$ is tight by \cite[Theorem~VI.4.13]{JS}. In addition, it follows from (\ref{taylor})
that, for $(x,y)$ such that $q^{(n)}(x,y)>0$, 
\begin{align}\label{qdiff}
\|q^{(n),w_n}(x,y,\cdot)/q^{(n)}(x,y)-1\|_\infty \leq C_{\ref{qdiff}}w_n
\end{align} 
for some constant $C_{\ref{qdiff}}$ depending only on $\Pi$. Hence, for any $\vep>0$,
the definition (\ref{def:D}) of $D^{(n)}$ implies
\begin{align}\label{Dsupbdd}
\limsup_{n\to\infty}\P^{(n)}_{\lambda_n}\left(\sup_{s:0\leq s\leq t}|\Delta D^{(n)}_s|>\vep\right)\leq \limsup_{n\to\infty}\P^{(n)}_{\lambda_n}\left(\sup_{t:0\leq s\leq t}C_{\ref{qdiff}}w_nD^{(n)}_{s}>\vep\right)=0.
\end{align}
where the last inequality follows from Doob's weak $L^2$-inequality \cite[Theorem~II.1.7]{Revuz_2005} since $\sup_{n\in \Bbb N}\E^{(n)}_{\lambda_n}\big[\big(D^{(n)}_t\big)^2\big]$ is finite by (\ref{cond:ui}), (\ref{exp:Well}) and (\ref{wc}).

By the foregoing display and the tightness of the sequence $ \big\{\ms L(D^{(n)})\big\}$, \cite[Proposition~VI.3.26]{JS} applies and we get the $C$-tightness of the sequence $\big\{\ms L(D^{(n)})\big\}$.\\

\noindent 2$^\circ$) 
First, suppose that 
\begin{align}\label{MD:mart}
\begin{split}
&
D^{(n)},
\langle M^{(n)},M^{(n)}\rangle,
\langle M^{(n)},D^{(n)}\rangle 
,\mbox{and }\langle D^{(n)},D^{(n)}\rangle\\
&\mbox{are $L^p$-bounded on compacts for every $p\in [1,\infty)$.}
\end{split}
\end{align}
In the above proof of the $C$-tightness of the sequence $\big\{\ms L\big(M^{(n)}\big)\big\}$, we have seen that  
$M$ coincides with the martingale part of $Y$.
In addition, it follows from
(\ref{MD:mart}) 
 that $Z$ is a continuous vector semimartingale and $(M,D)$ is a vector martingale with respect to the filtration generated by $(Y,M,D)$.
Below we first show that (\ref{conv:MD}) holds (the argument is very similar to that for \cite[Theorem~5.1 (2) and (3)]{CCC}) and then (\ref{MD:mart}).

By (\ref{MD:mart}) and \cite[Corollary~VI.6.30]{JS}, we obtain 
\begin{align}\label{jointjoint} 
\big((M^{(n)},D^{(n)}),[(M^{(n)},D^{(n)}),(M^{(n)},D^{(n)})]\big)\xrightarrow[n\to\infty]{(\rm d)} \big((M,D),[(M,D),(M,D)]\big).
\end{align}
Since $(M,D)$ is a continuous vector martingale,  $[(M,D),(M,D)]=\langle (M,D),(M,D)\rangle $. 
We also have the uniform integrability of $(D^{(n)})^2$ and $\big\langle (M^{(n)},D^{(n)}),(M^{(n)},D^{(n)})\big\rangle$ from (\ref{MD:mart}), 
and the fact that any weak subsequential limit of the laws of $\langle (M^{(n)},D^{(n)}),(M^{(n)},D^{(n)})\rangle$ must be the law of a matrix of continuous finite variation processes.  
With these considerations,  
we deduce from the martingale characterization of predictable covariations (cf. \cite[Theorem~I.4.2]{JS})
that (\ref{jointjoint}) can be reinforced to  (\ref{conv:MD}).

It remains to prove (\ref{MD:mart}). 
First, for every $a\in (0,\infty)$, $\big(D^{(n)}\big)^a$ are $L^p$-bounded on compacts by 
(\ref{cond:ui}) and Proposition~\ref{prop:expbdd} $1^\circ)$. Second, we have seen that 
$\langle M^{(n)},M^{(n)}\rangle$ are $L^p$-bounded on compacts. 
 Finally, the required $L^p$-boundedness on compacts of $\langle M^{(n)},D^{(n)}\rangle $ and $\langle D^{(n)},D^{(n)}\rangle$ follows from Lemma~\ref{lem:qvbdd} and Proposition~\ref{prop:expbdd} $1^\circ)$. 
We have proved 2$^\circ$).\\

\noindent 3$^\circ$) We have seen in the proof of 2$^\circ$) that $D^{(n)}$ are $L^2$-bounded compacts. This is enough for the required property. 
\end{proof}

\subsection{Proof of the main theorem: identification of limits}\label{sec:id}
The goal of this section is to complete the proof of Theorem~\ref{thm:main2} by proving its 2$^\circ$) and 3$^\circ$). To this end, we  first prove a key `moment-closure property' for the processes $\big(W_\ell(\xi_{\gamma_n t}),\P^{(n)}_{\lambda_n}\big)$,
where $W_\ell$ are defined by (\ref{def:Pell}). Roughly speaking, this property shows that we can approximate these processes by 
polynomial functions of the limiting voter density process in the limit of large population size.

Below we work with the dual functions $H(\xi;x,y)$ defined by (\ref{def:H}), which allow us to invoke the coalescing Markov chains $\{B^x\}$ through the duality equation (\ref{MG}). 
It is also convenient to use the following density functions: for $\ell\geq 1$,
\begin{align}
H_\ell(\xi)= &\sum_{x,y\in E:x\neq y}\pi(x)q^\ell(x,y)H(\xi;x,y)\label{def:Hell}\\
=&W_\ell(\xi)+\sum_{x,y\in E:x\neq y}\pi(x)q^{\ell}(x,y)\big[-\overline{\mu}(1)\dxi(y)-\overline{\mu}(0)\xi(x)+\overline{\mu}(1)\overline{\mu}(0)\big]\notag\\
\begin{split}
=&W_\ell(\xi)-\overline{\mu}(1)[1-p_1(\xi)]-\overline{\mu}(0)p_1(\xi)\\
&+\sum_{x\in E}\pi(x)q^\ell(x,x)\big[\overline{\mu}(1)\dxi(x)
+\overline{\mu}(0)\xi(x)\big]
+\sum_{x,y\in E:x\neq y}\pi(x)q^\ell(x,y)\overline{\mu}(1)\overline{\mu}(0).\label{Hell}
\end{split}
\end{align}

\begin{lem}\label{lem:Hmart}
{\rm 1$^\circ$)} Fix $x\neq y$. For any $\lambda\in \ms P(S^E)$, the process
\begin{align}\label{def:Mxy}
\begin{split}
M^{x,y}_t= &\;e^{2(1+\mu(\1))t}H(\xi_t;x,y)-H(\xi;x,y)\\
&-\int_0^t e^{2(1+\mu(\1))s}\left(\sum_{z\in E}q(x,z)H(\xi_s;z,y)+\sum_{z\in E}q(y,z)H(\xi_s;x,z)\right) ds
\end{split}
\end{align}
is an $(\F_t,\P_\lambda)$-martingale.\\

\noindent {\rm 2$^\circ$)} For any $\ell\in \Bbb N$ and $\lambda\in \ms P(S^E)$, the $(\F_t,\P_\lambda)$-martingale
\begin{align}\label{def:Mell}
M^\ell_t=\sum_{x,y\in E:x\neq y}\pi(x)q^\ell(x,y)M^{x,y}_t
\end{align}
 satisfies
\begin{align}\label{ineq:Mell}
\E_\xi\big[(M^\ell_t)^2\big]\leq 18\pi_{\max}\int_0^t e^{4(1+\mu(\1))s}\E_\xi[W_1(\xi_s)]ds+\frac{9\mu(\1)\pi_{\max}}{4+4\mu(\1)}e^{4(1+\mu(\1))t}.
\end{align} 
\end{lem}
\begin{proof}
1$^\circ$) Fix $x\neq y$. Let $J$ denote the first epoch time of the set-valued process $B^{\{x,y\}}_t=\{B^x_t,B^y_t\}$. For $t\geq J$, we have
\[
H(\xi;B^{x}_t,B^y_t)\exp\left(-\mu(\1)\int_0^t |B^{\{x,y\}}_s|ds\right)=H(\xi;B^{x}_t,B^y_t)\exp\left(-2\mu(\1)J-\mu(\1)\int_J^t |B^{\{x,y\}}_s|ds\right)
\]
and
\[
\int_0^t\1_{\{B^x_s=B^y_s\}}\exp\left(-\mu(\1)\int_0^s |B^{\{x,y\}}_r|dr\right)ds=\int_J^{t}\1_{\{B^x_s=B^y_s\}}\exp\left(-2\mu(\1)J-\mu(\1)\int_J^s |B^{\{x,y\}}_r|dr\right)ds.
\]
We apply to the right-hand side of (\ref{MG}) an argument similar to (\ref{F_recur}). Then the above two displays and (\ref{MG}) imply
that
\begin{align*}
&\E_\xi[H(\xi_t;x,y)]=e^{-2(1+\mu(\1))t}H(\xi;x,y)\\
&+\int_0^t e^{-2(1+\mu(\1))(t-s)}\left(\sum_{z\in E}q(x,z)\E_\xi[H(\xi_{s};z,y)]+\sum_{z\in E}q(y,z)\E_\xi[H(\xi_{s};x,z)]\right)ds.
\end{align*}
The foregoing equality proves $\E_\xi[M^{x,y}_t]=0$ for all $\xi\in S^E$ and $t$, and the required result then follows from the Markov property of $(\xi_t)$. \\

\noindent {\rm 2$^\circ$)}
By \cite[Lemma~I.4.14 (b), Lemma~I.4.51]{JS}, the quadratic variation of $M^\ell$ is given by
\begin{align*}
[M^\ell,M^\ell]_t=&\sum_{s:s\leq t}(\Delta M_s^\ell)^2=\sum_{x,y\in E}\int_0^t (\Delta M^\ell_s)^2d\Lambda_s(x,y)+\sum_{\sigma\in S}\sum_{x\in E}\int_0^t (\Delta M^\ell_s)^2 d\Lambda^\sigma_s(x).
\end{align*}
Notice that
\[
\Delta M^\ell_s=\sum_{x,y\in E:x\neq y}\pi(x)q^\ell(x,y)e^{2(1+\mu(\1))s}\Delta H(\xi_s;x,y)=e^{2(1+\mu(\1))s}\Delta H_\ell(\xi_s),
\]
where the first equality follows from  (\ref{def:Mxy}) and (\ref{def:Mell}) and the last one from the definition (\ref{def:Hell}) of $H_\ell$. Putting the last two displays together and using the definition of $(\xi_t)$, we get
\begin{align}
\begin{split}
[M^\ell,M^\ell]_t=&\sum_{x,y\in E}\int_0^t e^{4(1+\mu(\1))s}[\xi_{s-}(x)\dxi_{s-}(y)+\dxi_{s-}(x)\xi_{s-}(y)]\\
&\times \big[H_\ell\big((\xi_{s-})^x\big)-H_\ell(\xi_{s-})\big]^2d\Lambda_s(x,y)\\
&+\sum_{\sigma\in S}\sum_{x\in E}\int_0^te^{4(1+\mu(\1))s}\big[H_\ell\big((\xi_{s-})^{x|\sigma}\big)-H_\ell(\xi_{s-})\big]^2d\Lambda^\sigma_s(x).\label{Mellsq}
\end{split}
\end{align} 
To bound the right-hand side of the above equality, notice that the definition of $W_\ell$ in (\ref{def:Pell}) implies
\begin{align*}
W_\ell(\xi^x)-W_\ell(\xi)=\pi(x)\sum_{y\in E}q^\ell(x,y)\big[\xi(x)\xi(y)-\xi(x)\dxi(y)+\dxi(x)\dxi(y)-\dxi(x)\xi(y)\big],\quad \forall\;x\in E.
\end{align*}
The foregoing equality and (\ref{Hell}) then imply that
\[
|H_\ell(\xi^x)-H_\ell(\xi)|\leq 3\pi(x).
\]
By Poisson calculus, (\ref{Mellsq}) implies
\begin{align*}
\E_\xi[(M^\ell_t)^2]=&\E_\xi\big[[M^\ell,M^\ell]_t\big]\leq \sum_{x,y\in E}\int_0^te^{4(1+\mu(\1))s}\E_\xi[\xi_{s-}(x)\dxi_{s-}(y)+\dxi_{s-}(x)\xi_{s-}(y)]9\pi(x)^2q(x,y)ds \\
&+\sum_{\sigma\in S}\sum_{x\in E}\int_0^t e^{4(1+\mu(\1))s}9\pi(x)^2\mu(\sigma)ds\\
\leq &18\pi_{\max}\int_0^t e^{4(1+\mu(\1))s}\E_{\xi}[W_1(\xi_s)]ds+\frac{9\mu(\1)\pi_{\max}}{4+4\mu(\1)}\big(e^{4(1+\mu(\1))t}-1\big),
\end{align*}
which gives (\ref{ineq:Mell}). The proof is complete.
\end{proof}

We are ready to prove the moment closure property announced before.

\begin{prop}\label{prop:momclos}
Under Assumption~\ref{ass:stat}, Assumption~\ref{ass:vm} and Assumption~\ref{ass:homo}, we have the following.

\begin{enumerate}
\item [\rm 1$^\circ$)] For all $1\leq \ell\leq 2$, we have the following convergence in distribution of continuous processes:
\begin{align}
\begin{split}
&\left(\gamma_n\nu_n(\1)\int_0^t \big[W_{\ell+1}(\xi_{\gamma_n s})- W_\ell(\xi_{\gamma_n s})-R_\ell W_1(\xi_{\gamma_n s})\big]ds\right)_{t\geq 0}\xrightarrow[n\to\infty]{(\rm d)}0,\label{mom:1-1}
\end{split}
\end{align}
where $R_\ell$ are chosen in {\rm Assumption~\ref{ass:homo}} with $L=2$. 
An analogous result for the convergence in (\ref{mom:1-1}) with $\ell=3$ holds if {\rm Assumption~\ref{ass:homo}} with $L=3$ applies.
\item [\rm 2$^\circ$)] If, moreover, the voting kernels $q^{(n)}$ are symmetric, then we have
\begin{align}
&\left(\gamma_n\nu_n(\1)\int_0^t W_1(\xi_{\gamma_ns})ds-\frac{1}{2}\int_0^t  Y_s^{(n)}(1-Y^{(n)}_s)ds\right)_{t\geq 0}\xrightarrow[n\to\infty]{(\rm d)}0.\label{mom:1-0}
\end{align}

\item [\rm 3$^\circ$)] Convergences in distribution of one-dimensional marginals of the processes in {\rm 1$^\circ$)} and {\rm 2$^\circ$)} 
can be reinforced to $L^p$-convergences for any $p\in [1,\infty)$. 
\end{enumerate}
\end{prop}

The proofs of Proposition~\ref{prop:momclos} 1$^\circ)$ and its extension in $3^\circ)$ begin with the proof of a particular convergence in the following lemma. 

\begin{lem}\label{lem:momclos}
Under Assumption~\ref{ass:stat}, Assumption~\ref{ass:vm} and Assumption~\ref{ass:homo}, we have, for all $t\in(0,\infty)$ and $1\leq \ell\leq 2$,
\begin{align}\label{lim:unif}
\lim_{n\to\infty}\sup_{\lambda\in \ms P(S^{E_n})}\E^{(n)}_\lambda\left[\left(\gamma_n\nu_n(\1)\int_0^t \big[W_{\ell+1}(\xi_{\gamma_n s})- W_\ell(\xi_{\gamma_n s})-R_\ell W_1(\xi_{\gamma_n s})\big]ds\right)^2\right]=0.
\end{align}
The above convergence for $\ell=3$ holds if {\rm Assumption~\ref{ass:homo}} with $L=3$ applies.
\end{lem}
\begin{proof}
Recall the functions $H_\ell(\xi)$ and the martingales $M^\ell$ defined in (\ref{def:Hell}) and (\ref{def:Mell}), respectively. 
By the reversibility of $q^{(n)}$, the equation satisfied by the martingale $M^{(n),\ell}$ under $\P^{(n)}_\lambda$ can be written as
\begin{align}
M^{(n),\ell}_t=&e^{2(1+\mu_n(\1))t}H_\ell(\xi_t)-H_\ell(\xi_0)
-\int_0^t e^{2(1+\mu_n(\1))s}
\sum_{x,y\in E_n:x\neq y}\pi^{(n)}(x)q^{(n),\ell}(x,y)\notag\\
&\times \left(\sum_{z\in E_n}q^{(n)}(x,z)H(\xi_s;z,y)+\sum_{z\in E_n}q^{(n)}(y,z)H(\xi_s;x,z)\right) ds\notag\\
\begin{split}
=&\;e^{2(1+\mu_n(\1))t}H_\ell(\xi_t)-H_\ell(\xi_0)
-2\int_0^t e^{2(1+\mu_n(\1))s}H_{\ell+1}(\xi_s) ds
+\int_0^t e^{2(1+\mu_n(\1))s}I(\xi_s)ds,\label{Mell1}
\end{split}
\end{align}
where the function $I(\xi)$ is given by
\begin{align}
\begin{split}\label{def:I}
I(\xi)=&-2\sum_{x\in E_n}\pi^{(n)}(x)q^{(n),\ell+1}(x,x)H(\xi;x,x)\\
&\hspace{-.2cm}+\sum_{x\in E_n}\pi^{(n)}(x)q^{(n),\ell}(x,x) \left(\sum_{z\in E_n}q^{(n)}(x,z)H(\xi;z,x)+\sum_{z\in E_n}q^{(n)}(x,z)H(\xi;x,z)\right).
\end{split}
\end{align}
Therefore from (\ref{Mell1}), we have, for fixed $t\in (0,\infty)$,
\begin{align}
\int_0^{t} e^{-2(1+\mu_n(\1))s}M^{(n),\ell}_{s}ds
=&\int_0^t H_\ell(\xi_{s})ds
-\int_0^t e^{-2 (1+\mu_n(\1))s}dsH_\ell(\xi_0)
\notag\\
&-\int_0^t\int_0^s 2e^{-2(1+\mu_n(\1))s+2(1+\mu_n(\1))r}H_{\ell+1}(\xi_r)drds\notag\\
&+\int_0^t\int_0^s e^{-2(1+\mu_n(\1))s+2(1+\mu_n(\1))r}I(\xi_r)drds\notag\\
\begin{split}\label{eq:momcl}
=&\int_0^t H_\ell(\xi_s)ds-\int_0^t e^{-2(1+\mu_n(\1))s}dsH_\ell(\xi_0)-\frac{2}{2+2\mu_n(\1)}
\int_0^tH_{\ell+1}(\xi_s)ds\\
&+\frac{2}{2+2\mu_n(\1)}\int_0^te^{-2(1+\mu_n(\1))(t-s)}H_{\ell+1}(\xi_s)ds\\
&+\frac{1}{2+2\mu_n(\1)}\int_0^tI(\xi_s)ds
-\frac{1}{2+2\mu_n(\1)}\int_0^te^{-2(1+\mu_n(\1))(t-s)}I(\xi_s)ds.
\end{split}
\end{align}

We multiply both sides of (\ref{eq:momcl}) by $\nu_n(\1)$ and change time scales by replacing $t$ by $\gamma_n t$ for fixed $t\in (0,\infty)$. Now suppose that
\begin{align}
&\lim_{n\to\infty}\sup_{\lambda\in \ms P(S^{E_n})}\E^{(n)}_{\lambda}\left[\left(\gamma_n\nu_n(\1)\int_0^t e^{-2(1+\mu_n(\1))\gamma_n s}M^{(n),\ell}_{\gamma_n s}ds\right)^2\right]=0,\label{L2-1}\\
&\lim_{n\to\infty}\sup_{\lambda\in \ms P(S^{E_n})}\E^{(n)}_{\lambda}\left[\left(\gamma_n\nu_n(\1)\int_0^t e^{-2(1+\mu_n(\1))\gamma_n s}dsH_\ell(\xi_0)\right)^2\right]=0,\label{L2-2}\\
&\lim_{n\to\infty}\sup_{\lambda\in \ms P(S^{E_n})}\E^{(n)}_{\lambda}\Bigg[\Bigg(\frac{2\gamma_n\nu_n(\1)}{2+2\mu_n(\1)}\int_0^t e^{-2(1+\mu_n(\1)) \gamma_n(t-s)}H_{\ell+1}(\xi_{\gamma_ns})ds\Bigg)^2\Bigg]=0,\label{L2-3}\\
&\lim_{n\to\infty}\sup_{\lambda\in \ms P(S^{E_n})}\E^{(n)}_{\lambda}\Bigg[\Bigg(\frac{\gamma_n\nu_n(\1)}{2+2\mu_n(\1)}\int_0^te^{-2(1+\mu_n(\1)) \gamma_n(t-s)}I(\xi_{\gamma_ns}) ds\Bigg)^2\Bigg]=0,\label{L2-4}
\end{align}
which are used to handle terms among those on the two sides of (\ref{eq:momcl}). Then our focus for (\ref{eq:momcl}) will be on the remaining terms, that is the first, third, and fifth terms on its right-hand side (after multiplying them by $\nu_n(\1)$ and changing $t$ to $\gamma_nt$). By (\ref{eq:momcl}) and the assumed identities (\ref{L2-1})--(\ref{L2-4}), we have
\begin{align*}
\lim_{n\to\infty}\sup_{\lambda\in \ms P(S^{E_n})}&\E^{(n)}_{\lambda}\Bigg[\Bigg(\gamma_n\nu_n(\1)\int_0^t H_\ell(\xi_{\gamma_n s})ds-\frac{2\gamma_n\nu_n(\1)}{2+2\mu_n(\1)}
\int_0^tH_{\ell+1}(\xi_{\gamma_n s})ds \\
&\hspace{6cm}+\frac{\gamma_n\nu_n(\1)}{2+2\mu_n(\1)}\int_0^tI(\xi_{\gamma_n s})ds \Bigg)^2\Bigg]=0.
\end{align*}
Recall that $\sup_n\gamma_n\mu_n(\1)<\infty$ by Assumption~\ref{ass:vm}, $H_\ell$ are uniformly bounded, and $\pi^{(n)}$'s are comparable to uniform distributions by Assumption~\ref{ass:stat}. Hence, the foregoing equality implies
\begin{align}\label{mc:lim}
&\lim_{n\to\infty}\sup_{\lambda\in \ms P(S^{E_n})}\E^{(n)}_{\lambda}\Bigg[\Bigg(\gamma_n\nu_n(\1)\int_0^t \Big(H_\ell(\xi_{\gamma_n s})- 
 H_{\ell+1}(\xi_{\gamma_n s})+\frac{1}{2}I(\xi_{\gamma_ns})\Big)ds\Bigg)^2\Bigg]=0.
\end{align}

We show that the foregoing limit implies the required limit (\ref{lim:unif}).
We use (\ref{Hell}) and (\ref{def:I}) to write out the integrand in (\ref{mc:lim}):
\begin{align*}
&\gamma_n\nu_n(\1)\int_0^t \Big(H_\ell(\xi_{\gamma_n s})- 
 H_{\ell+1}(\xi_{\gamma_n s})+\frac{1}{2}I(\xi_{\gamma_ns})\Big)ds\\
=&\gamma_n\nu_n(\1)\int_0^t W_\ell(\xi_{\gamma_ns})ds-\gamma_n\nu_n(\1)\int_0^t W_{\ell+1}(\xi_{\gamma_n s})ds\\
&+\gamma_n\nu_n(\1)\int_0^t \sum_{x\in E_n}\pi^{(n)}(x)\big(q^{(n),\ell}(x,x)-q^{(n),\ell+1}(x,x)\big)\big(\overline{\mu}_n(1)\dxi_{\gamma_n s}(x)+\overline{\mu}_n(0)\xi_{\gamma_n s}(x)\big)ds\\
& +\gamma_n\nu_n(\1)\int_0^t \sum_{x,y\in E_n:x\neq y}\pi^{(n)}(x)\big(q^{(n),\ell}(x,y)-q^{(n),\ell+1}(x,y)\big)\overline{\mu}_n(1)\overline{\mu}_n(0)ds\\
&-\gamma_n\nu_n(\1)\int_0^t \sum_{x\in E_n}\pi^{(n)}(x)q^{(n),\ell+1}(x,x)\big(-\overline{\mu}_n(1)\dxi_{\gamma_n s}(x)-\overline{\mu}_n(0)\xi_{\gamma_n s}(x)+\overline{\mu}_n(1)\overline{\mu}_n(0)\big)ds
\\
&+\frac{\gamma_n\nu_n(\1)}{2}\int_0^t\sum_{x\in E_n}\pi^{(n)}(x)q^{(n),\ell}(x,x)\sum_{z\in E_n}q^{(n)}(x,z)\big(\dxi_{\gamma_n s}(x)\xi_{\gamma_n s}(z)+\dxi_{\gamma_n s}(z)\xi_{\gamma_n s}(x)\big)ds
\\
&+\frac{\gamma_n\nu_n(\1)}{2}\int_0^t \sum_{x\in E_n}\pi^{(n)}(x)q^{(n),\ell}(x,x)\Bigg(\sum_{z\in E_n}q^{(n)}(x,z)\big(-\overline{\mu}(1)\dxi_{\gamma_n s}(x)-\overline{\mu}_n(0)\xi_{\gamma_n s}(z)+\overline{\mu}_n(1)\overline{\mu}_n(0)\big)
\\
&+\sum_{z\in E_n}q^{(n)}(x,z)\big(-\overline{\mu}(1)\dxi_{\gamma_n s}(z)-\overline{\mu}_n(0)\xi_{\gamma_n s}(x)+\overline{\mu}_n(1)\overline{\mu}_n(0)\big)\Bigg)ds\\
=&\gamma_n\nu_n(\1)\int_0^t W_\ell(\xi_{\gamma_n s})ds-\gamma_n\nu_n(\1)\int_0^t W_{\ell+1}(\xi_{\gamma_n s})ds\\
&+\frac{\gamma_n\nu_n(\1)}{2}\int_0^t\sum_{x\in E_n}\pi^{(n)}(x)q^{(n),\ell}(x,x)\sum_{z\in E_n}q^{(n)}(x,z)\big[\dxi_{\gamma_n s}(x)\xi_{\gamma_n s}(z)+\dxi_{\gamma_n s}(z)\xi_{\gamma_n s}(x)\big]ds\\
&+\frac{\gamma_n\nu_n(\1)}{2}\int_0^t \sum_{x\in E_n}\pi^{(n)}(x)q^{(n),\ell}(x,x)\Bigg(\overline{\mu}(0)\sum_{z\in E_n}q^{(n)}(x,z)\big[\xi_{\gamma_n s}(x)-\xi_{\gamma_n s}(z)\big]\\
&\hspace{.5cm}+\overline{\mu}(1)\sum_{z\in E_n}q^{(n)}(x,z)\big[\dxi_{\gamma_ns}(x)-\dxi_{\gamma_n s}(z)\big]\Bigg)ds.
\end{align*}
We use Assumption~\ref{ass:homo} to handle the last equation.  
For $1\leq \ell\leq 2$, it follows from the validity of (\ref{cond:hom}) with $L=2$ and Proposition~\ref{prop:expbdd} that, with respect to
the uniform $L^2$-limit as in (\ref{lim:unif}),
only the first three terms on the right-hand side of the above equality can survive
and the third term approximates
$\gamma_n\nu_n(\1) \int_0^t R_\ell W_1(\xi_{\gamma_n s})ds $. Hence, (\ref{mc:lim}) imply (\ref{lim:unif}).

We still need to verify the limits in (\ref{L2-1})--(\ref{L2-4}). For (\ref{L2-1}), we use the fact that $M^{(n),\ell}$ is a martingale in the first equality below and then Lemma~\ref{lem:Hmart} 2$^\circ$) in the first inequality: \begin{align*}
&\E^{(n)}_{\lambda}\left[\left(\gamma_n\nu_n(\1)\int_0^t e^{-2(1+\mu_n(\1))\gamma_n s}M^{(n),\ell}_{\gamma_n s}ds\right)^2\right]\\
=&2\gamma_n^2\nu_n(\1)^2\int_0^t\int_0^se^{-2(1+\mu_n(\1))\gamma_n r-2(1+\mu_n(\1))\gamma_n s}\E^{(n)}_{\lambda}[(M^{(n),\ell}_{\gamma_n r})^2]drds\\
\leq &36\gamma_n^3\nu_n(\1)^2\pi^{(n)}_{\max}\int_0^t\int_0^s e^{-2(1+\mu_n(\1))\gamma_n r-2(1+\mu_n(\1))\gamma_n s}\int_0^r e^{4(1+\mu_n(\1))\gamma_n q}\E^{(n)}_{\lambda}[W_1(\xi_{\gamma_n q})]dqdrds\\
&+\frac{9\gamma_n^2\nu_n(\1)^2\mu_n(\1)\pi_{\max}^{(n)}}{2+2\mu_n(\1)}\int_0^t \int_0^s e^{2(1+\mu_n(\1))\gamma_n r-2(1+\mu_n(\1))\gamma_n s}drds\\
\leq  &\frac{36\nu_n(\1)\pi^{(n)}_{\max}}{8[1+\mu_n(\1)]^2} \left(\gamma_n\nu_n(\1)\int_0^t \E^{(n)}_{\lambda}[W_1(\xi_{\gamma_n q})]dq\right)+\frac{9\gamma_n \nu_n(\1)^2\mu_n(\1)\pi^{(n)}_{\max}t}{[2+2\mu_n(\1)]^2}\xrightarrow[n\to\infty]{}0,
\end{align*}
where the convergence 
is uniform in $\lambda\in \ms P(S^{E_n})$ and
follows from
 Proposition~\ref{prop:expbdd} $1^\circ)$ and the fact that $\pi^{(n)}$'s are comparable to uniform distributions and $\sup_n\gamma_n\mu_n(\1)<\infty$ by (\ref{conv:WF}). To obtain the remaining limits (\ref{L2-2})--(\ref{L2-4}), we  note that
\begin{align*}
&\left(\nu_n(\1)\int_0^{\gamma_n t}e^{-2(1+\mu_n(\1))s}ds\right)^2\leq \nu_n(\1)^2\xrightarrow[n\to\infty]{}  0.
\end{align*}
The proof is complete.
\end{proof}

\paragraph{\bf Proof of Proposition~\ref{prop:momclos}}
We start with the proof of $1^\circ)$.  
The sequence of laws of the processes on the left-hand side of (\ref{mom:1-1}) is tight by Proposition~\ref{prop:expbdd} $1^\circ)$ and \cite[Theorem~VI.4.5]{JS} (cf. the proof of Proposition~\ref{prop:tight}). 
By Lemma~\ref{lem:momclos} 
 and \cite[Lemma~3.7.8 (b)]{EK:MP}, we deduce that the sequence converges in distribution to the zero process.  
 
Since voting kernels are symmetric, the proof of $2^\circ)$ can be obtained by the same argument as that of \cite[Corollary~5.2]{CCC}. It can be detailed as follows. 
Since $\langle Y,Y\rangle_t=\int_0^t Y_s(1-Y_s)ds$ and 
\begin{align}\label{intmap}
w\lmt  \left(\int_0^t f\big(w(s)\big)ds\right)_{t\geq 0}:D(\R_+,[0,1])\to D(\R_+,\R)\quad\mbox{ is continuous}
\end{align}
for every bounded continuous function $f:\R\to \R$ by the proof of \cite[Proposition~3.7.1]{EK:MP}, Assumption~\ref{ass:vm} and Proposition~\ref{prop:tight} imply that
\begin{align*}
&\left(\gamma_n\nu_n(\1)\int_0^t W_1(\xi_{\gamma_ns})ds-\frac{1}{2}\int_0^t Y_s^{(n)}(1-Y^{(n)}_s)ds\right)_{t\geq 0}\\
=&\left(\frac{1}{2}\langle M^{(n},M^{(n)}\rangle_t-\frac{1}{2}\int_0^t Y_s^{(n)}(1-Y^{(n)}_s)ds\right)_{t\geq 0}\xrightarrow[n\to\infty]{(\rm d)}\left(\frac{1}{2}\langle M,M\rangle_t-\frac{1}{2}\int_0^tY_s(1-Y_s)ds\right)_{t\geq 0},
\end{align*}
which is the zero process.

Finally, $L^p$-convergences of the one-dimensional marginals of the processes under consideration are immediate consequences of Proposition~\ref{prop:expbdd} $1^\circ)$ and a standard result of uniform integrability. This proves $3^\circ)$.  
\qed \\

The following proposition proves Theorem~\ref{thm:main2} 2$^\circ$) and 3$^\circ$), and thus,
completes the proof of the theorem.

\begin{prop}\label{prop:id}
Suppose that the assumptions for {\rm Theorem~\ref{thm:main2} 2$^\circ$)} are in force, and by choosing a subsequence if necessary, $(Y^{(n)},M^{(n)}, D^{(n)})$ under $\P^{(n)}_{\lambda_n}$ converges in distribution to a continuous vector semimartingale $(Y,M,D)$ under $\P^{(\infty)}$. Then 
\begin{enumerate}
\item [\rm 1$^\circ$)]
$(Y,M,D)$ satisfies the covariation equations in (\ref{YD:covar}).
\item [\rm 2$^\circ$)] Moreover, if {\rm Assumption~\ref{ass:homo}} with $L=3$ applies, then $(Y,M,D)$ can be characterized by the system in (\ref{SDE}). 
\end{enumerate}
In addition, there is pathwise uniqueness in the system in (\ref{SDE}).
\end{prop}
\begin{proof}In this proof, we write (d)-$\lim_{n\to\infty}$ for limits in distribution of continuous processes defined under $\P^{(n)}_{\lambda_n}$, with the limiting objects defined under $\P^{(\infty)}$.\\

\noindent  1$^\circ$) 
We start with some limiting identities. First, by Proposition~\ref{prop:momclos} $1^\circ)$ and $2^\circ)$, we deduce that
\begin{align}
\begin{split}
\label{conv:WR}
&\mbox{(d)-}\lim_{n\to\infty}\gamma_n\nu_n(\1)\int_0^\cdot   W_{\ell+1}(\xi_{\gamma_ns})ds
-\frac{R_\ell+\cdots+R_0}{2}\int_0^\cdot   Y_s^{(n)}(1-Y^{(n)}_s)ds=0
\end{split}
\end{align}
for all $0\leq \ell\leq 2$ (recall that $R_0=1$). By the assumed convergence in distribution of $(Y^{(n)},D^{(n)})$ towards $(Y,D)$, (\ref{intmap}) and (\ref{conv:WR}), we get the following convergence of two-dimensional processes:
\begin{align*}
\mbox{(d)-}\lim_{n\to\infty}\left(\gamma_n\nu_n(\1)\int_0^\cdot   W_{\ell+1}(\xi_{\gamma_ns})ds,D^{(n)}\right)=\left(\frac{R_\ell+\cdots+R_0}{2}\int_0^\cdot   Y_s(1-Y_s)ds,D\right).
\end{align*}
Hence, by \cite[Proposition~VI.6.12, Theorem~VI.6.22]{JS}, 
we deduce that, for all integers $m\geq 0$ and $0\leq \ell\leq 2$,
\begin{align}
\begin{split}
\label{conv:WR1}
\mbox{(d)-}\lim_{n\to\infty}\gamma_n\nu_n(\1)\int_0^\cdot  \big(D^{(n)}_s\big)^m  W_{\ell+1}(\xi_{\gamma_ns})ds
=\frac{R_\ell+\cdots+R_0}{2}\int_0^\cdot  D^m_s Y_s(1-Y_s)ds.
\end{split}
\end{align}
Second, since 
\begin{align}\label{Rn}
|R^{w_n}_1(\xi_{\gamma_n s})|\leq C_{\ref{Rn}} \sum_{\ell=1}^4W_\ell(\xi_{\gamma_n s})
\end{align}
for some constant $C_{\ref{Rn}}$ depending only on $\Pi$ by (\ref{Dbdd1}) and the definition (\ref{def:R1w}) of $R^w_1$, (\ref{conv:WR1}) and Assumption~\ref{ass:stat} imply that
\begin{align}\label{Rnlimit}
\mbox{(d)-}\lim_{n\to\infty}w_n^2
\gamma_n\int_0^\cdot  D^{(n)}_s  R^{w_n}_{1}(\xi_{\gamma_ns})ds
=0.
\end{align}

We are ready to prove (\ref{YD:covar}).
Recall the definition of $w_\infty$ in (\ref{selection2}).
Then applying Proposition~\ref{prop:tight} 2$^\circ$), (\ref{Dp}), (\ref{Rnlimit}), (\ref{eq:Dbar}) and (\ref{conv:WR1}) in order below shows that 
\begin{align}
\langle M,D\rangle=&\mbox{(d)-}\lim_{n\to\infty}\langle M^{(n)},D^{(n)}\rangle
=\mbox{(d)-}\lim_{n\to\infty}w_n\gamma_n\int_0^\cdot  D^{(n)}_{s}\overline{D}^{(n)}(\xi_{\gamma_n s})ds\notag\\
=&\mbox{(d)-}\lim_{n\to\infty}\left(\frac{w_n}{\nu_n(\1)}\right)\gamma_n\nu_n(\1)\int_0^\cdot  D^{(n)}_{s}\big[b\big(W_3(\xi_{\gamma_ns})-W_1(\xi_{\gamma_n s})\big)
-cW_2(\xi_{\gamma_n s})\big]ds\notag\\
=& w_\infty \left(\frac{b(R_2+R_1)-c(R_1+R_0)}{2}\right)\int_0^\cdot  D_{s}Y_s(1-Y_s)ds,\label{MD:final}
\end{align}
as required.\\ 

\noindent 2$^\circ$) Suppose that Assumption~\ref{ass:homo} with $L=3$ applies. We use (\ref{DD}) and (\ref{eq:Dbar2}) in place of (\ref{Dp}) and (\ref{eq:Dbar}) in the proof of (\ref{MD:final}), respectively, and obtain 
\begin{align}
\langle D,D\rangle=&\mbox{(d)-}\lim_{n\to\infty}w_n^2\gamma_n\int_0^\cdot  \big(D^{(n)}_{s}\big)^2\sum_{x,y\in E_n}q^{(n)}(x,y)[A(x,\xi_{\gamma_ns})-B(y,\xi_{\gamma_ns})]^2ds\notag\\
=&\mbox{(d)-}\lim_{n\to\infty}\left(\frac{w_n^2N_n}{\nu_n(\1)}\right)\gamma_n\nu_n(\1)\notag\\
&\times \int_0^\cdot  \big(D^{n}_{s}\big)^2 \Big[b^2\big(W_4(\xi_{\gamma_ns})-W_2(\xi_{\gamma_n s})\big)
-2bc\big(W_3(\xi_{\gamma_n s})-W_1(\xi_{\gamma s})\big)+c^2W_2(\xi_{\gamma_ns})\Big]ds\notag\\
=& w_\infty^2 \left(\frac{b^2(R_3+R_2)-2bc(R_2+R_1)+c^2(R_1+R_0)}{2}\right)\int_0^\cdot  D^2_{s}Y_s(1-Y_s)ds.\label{DDl}
\end{align}
Notice that we have use the fact that $\nu_n(\1)=N_n^{-1}$ by the assumed symmetry of $q^{(n)}$. 
Putting (\ref{conv:WF}), (\ref{MD:final}), and (\ref{DDl}) together,
we see that the continuous vector semimartingale $(Y,M,D)$ satisfies the system in (\ref{SDE}) by an enlargement of the underlying probability space if necessary (cf. \cite[Theorem VII.2.7]{Revuz_2005}).

It remains to prove pathwise uniqueness in the system defined by (\ref{SDE}).
Notice that $D$
is equal to the Dol\'eans-Dade exponential of $\widetilde{M}$, that is $D=\mathcal  E(\widetilde{M})$, by \cite[Theorem IX.2.1]{Revuz_2005}, where $\widetilde{M}$ is a martingale given by
\[
\widetilde{M}_t=\int_0^t w_\infty\, \sqrt{Y_s(1-Y_s)}\Big(K_1(b,c)dW^{1}_s+\sqrt{K_2(b,c)-K_1(b,c)^2}dW^{2}_s\Big).
\]
The equality $D=\mathcal E(\widetilde{M})$ and pathwise uniqueness in the closed system of $(Y,M)$ by the Yamada-Watanabe theorem \cite[Theorem~IX.3.5]{Revuz_2005} plainly imply pathwise uniqueness in the system defined by (\ref{SDE}). The proof is complete.
\end{proof}

\section{Wasserstein convergence of occupation measures of the game density processes}\label{sec:wass}
Throughout this section, we consider voter models or evolutionary games without mutation. 
With respect to an evolutionary game with its generator as $\mathsf L^{w,\mu}$, we write
 $T=T_{\mathbf 1}\wedge T_{\mathbf 0}$'s for its consensus time and $T_\eta$ denotes the first hitting time of a population configuration $\eta$. We also write $\bs\sigma$ for the all-$\sigma$ population configuration.  
 
Our goal in this section is to prove Wasserstein convergence of the occupation measures 
\begin{align}\label{occupy}
\int_0^\infty f(Y^{(n)}_t)dt=\int_0^{T^{(n)}}f(Y^{(n)}_t)dt,
\end{align}
of the time-changed game density processes $Y^{(n)}$ defined by (\ref{def:Zn}) under $\P^{(n),w_n}$. Here, under $\P^{(n),w_n}$, $T^{(n)}=T/\gamma_n$ and $T=T_\mathbf 1\wedge T_{\mathbf 0}$ is the consensus time of the evolutionary game with its generator given by $\mathsf L^{w_n,\mu_n}$; $T_\eta$ is the first hitting time of $\eta$ of the evolutionary game. The study below uses a good control of the rescaled times $T^{(n)}$ under voter models due to Oliveira \cite{Oliveira_2012, Oliveira_2013}, and thus starts with an investigation of links between these random objects and convergences of the occupation measures in (\ref{occupy}).

Since mutation is absent, 
the Feynman-Kac duality  between voter models and coalescing Markov chains as discussed at the beginning of Section~\ref{sec:weak} is simpler and gives the following moment duality equations: with respect to a given voting kernel $(E,q)$, we have
\begin{align}\label{duality}
\E_\xi\left[\,\prod_{x\in A}\xi_t(x)\right]=\E\left[\,\prod_{x\in A}\xi(B^x_t)\right],\quad \forall \;\xi \in S^E,\;A\subseteq E.
\end{align}
See (\ref{gen1}) for the proof of (\ref{duality}). Also, 
we write $\P_{\bs x}$ for $\bs x=(x_1,\cdots,x_N)\in E^N$ when a system of coalescing $(E,q)$-Markov chains starting from distinct components in $\bs x$ is under consideration. The first time that the number of distinct components in a system of coalescing Markov chains becomes less than or equal to $k$ is denoted by $\mathsf C_k$. Finally, we write $\bs x_E\in E^{N}$ for a vector whose components range over all points of $E$.

\begin{lem}\label{lem:coal}
Let a sequence of voting kernels $(E_n,q^{(n)})$ and a sequence of constants $\gamma_n$ increasing to infinity
be given such that, for ${\mathsf C}_1^{(n)}={\mathsf C}_1/\gamma_n$ under $\P^{(n)}_{\bs  x_{E_n}}$,  we have
\begin{align}\label{C:conv}
\big({\mathsf C}_1^{(n)},\P^{(n)}_{\bs  x_{E_n}}\big)\xrightarrow[n\to\infty]{(\rm d)}\big({\mathsf C}_1,
\P^{(\infty)}\big),
\end{align}
where the limiting object satisfies $\P^{(\infty)}(\mathsf C_1=\infty)<1$.
Then for some $\theta>0$, it holds that
\begin{align}
&\sup_{n\in \Bbb N}\sup_{\bs x\in E^{N_n}_n}\E_{\bs x}^{(n)}\big[\exp\big\{\theta \mathsf C_1^{(n)}\big\}\big]\label{C:bdd}
\end{align}
and 
\begin{align}
&\sup_{n\in \Bbb N}\sup_{\lambda\in \ms P(S^{E_n})}\E^{(n)}_{\lambda}\big[\exp\big\{\theta T^{(n)}\big\}\big]<\infty.\label{tau:bdd}
\end{align}
In particular, $\P^{(\infty)}(\mathsf C_1<\infty)=1$.
\end{lem}
\begin{proof}
Since $\P^{(\infty)}(\mathsf C_1=\infty)<1$, we can find $t_0\in(0,\infty)$ such that $\P^{(\infty)}(\mathsf C_1>t_0)=\P^{(\infty)}(\mathsf C_1\geq t_0)<1$. 
Then by (\ref{C:conv}), given $\vep>0$ such that 
$\delta= 1-\P^{(\infty)}(\mathsf C_1>t_0)-\vep>0$, 
we can find some large enough integer $N_0\geq 1$ such that 
\begin{align}\label{ineq:coal}
\sup_{\bs x\in E_n^{N_n}}\P_{\bs x}^{(n)}(\mathsf C_1^{(n)}> t_0)\leq \P_{\bs x_{E_n}}^{(n)}(\mathsf C^{(n)}_1> t_0)\leq \P^{(\infty)}(\mathsf C_1>t_0)+\vep,\quad \forall\;n\geq N_0.
\end{align}

Now the proof of \cite[Proposition~4.1]{Oliveira_2012} shows (\ref{C:bdd}). In detail, first note that the Markov property of coalescing Markov chains and (\ref{ineq:coal}) imply 
\begin{align}\label{ineq:coal1}
\sup_{\bs x\in E_n^{N_n}}\P_{\bs x}^{(n)}({\mathsf C}_1^{(n)}>k t_0)\leq (1-\delta)^k,\quad \forall\; k\geq 0.
\end{align}
Hence, for any $n\geq N_0$, $\bs x\in E^{N_n}_n$ and $\theta>0$ such that $e^{\theta t_0}(1-\delta)<1$,
\begin{align*}
\E^{(n)}_{\bs x}\big[\exp\big\{\theta\mathsf C^{(n)}_1\big\}\big]=
&\int_0^\infty \theta e^{\theta s}\P^{(n)}_{\bs x}(\mathsf C_1^{(n)}>s)ds+1\\
\leq &\sum_{k=0}^\infty \theta t_0e^{\theta (k+1)t_0}\P_{\bs x}^{(n)}(
\mathsf C^{(n)} >kt_0)+1
\leq \theta t_0e^{\theta t_0}\sum_{k=0}^\infty e^{\theta t_0k}(1-\delta)^k+1<\infty,
\end{align*}
where the first inequality follows from (\ref{ineq:coal1}) and the second inequality follows from the choice of $\theta$. The last inequality proves (\ref{C:bdd}).

The inequality (\ref{tau:bdd}) follows from (\ref{C:bdd}) and the stochastic dominance of $(T^{(n)},\P^{(n)}_\xi)$ by $\big(\mathsf C^{(n)}_1,\P^{(n)}_{\bs x_{E_n}}\big)$: for any $\xi\in S^{E_n}$ and $t\geq 0$, 
\begin{align*}
\P^{(n)}_\xi(T^{(n)}\leq t)=&\E^{(n)}_{\xi}\left[\prod_{x\in E_n}\xi_{\gamma_n t}(x)+\prod_{x\in E_n}\dxi_{\gamma_n t}(x)\right]\\
=&\E^{(n)}_{\bs x_{E_n}}\left[\prod_{x\in E_n}\xi(B^x_{\gamma_n t})+\prod_{x\in E_n}\dxi(B^x_{\gamma_n t})\right]\geq \P^{(n)}_{\bs x_{E_n}}(\mathsf C_1^{(n)}\leq t),
\end{align*}
where the second equality follows from (\ref{duality}) and the inequality follows since, when $\mathsf C_1^{(n)}\leq t$, the set $\big\{B^x_{\gamma_nt}; x\in E_n\big\}$ is singleton. 
The proof is complete. 
\end{proof}

The main result of this section is Theorem~\ref{thm:main3} below. We recall that the Wasserstein distance of order $1$ between two probability measures on $\R$ with finite first moments is defined by 
\[
\mathcal W_1(\mu,\nu)=\inf\left\{\int_{\R^2} |x-y|d\pi(x,y);\pi\in \ms P(\R^2), \pi(\,\cdot \,\times \R)=\mu,\pi(\R\times \,\cdot\,)=\nu\right\}.
\]
The proof of Theorem~\ref{thm:main3} uses a standard result of Wasserstein distances in \cite{Vallender_1974}, which gives an alternative characterization of $\mathcal W_1(\mu,\nu)$ as 
\begin{align}\label{vall}
\mathcal W_1(\mu,\nu)=\int_{\R}\big|\mu\big([x,\infty)\big)-\nu\big([x,\infty)\big)\big|dx.
\end{align}
We write $\xrightarrow[n\to\infty]{(\mathcal W_1)}$ for convergence with respect to the metric $\mathcal W_1$.

\begin{thm}\label{thm:main3}
Let the assumptions of {\rm Theorem~\ref{thm:main2-1}} $3^\circ)$ with $\mu_n\equiv 0$ be in force, and assume
\begin{align}\label{coalk}
\left(\mathsf C^{(n)}_\ell,\P^{(n)}_{\bs  x_{E_n}}\right)\xrightarrow[n\to\infty]{({\rm d})}\sum_{m=\ell+1}^\infty \frac{\mathbf e_m}{m(m-1)/2},\quad \forall\;\ell\geq 1,
\end{align}
where $\mathbf e_m$ are i.i.d. standard exponentials. Then there exists $\overline{w}_\infty>0$ such that for any $\{w_n\}$ satisfying (\ref{selection2}) with $w_\infty\in [0,\overline{w}_\infty]$, it holds that  
\begin{align}\label{Wconv:1}
\big(Y^{(n)},T^{(n)}\big)\;\mbox{under }\P^{(n),w_n}_{\lambda_n}&\xrightarrow[n\to\infty]{\rm (d)}(Y,\widetilde{T})\;\mbox{under }\P^{(\infty),w_\infty}_{\widetilde{\lambda}_\infty},\\
\P_{\lambda_n}^{(n),w_n}(T_{\mathbf 1}<T_{\mathbf 0})&\xrightarrow[n\to\infty]{}\P^{(\infty),w_\infty}_{\widetilde{\lambda}_\infty}(T_{\mathbf 1}<T_{\mathbf 0}),\label{Wconv:1-0}\\
\label{Wconv:2}
\ms L\left(\int_0^{T^{(n)}}f(Y^{(n)}_s)ds\right)\mbox{ under $\P^{(n),w_n}_{\lambda_n}$}&\xrightarrow[n\to\infty]{(\mathcal W_1)}\ms L\left(\int_0^{\widetilde{T}}f(Y_s)ds\right)\mbox{ under $\P^{(\infty),w_\infty}_{\widetilde{\lambda}_\infty}$}
\end{align}
for any nonnegative continuous function $f$ on $[0,1]$. Here, $\widetilde{T}$ is the time to absorption of $Y$.  
\end{thm}
\begin{proof}
Before proving the required convergences in (\ref{Wconv:1})--(\ref{Wconv:2}), we first claim that 
\begin{align}
\big(Y^{(n)},D^{(n)},T^{(n)}\big)\;\mbox{under }\P^{(n)}_{\lambda_n}\xrightarrow[n\to\infty]{\rm (d)}(Y,D,\widetilde{T})\;\mbox{under }\P^{(\infty)}.\label{YDT}
\end{align}
We already have the convergence in distribution of $T^{(n)}$ to $\widetilde{T}$ by (\ref{conv:WF}) with $\mu_n=0$, (\ref{coalk}), and \cite[Proposition~2.6]{CCC}. 
So we only need to show the joint convergence in (\ref{YDT}). 
By taking a subsequence if necessary and using Skorokhod's representation theorem (cf. \cite[Theorem~3.1.8]{EK:MP}), we may reinforce the convergence in (\ref{YDT}) to almost sure convergence, except that $T^{(n)}$ is only known a-priori to converge almost surely to a random variable $\widehat{T}$ with the same distribution as $\widetilde{T}$. Since $T^{(n)}$ (resp. $\widetilde{T}$) is a.s. equal to the time to absorption of $Y^{(n)}$ (resp. $Y$) and $Y^{(n)}$ converges to $Y$ a.s., $\widehat{T}=\lim_{n\to\infty}T^{(n)}\geq \widetilde{T}$ a.s. The fact that $\widetilde{T}\stackrel{({\rm d})}{=}\widehat{T}$ then implies $\widetilde{T}=\widehat{T}$ a.s., and we get
\begin{align}\label{YDT:as}
(Y^{(n)},D^{(n)},T^{(n)})\xrightarrow[n\to\infty]{{\rm a.s.}}(Y,D,\widetilde{T}).
\end{align} 
The claim in (\ref{YDT}) follows.

Now we prove (\ref{Wconv:1}) and may assume (\ref{YDT:as}). It follows from (\ref{cond:ui}) and Proposition~\ref{prop:expbdd} 2$^\circ$) that for some $\overline{w}_\infty>0$, the sequence $\big(D^{(n)}_{T^{(n)}},\P^{(n)}_{\lambda_n}\big)$ is uniformly integrable 
for any  $\{w_n\}$ satisfying (\ref{selection2}) with $w_\infty\leq \overline{w}_\infty$, where $w_\infty$ is defined by (\ref{selection2}). 
Therefore, for every bounded continuous function $F$ on $D(\R_+,[0,1])\times \R_+ $, we have
\begin{align*}
&\E^{(n),w_n}_{\lambda_n}\big[F\big(Y^{(n)},T^{(n)}\big)\big]=\E^{(n)}_{\lambda_n}\big[F(Y^{(n)},T^{(n)})D^{(n)}_{T^{(n)}}\big]\xrightarrow[n\to\infty]{} \E^{(\infty)}[F(Y,\widetilde{T})D_{\widetilde{T}}]=\E_{\widetilde{\lambda}_\infty}^{(\infty),w_\infty}[F(Y,\widetilde{T})],
\end{align*}
which is enough for (\ref{Wconv:1}).

For the proof of (\ref{Wconv:1-0}), notice that $Y^{(n)}_{T^{(n)}}$ under $\P^{(n),w_n}_{\lambda_n}$ converges in distribution to $ Y_{\widetilde{T}}$ under $\P^{(\infty),w_\infty}_{\widetilde{\lambda}_\infty}$ by (\ref{Wconv:1}) and \cite[Proposition~3.6.5]{EK:MP}. Since $Y^{(n)}_{T^{(n)}}$ take values  in $\{1,0\}$, $\big\{T^{(n)}_{\mathbf 1}<T^{(n)}_{\mathbf 0}\big\}=\big\{Y^{(n)}_{T^{(n)}}=1\big\}$ and a similar equality holds under $Y$, the convergence in (\ref{Wconv:1-0}) follows.

For the proof of (\ref{Wconv:2}), notice that by (\ref{intmap}), (\ref{Wconv:1}) and \cite[Proposition~3.6.5]{EK:MP},  
\begin{align*}
\left(\int_0^{T^{(n)}}f\big(Y^{(n)}_s\big)ds,\P^{(n),w_n}_{\lambda_n}\right)\xrightarrow[n\to\infty]{({\rm d})} \left(\int_0^{\widetilde{T}}f\big(Y_s\big)ds,\P^{(\infty),w_\infty}_{\widetilde{\lambda}_\infty}\right).
\end{align*}
Hence, to verify the required convergence in the Wasserstein distance by means of (\ref{vall}), it is enough to prove uniformly exponential tails of the distributions of
$(T^{(n)},\P^{(n),w_n}_{\lambda_n})$.  
Notice that the conclusions of Lemma~\ref{lem:coal} apply by (\ref{coalk}) with $\ell=1$. 
Therefore, 
\begin{align*}
\sup_{n\in \Bbb N}\E^{(n),w_n}_{\lambda_n }\big[\exp\big\{\theta T^{(n)}\big\}\big]=&\sup_{n\in \Bbb N}\E^{(n)}_{\lambda_n}\big[\exp\big\{\theta T^{(n)}\big\}D^{(n)}_{T^{(n)}}\big]\\
\leq &\sup_{n\in \Bbb N}\E^{(n)}_{\lambda_n}\big[\exp\big\{2\theta T^{(n)}\big\}\big]^{1/2}
\E^{(n)}_{\lambda_n}\big[\big(D^{(n)}_{T^{(n)}}\big)^2\big]^{1/2}<\infty,
\end{align*}
where the last inequality follows from from Lemma~\ref{lem:coal}, (\ref{cond:ui}) and Proposition~\ref{prop:expbdd} $2^\circ)$, if $\theta>0$ is small enough and $w_\infty\leq \overline{w}_\infty$ by lowering the constant $\overline{w}_\infty$ chosen above if necessary. The foregoing inequality is enough for the proof of (\ref{Wconv:2}). The proof of the theorem is complete. 
\end{proof}

The following proposition proves the diffusion approximation of absorbing probabilities in \cite[SI]{Ohtsuki_2006}.

\begin{cor}\label{cor:O}
For any fixed $k\geq 3$, the conclusions of {\rm Theorem~\ref{thm:main3}}
apply to any sequence of random $k$-regular graphs on $N_n$ vertices with $N_n\nearrow\infty$. 
\end{cor}
\begin{proof}
Since $q^{(n)}(x,y)\leq 1/k$ and $\pi^{(n)}(x)\equiv 1/N_n$,
the proofs of \cite[Theorem~1.1 and Theorem~1.2]{Oliveira_2013} show that (\ref{coalk}) holds if 
\begin{align}
\mathbf t^{(n)}_{\rm min}/N_n\xrightarrow[n\to\infty]{} 0+\label{Oliveira}
\end{align}
(see \cite[Lemma~5.1 and Section~6]{Oliveira_2013} in particular). Then to verify that condition in Theorem~\ref{thm:WF}, we recall a standard result of Markov chains:
\begin{align}\label{tmix}
\mathbf t^{(n)}_{\rm mix}\leq \mathbf (\mathbf g^\star_n)^{-1}\log (2e/\pi^{(n)}_{\min}\big)=\mathbf (\mathbf g^\star_n)^{-1}\log (2eN_n)
\end{align}
(cf. \cite[Theorem~12.3]{LWP}). Here, $\mathbf g^\star_n$ is the absolute spectral gap of $(E_n,q^{(n)})$, and is given by the distance between $1$ and the maximal absolute values of eigenvalues of $(E_n,q^{(n)})$ excluding the largest one. The fact that $\mathbf g_n^\star$ are bounded away from zero is also contained in the main results of \cite{Friedman_2008, Bordenave_2015}, and can be applied to (\ref{tmix}) to
validate (\ref{Oliveira}). 
\end{proof}

\section{Expansions of the game absorbing probabilities in selection strength}\label{sec:expansion}
As in the previous section, we focus on the context where mutation is absent. We use payoff matrices with general entries throughout this section unless otherwise.

With respect to a voting kernel $(E,q)$ and $\lambda\in \ms P(S^E)$ such that 
\begin{align}\label{zero-der}
\partial_w \P^w_\lambda(T_{\mathbf 1}<T_{\mathbf 0})\big|_{w=0}\neq 0, 
\end{align}
we define $w^\star(\lambda;E,q)$ to be the supremum of $w''\in [0,\overline{w}]$ such that
\[
\sgn\big(\partial_w \P^w_\lambda(T_{\mathbf 1}<T_{\mathbf 0})\big|_{w=0}\big)=\sgn\big(\P^{w'}_\lambda(T_{\mathbf 1}<T_{\mathbf 0})-\P^0_\lambda(T_{\mathbf 1}<T_{\mathbf 0})\big),\quad \forall\; w'\in (0,w''],
\]
or
\[
\sgn\big(\partial_w \P^w_\lambda(T_{\mathbf 1}<T_{\mathbf 0})\big|_{w=0}\big)=-\sgn\big(\P^{w'}_\lambda(T_{\mathbf 1}<T_{\mathbf 0})-\P^0_\lambda(T_{\mathbf 1}<T_{\mathbf 0})\big), \quad \forall\;  w'\in (0,w''],
\]
where we set $\sgn(0)=0$. In other words, the interval $(0,w^\star(\lambda;E,q)]$ gives a maximal range of selection strengths $w'$ such that the first-order derivative in (\ref{zero-der}) has the same sign of 
$\P^{w'}_\lambda(T_{\mathbf 1}<T_{\mathbf 0})-\P^0_\lambda(T_{\mathbf 1}<T_{\mathbf 0})$. Our goal in this section is to estimate the order of $w^\star(\lambda;E,q)$ relative to $N$. To this end, we will study the remainders in the first-order Taylor expansions of $w\mapsto \P^w_\lambda(T_{\mathbf 1}<T_{\mathbf 0})$.

Let us introduce some notation for the use of coalescing Markov chains. We write $\ms P$ for the set of functions $F$ on $S^E$ taking the form 
\begin{align}\label{Frep}
F(\xi)=\sum_{(A_1,A_2)\in \ms A} C(A_1,A_2)\prod_{x\in A_1}\xi(x)\prod_{x\in A_2}\dxi(x),
\end{align}
where $(A_1,A_2)$ are pairs of disjoint nonempty subsets of $E$ and $C(A_1,A_2)$ are constants. Notice that $F$ is a polynomial in $\xi(x)$ for $x\in E$ and satisfies $F(\mathbf 1)=F(\mathbf 0)=0$.

Recall the auxiliary discrete-time Markov chains $(X_\ell)$ and $(Y_\ell)$ defined at the beginning of Section~\ref{sec:tight}.
The following lemma follows from a plain generalization of the argument for (\ref{Dbdd1}) (see (\ref{taylor})), and so its proof is omitted.

\begin{lem}\label{lem:defQ}
For all $m\geq 1$, $w\in [0,\overline{w}]$ and $\xi\in S^E$,
\begin{align}\label{def:Q}
\sum_{x,y\in E}\pi(x)\left|q^w(x,y,\xi)-\sum_{j=0}^{m-1}\left(\partial^j_wq^w(x,y,\xi)\big|_{w=0}\right)w^j\right|\leq w^mQ_m(\xi),
\end{align}
where the function $Q_m(\xi)\in \ms P_+$ is given by
\begin{align}\label{Qj}
Q_m(\xi)=C_{\ref{Qj}}(m)\sum_{\ell=1}^4W_\ell(\xi)
\end{align}
for some constant $C_{\ref{Qj}}(m)> 0$ depending only on $m$ and $\Pi$. 
\end{lem}

 For any nonempty set $A\subseteq E$, we write $B^A$ for the subsystem $\{B^x;x\in A\}$ of coalescing $q$-Markov chains. We also write $M_{A_1,A_2}$ for the first meeting time of the two subsystems $B^{A_1}$ and $B^{A_2}$, or more precisely, the first time $t$ when $B^x_t=B^y_t$ for some $x\in A_1$ and $y\in A_2$. If we follow the usual alternative viewpoint that the processes $B^{A_1}$ and $B^{A_2}$ are set-valued processes, then $M_{A_1,A_2}$ is the first time that the two processes `intersect'.

\begin{lem}\label{lem:QF}
For all $w\in [0,\overline{w}]$, $\xi\in S^E$, and $F\in \ms P$ taking the form (\ref{Frep}), it holds that
\begin{align}\label{FQ}
\begin{split}
&\left|\E^0_\xi\left[\int_0^\infty D^w_sF(\xi_s)ds \right]- \E_\xi^0\left[\int_0^\infty F(\xi_s)ds\right]\right|\\
&\hspace{4cm}\leq \frac{wC_{\ref{Qj}}(1)\cdot C_{\ref{def:CF}}(F)}{\pi_{\min}}  \E_\xi^0\left[\int_0^\infty D^w_sQ_1(\xi_s)ds\right],
\end{split}
\end{align}
where $Q_1$ is chosen in {\rm Lemma~\ref{lem:defQ}} and the constant $C_{\ref{def:CF}}(F)$ is defined with respect to (\ref{Frep}) by
\begin{align}\label{def:CF}
\begin{split}
C_{\ref{def:CF}}(F)= &\max_{x\in E}\sum_{(A_1,A_2)\in \ms A}|C(A_1,A_2)|\int_0^\infty \P\big(x\in B^{A_1}_t\cup B^{A_2}_t,M_{A_1,A_2}>t\big)dt.
\end{split}
\end{align}
In particular, if we choose $F=Q_1$ and selection strength $w$ satisfying
\begin{align}\label{w:constraint}
0\leq w\leq  \min\left\{\overline{w},\frac{\pi_{\min}}{2C_{\ref{Qj}}(1)\cdot C_{\ref{def:CF}}(Q_1)}\right\}, 
\end{align}
then  (\ref{FQ}) gives
\begin{align}\label{FQ1}
\E^0_\xi\left[\int_0^\infty D^w_sQ_1(\xi_s)ds\right]\leq \frac{1}{1-wC_{\ref{Qj}}(1)\cdot C_{\ref{def:CF}}(Q_1)\pi_{\min}^{-1}}\E^0_\xi\left[\int_0^\infty Q_1(\xi_s)ds\right].
\end{align}
\end{lem}
\begin{proof}
By the linear equation (\ref{D:si}) satisfied by $D^w$, it holds that
\begin{align}
\begin{split}
&\int_0^\infty D^w_sF(\xi_s)ds
=\int_0^\infty F(\xi_s)ds\\
&\hspace{2cm}+\int_0^\infty \sum_{x,y\in E}\int_0^sD^w_{r-}\left(\frac{q^w(x,y,\xi_{r-})}{q(x,y)}-1\right) d\Lambda_r(x,y) F(\xi_s)ds.\label{DF}
\end{split}
\end{align}
Notice that
the left-hand side has a finite $\P^0_\xi$ expectation since the time to absorption under $\E^w_\xi$ is integrable, and so does the first term on the right-hand side for a similar reason. 
The $\P^0_\xi$-expectation of the second term on the right-hand side of (\ref{DF}) satisfies
\begin{align}
&\E^0_\xi\left[\int_0^\infty\sum_{x,y\in E}\int_0^sD^w_{r-}\left(\frac{q^w(x,y,\xi_{r-})}{q(x,y)}-1\right) d\Lambda_r(x,y) F(\xi_s)ds\right]\notag\\
=&\sum_{x,y\in E}\E^0_\xi\left[\int_0^\infty D^w_{r-}\left(\frac{q^w(x,y,\xi_{r-})}{q(x,y)}-1\right) \int_r^\infty F(\xi_s)dsd\Lambda_r(x,y)\right]\notag\\
=&\sum_{x,y\in E}\E^0_\xi\left[\int_0^\infty D^w_{r-}\left(\frac{q^w(x,y,\xi_{r-})}{q(x,y)}-1\right) \E^0_{\xi_r}\left[\int_0^\infty F(\xi_s)ds\right]d\Lambda_r(x,y)\right]\notag\\
=&\sum_{x,y\in E}\E^0_\xi\left[\int_0^\infty D^w_{r-}\left(\frac{q^w(x,y,\xi_{r-})}{q(x,y)}-1\right) \E^0_{\xi_{r-}}\left[\int_0^\infty F(\xi_s)ds\right]d\Lambda_r(x,y)\right]\notag\\
&+\sum_{x,y\in E}\E_\xi^0\left[\int_0^\infty D^w_{r-}\left(\frac{q^w(x,y,\xi_{r-})}{q(x,y)}-1\right) \right.\notag\\
&\times\left.\left(\E^0_{\xi_{r}}\left[\int_0^\infty F(\xi_s)ds\right]-\E^0_{\xi_{r-}}\left[\int_0^\infty F(\xi_s)ds\right]\right)d\Lambda_r(x,y)\right]\notag\\
\begin{split}\label{Dwp2}
=&\sum_{x,y\in E}\E^0_\xi\left[\int_0^\infty D^w_{r-}\left(\frac{q^w(x,y,\xi_{r-})}{q(x,y)}-1\right) \right.\\
&\times\left.\left(\E^0_{\xi_{r}}\left[\int_0^\infty F(\xi_s)ds\right]-\E^0_{\xi_{r-}}\left[\int_0^\infty F(\xi_s)ds\right]\right)d\Lambda_r(x,y)\right],
\end{split}
\end{align}
where the second equality follows from the $(\F_t)$-strong Markov property of $(\xi_t)$ and the last equality follows from Poisson calculus and the fact that 
\begin{align}\label{equil}
1=\sum_{y\in E}q^w(x,y,\xi)= \sum_{y\in E}q(x,y),\quad \forall\;x\in E,\;\xi\in S^E.
\end{align}

We study the right-hand side of (\ref{Dwp2}). Since $F$ is a polynomial in $\xi(x)$ for $x\in E$, Equation~(\ref{duality}) applies to the evaluation of $\E^0_\xi\left[\int_0^\infty F(\xi_s)ds\right]$ by  $\{B^x\}$. We also observe that
\[
\sup_{\xi\in S^E}\left|\prod_{y\in A_1}\widehat{\xi^x}(B^y_t)\prod_{y\in A_2}\xi^x(B^y_t)-\prod_{y\in A_1}\dxi(B^y_t)\prod_{y\in A_2}\xi(B^y_t)\right|\leq \1_{(t,\infty)}(M_{A_1,A_2})\1_{B^{A_1}_t\cup B^{A_2}_t}(x),\quad \forall \;x\in E.
\]
Hence, we see that, for all $x\in E$ and $\xi\in S^E$,
\begin{align}\label{MAA}
\begin{split}
&\left|\E^0_{\xi^x}\left[\int_0^\infty F(\xi_s)ds\right]-\E^0_{\xi}\left[\int_0^\infty F(\xi_s)ds\right]\right|\\
\leq &\sum_{(A_1,A_2)\in \ms A}C(A_1,A_2)\int_0^\infty \P\big(x\in B^{A_1}_t\cup B^{A_2}_t,M_{A_1,A_2}>t\big)dt=C_{\ref{def:CF}}(F)
\end{split}
\end{align}
by the definition (\ref{def:CF}) of $C_{\ref{def:CF}}(F)$. 
Now we apply the foregoing inequality to the right-hand side of (\ref{Dwp2}) and use Poisson calculus and the inequality (\ref{def:Q}) with $m=1$. Then by (\ref{Dwp2}) and (\ref{MAA}), we get
\begin{align*}
&\left|\E^0_\xi\left[\int_0^\infty\sum_{x,y\in E}\int_0^sD^w_{r-}\left(\frac{q^w(x,y,\xi_{r-})}{q(x,y)}-1\right) d\Lambda_r(x,y) F(\xi_s)ds\right]\right|\\
&\hspace{6cm}\leq \frac{wC_{\ref{Qj}}(1)\cdot C_{\ref{def:CF}}(F)}{\pi_{\min}}  \E^0_\xi\left[\int_0^\infty D^w_sQ_1(\xi_s)ds\right].
\end{align*}
The required inequality (\ref{FQ}) follows from the last inequality and (\ref{DF}). 
The proof is complete. 
\end{proof}

The next result recovers the first-order expansion of the game absorbing probabilities in \cite{Chen_2013}. 
Notice that the function $\overline{D}$ defined by (\ref{def:Dbar1}) is in $ \ms P$ and so the above two lemmas are applicable.

\begin{prop}\label{prop:selection}
Fix a choice of functions $Q_1$ and $Q_2$ defined by (\ref{Qj}). 
For any $w\in [0,\overline{w}]$ and $\xi\in S^E$, it holds that 
\begin{align}
\begin{split}
&\left|\P^w_\xi(T_{\mathbf 1}<T_{\mathbf 0})-\P^0_\xi(T_{\mathbf 1}<T_{\mathbf 0})-w\E^0_\xi\left[\int_0^\infty \overline{D}(\xi_s)ds\right]\right|\\
&\hspace{.5cm}\leq \frac{w^2C_{\ref{Qj}}(1)\cdot C_{\ref{def:CF}}(\overline{D})}{\pi_{\min}}\E^0_\xi\left[\int_0^\infty D^w_sQ_1(\xi_s)ds\right]+w^2C_{\ref{Qj}}(2)\E^0_\xi\left[\int_0^\infty D^w_sQ_2(\xi_s)ds\right],\label{Q1Q2}
\end{split}
\end{align} 
where $\overline{D}$ is defined by (\ref{def:Dbar1}). 
\end{prop}
\begin{proof}
By the definition of $\P^w$ in (\ref{def:Pw}) and the trivial fact that $p_1(\mathbf 1)=1$ and $p_1(\mathbf 0)=0$, it holds that for all $\xi\in S^E$,
\begin{align}
\P^w_\xi(T_{\mathbf 1}<T_{\mathbf 0})=&\lim_{t\to\infty}\P^w_\xi[p_1(\xi_t)]=\lim_{t\to\infty}\E^0_\xi\left[D_t^wp_1(\xi_t)\right]\notag\\
=&\lim_{t\to\infty}\big(p_1(\xi)+\E^0_\xi[\langle D^w,p_1(\xi_\cdot)\rangle_t]\big)
=p_1(\xi)+\E^0_\xi[\langle D^w,p_1(\xi_\cdot)\rangle_\infty],
\label{fixeq1}
\end{align}
where the third equality follows from integration by parts since $D^w$ and $p_1(\xi_\cdot)$ are both $\P^0_\xi$-martingales \cite[Theorem I.4.2]{JS}, and the last equality follows from dominated convergence by (\ref{MD1}) since the time to absorption is integrable under $\P^w_\xi$.

Now we expand the last term in (\ref{fixeq1}) in $w$. 
It follows from (\ref{Dp}) and (\ref{def:Dbar1}) that
\begin{align}
\langle D^w,p_1(\xi_\cdot)\rangle_\infty
=&w\int_0^\infty D^w_s\overline{D}(\xi_s)ds
+w^2\int_0^\infty D^w_sR^{w}_1(\xi_s)ds.\label{Dp1}
\end{align}
By Lemma~\ref{lem:QF}, the first term on the right-hand side of (\ref{Dp1}) satisfies
\[
\left|w\E^0_\xi\left[\int_0^\infty D^w_s\overline{D}(\xi_s)ds\right]-w\E^0_\xi\left[\int_0^\infty \overline{D}(\xi_s)ds\right]\right|\leq \frac{w^2C_{\ref{Qj}}(1)\cdot C_{\ref{def:CF}}(\overline{D})}{\pi_{\min}}\E^0_\xi\left[\int_0^\infty D^w_sQ_1(\xi_s)ds\right].
\]
By the definition of $R^w_1$ in (\ref{def:R1w}) (see also (\ref{eq:qwexp0})) and the choice of $Q_2$ according to (\ref{def:Q}), we also have
\[
\left|w^2\E^0_\xi\left[\int_0^\infty D^w_sR^w_1(\xi_s)ds\right]\right|\leq w^2C_{\ref{Qj}}(2)\E^0_\xi\left[\int_0^\infty D^w_s Q_2(\xi_s)ds\right].
\]
Applying the last three displays to the last term in (\ref{fixeq1}), we deduce the required inequality (\ref{Q1Q2}). The proof is complete.  
\end{proof}

We use the following lemma to estimate the constant $C_{\ref{def:CF}}(\overline{D})$. 

\begin{lem}\label{lem:twopoint}
For any $\ell\geq 1$,
\begin{align}
\int_0^\infty \P_\pi(B^{X_0}_s=x,B^{X_\ell}_s=y)ds\leq C_{\ref{BX}}\left(\frac{\pi_{\max}}{\pi_{\min}}\right)\frac{\pi(x)\pi(y)}{\nu(\1)},\quad \forall\;x\neq y,\label{BX}
\end{align}
for some constant $C_{\ref{BX}}$ depending only on $\ell$. In particular, we have
\begin{align}\label{BX1}
C_{\ref{def:CF}}(W_\ell)\leq 2C_{\ref{BX}}\left(\frac{\pi_{\max}}{\pi_{\min}}\right)^2.
\end{align}
\end{lem}
\begin{proof}
The required inequality (\ref{BX}) for $\ell=1$ is a particular consequence of Kac's formula (cf. \cite[Section~2.5.1]{AF_MC}), but below we give an alternative  proof of (\ref{BX}) with $\ell=1$ by voter model calculations.  
To see the proof of (\ref{BX}) for $\ell\geq 2$, we notice the following general result. If $F$ is a nonnegative function on $E\times E$ which vanishes on the diagonal, then (\ref{F_recur}) implies that for all $\ell\geq 1$,
\begin{align*}
&\int_0^\infty \E[F(B^{X_0}_t,B^{X_\ell}_t)]dt
=\frac{\E[F(X_0,X_\ell)]}{2}+\frac{1}{2}\int_0^\infty \E[F(B^{X_0}_t,B^{X_{\ell+1}}_t)\1_{\{X_0\neq X_\ell\}}]dt\\
&+\frac{1}{2}\int_0^\infty \E[F(B^{X_{\ell+1}}_t,B^{X_{0}}_t)\1_{\{X_0\neq X_\ell\}}]dt\\
=&\frac{\E[F(X_0,X_\ell)]}{2}+\int_0^\infty \E[F(B^{X_0}_t,B^{X_{\ell+1}}_t)]dt
-\int_0^\infty \E[F(B^{X_\ell}_t,B^{X_{\ell+1}}_t)\1_{\{X_0=X_\ell\}}]dt,
\end{align*}
and so
\begin{align*}
&\int_0^\infty \E[F(B^{X_0}_t,B^{X_\ell}_t)]dt
+\int_0^\infty \E[F(B^{X_0}_t,B^{X_1}_t)]dt\geq \int_0^\infty \E[F(B^{X_0}_t,B^{X_{\ell+1}}_t)]dt.
\end{align*}
Thus (\ref{BX}) for general $\ell\geq 2$ follows from the above inequality and the validity of (\ref{BX}) for $\ell=1$.  The inequality (\ref{BX1}) then follows by writing out $C_{\ref{def:CF}}(W_\ell)$:
\begin{align*}
C_{\ref{def:CF}}(W_\ell)=&\max_{x\in E}\sum_{u,v\in E}\pi(u)q^\ell(u,v)\int_0^\infty \P(x\in B^{\{u,v\}}_t,M_{u,v}>t)dt
\end{align*}
and using (\ref{BX}) and the fact that $\nu(\1)\geq \pi_{\min}$.

Now we give a proof of (\ref{BX}) with $\ell=1$ by voter model calculations.
It follows from (\ref{MM}) and (\ref{xixy})  that  for all $\xi\in S^E$,
\begin{align}\label{local-den}
\begin{split}
\E_\xi^0[p_1(\xi_t)p_0(\xi_t)]=&p_1(\xi)p_0(\xi)-\sum_{x,y\in E}\pi(x)^2q(x,y)\int_0^t \E^0_\xi[\xi_s(x)\dxi_s(y)+\dxi_s(x)\xi_s(y)]ds
\end{split}
\end{align}
(see also \cite[Theorem 3.1]{CCC}). Passing $t\to\infty$ for both sides of the foregoing equality, we get
\[
\sum_{x,y\in E}\pi(x)^2q(x,y)\int_0^\infty  \E^0_\xi[\xi_s(x)\dxi_s(y)+\dxi_s(x)\xi_s(y)]ds
=\frac{p_1p_0(\xi)}{2},\quad \forall \;\xi\in S^E.
\]
By the duality equation in (\ref{duality}), both sides of the foregoing equality take the form $\xi A\dxi$ of a matrix product for a symmetric $N\times N$-matrix $A$ with zero diagonal entries. Hence, to prove (\ref{BX}) for $\ell=1$, it suffices to show that, for a symmetric $N\times N$ matrix $A$ with zero diagonal,
\begin{align}\label{eq:vA}
\xi^\top  A\dxi=0\quad \;\forall \;\xi\in S^E \Longrightarrow A=0.
\end{align}

The following proof of (\ref{eq:vA}) is due to Rani Hod~\cite{Rani}. Let $\{e_x;x\in E\}$ be the standard basis of $S^E$.
Taking $\xi=e_x$ for $x\in E$ in (\ref{eq:vA}), we obtain that
$e_x^\top  A\big(\sum_{y:y\neq x}e_y\big)=0$
and so 
$e^\top_x A\1=0$
by the assumption that $A$ has zero diagonal entries. 
Hence, taking $\xi=e_x+e_y$ for $x\neq y$ in (\ref{eq:vA}), we obtain from (\ref{eq:vA}) that
\[
0=-(e_x+e_y)^\top A \left(\1-e_x-e_y\right) =A_{x,x}+A_{x,y}+A_{y,x}+A_{y,y}=A_{x,y}+A_{y,x}=2A_{x,y}
\]
since $A$ is symmetric. The last equality proves that $A=0$. This completes the proof.
\end{proof}

\begin{eg}\label{eg:order}
We show by an example that Lemma~\ref{lem:twopoint} gives a sharp estimate of $\E^0_\xi[\int_0^\infty \overline{D}(\xi_s)ds]$ in terms of its order relative to the population size $N$. Recall that $\E^0_\xi[\int_0^\infty \overline{D}(\xi_s)ds]$ is the first-order coefficient in the expansion (\ref{Q1Q2}) of the game absorbing probabilities. 

Let $(E,q)$ be the random walk transition probability on a finite, simple, connected, $k$-regular graph with $N$ vertices, and let $\mathbf u_m$ be the uniform probability measure on the set of $S^E$-valued configurations with exactly $m$ many $1$'s. Assume that $\Pi$ is given by the special payoff matrix (\ref{eq:payoff0}). By Lemma~\ref{lem:twopoint}, the moment duality equation (\ref{duality}) and the random-walk representation of $\overline{D}$ in (\ref{eq:Dbar}), we deduce that, for any $1\leq m\leq N-1$, 
\begin{align}\label{bcpayoff}
 \left|\E^0_{\mathbf u_m}\left[\int_0^\infty \overline{D}(\xi_s)ds\right]\right|\leq C_{\ref{bcpayoff}}
\frac{\pi_{\max}}{\pi_{\min}\nu(\1)}\sum_{x,y\in E}\mathbf u_m[\xi(x)\dxi(y)]\pi(x)\pi(y),
\end{align}
where the constant $C_{\ref{bcpayoff}}$ depends only on $\Pi$. In the present case of random walks on regular graphs, the above inequality simplifies to
\begin{align}
 \left|\E^0_{\mathbf u_m}\left[\int_0^\infty \overline{D}(\xi_s)ds\right]\right|\leq 
C_{\ref{bcpayoff}}\sum_{x,y\in E}\frac{\mathbf u_m[\xi(x)\dxi(y)] }{N}=C_{\ref{bcpayoff}}\frac{m(N-m)}{N}\label{bcpayoff1}
\end{align}
(cf. \cite[Eq. (68)]{Chen_2015} for the last equality). On the other hand, it has been proven that 
\begin{align}\label{bcpayoff0}
\E_{\mathbf u_m}^0\left[\int_0^\infty \overline{D}(\xi_s)ds\right]=\frac{m(N-m)}{2N(N-1)}[b(N-2k)-ck(N-2)]
\end{align}
(cf. \cite[Theorem~1]{Chen_2013} or \cite[Proposition~10]{Chen_2015}).

From (\ref{bcpayoff1}) and (\ref{bcpayoff0}), we see that Lemma~\ref{lem:twopoint} gives a sharp estimate of $\big|\E_{\mathbf u_m}^0[\int_0^\infty \overline{D}(\xi_s)ds]\big|$ relative to the population size $N$.   
\qed 
\end{eg}

The main result of Section~\ref{sec:expansion} is the following theorem. See Example~\ref{eg:order} for the choice of the denominators in one of its conditions, (\ref{Dlbdd}), and recall the notation $w^\star(\lambda;E,q)$ defined at the beginning of Section~\ref{sec:expansion}.

\begin{thm}
\label{thm:main1}
Let a sequence of voting kernels $\{(E_n,q^{(n)})\}$ and $\lambda_n\in \ms P(S^{E_n})$ be given such that (\ref{diag}) holds and 
\begin{align}
\liminf_{n\to\infty}\left.\left|\E^{(n)}_{\lambda_n}\left[\int_0^\infty \overline{D}(\xi_s)ds\right]\right|\right/\left(N_n\sum_{x,y\in E_n}\lambda_n[\xi(x)\dxi(y)]\pi^{(n)}(x)\pi^{(n)}(y)\right)>0.\label{Dlbdd}
\end{align}
Then for some positive constant $C_{\ref{main1}}$ depending only on $\Pi$, $\displaystyle \limsup  \pi^{(n)}_{\max}/\pi^{(n)}_{\min}$, and the above limit infimum, it holds that  
\begin{align}\label{main1}
\liminf_{n\to\infty}
N_nw^\star\big(\lambda_n;E_n,q^{(n)}\big)
\geq C_{\ref{main1}}>0.
\end{align}
\end{thm}

\begin{proof}
It follows from (\ref{def:Dbar1}) and a straightforward generalization of (\ref{eq:Dbar}) under a general payoff matrix that
$\overline{D}^{(n)}(\xi)$ can be dominated by a linear combination of the three functions 
$\E^{(n)}_\pi[\xi(X_0)\dxi(X_\ell)]=W_\ell^{(n)}(\xi)$, $1\leq \ell\leq 3$, with the coefficients being positive and depending only on $\Pi$. 
So if we apply (\ref{BX1}) to $\overline{D}^{(n)}$ and $Q_1^{(n)}$ (defined by (\ref{Qj})), it holds that
\begin{align}\label{Dbarineq1}
\sup_{n\in \Bbb N}\max\{C_{\ref{def:CF}}(\overline{D}^{(n)}),C_{\ref{def:CF}}(Q_1^{(n)})\}\leq C_{\ref{Dbarineq1}}
\end{align}
by Assumption~\ref{ass:stat}.

The foregoing inequality allows us to estimate the right-hand side of (\ref{Q1Q2}) with respect to the $n$-th model as follows. We apply (\ref{FQ}), (\ref{FQ1}) and (\ref{Dbarineq1}) to get the following:
\begin{align}
&\frac{w^2C_{\ref{Qj}}(1)\cdot C_{\ref{def:CF}}(\overline{D}^{(n)})}{\pi^{(n)}_{\min}}\E_{\lambda_n}^{(n)}\left[\int_0^\infty D^{w}_sQ_1^{(n)}(\xi_s)ds\right]\notag\\
&\hspace{4cm}+w^2C_{\ref{Qj}}(2) \E^{(n)}_{\lambda_n}\left[\int_0^\infty D^{w}_sQ_2^{(n)}(\xi_s)ds\right]\notag\\
\leq & C_{\ref{Dbarineq2}}N_nw^2\left(\E_{\lambda_n}^{(n)}\left[\int_0^\infty Q_1^{(n)}(\xi_s)ds\right]+\E_{\lambda_n}^{(n)}\left[\int_0^\infty Q_2^{(n)}(\xi_s)ds\right]\right)\label{Dbarineq2}\\
\leq &C_{\ref{Dbarineq3}}N_n w^2\left(N_n\sum_{x,y\in E_n}\lambda_n[\xi(x)\dxi(y)]\pi^{(n)}(x)\pi^{(n)}(y)\right),\quad \forall\; w\in [0, N_n^{-1}w_{\ref{Dbarineq3}}].\label{Dbarineq3}
\end{align}
In more detail, the constant $w_{\ref{Dbarineq3}}$ above is chosen to meet the constraints (\ref{w:constraint}) for all $n$ to validate (\ref{Dbarineq2}) by (\ref{FQ1}); it can be chosen to be bounded away from zero by (\ref{Dbarineq1}).
Also, (\ref{Dbarineq3}) follows from (\ref{Qj}), (\ref{BX}) and Assumption~\ref{ass:stat}.
 By decreasing $w_{\ref{Dbarineq3}}>0$ according to $C_{\ref{Dbarineq3}}$ and the limit infimum in (\ref{Dlbdd}) if necessary, we  deduce (\ref{main1}) from (\ref{Q1Q2}), (\ref{Dlbdd}) and (\ref{Dbarineq3}). The proof is complete.
\end{proof}

\section{Expansions of some covariation processes}\label{sec:flip}
In this section, we show some calculations to simplify the first-order coefficients in the expansions of $\langle M,D^w\rangle$ and $ \langle D^w,D^w\rangle$ in (\ref{Dp}) and (\ref{DD}).  
Recall the discrete-time $q$-Markov chains $(X_\ell)$ and $(Y_\ell)$ defined  at the beginning of Section~\ref{sec:tight}. We write $\E_x$ for the expectation under which the common starting point of $(X_\ell)$ and $(Y_\ell)$ is $x$.

\begin{lem}\label{lem:payoff}
If the payoff matrix $\Pi$ is given by (\ref{eq:payoff0}), then the function $\overline{D}(\xi)$ defined by (\ref{def:Dbar1}), which enters the first-order coefficient of $\langle M,D^w\rangle$ in (\ref{Dp}), can be written as
\begin{align}
\overline{D}(\xi)
=&b\big(\E_\pi[\xi(X_0)\dxi(X_3)]-\E_\pi[\xi(X_0)\dxi(X_1)]\big)
-c\E_\pi[\xi(X_0)\dxi(X_2)].\label{eq:Dbar}
\end{align}
For the integrand in the first-order expansion of $\langle D^w,D^w\rangle$ in (\ref{DD}), we have
\begin{align}
&\sum_{x,y\in E}\pi(x)q(x,y)[A(x,\xi)-B(y,\xi)]^2\notag\\
\begin{split}
=&b^2\E_\pi[\xi(X_0)\dxi(X_4)-\xi(X_0)\dxi(X_2)]-2bc\E_\pi[\xi(X_0)\dxi(X_3)-\xi(X_0)\dxi(X_1)]\\
&+c^2\E_\pi[\xi(X_0)\dxi(X_2)].\label{eq:Dbar2}
\end{split}
\end{align}
Here, the functions $A(x,\xi)$ and $B(y,\xi)$ are defined by (\ref{def:A}) and (\ref{def:B}), respectively. 
\end{lem}
\begin{proof}
The proof of (\ref{eq:Dbar}) is almost identical to the proof of \cite[Theorem~1 (1)]{Chen_2013}. We include its short proof here for the convenience of the reader. Recall the definitions of $A$ and $B$ in (\ref{def:A}) and (\ref{def:B}). Since $\Pi\big(\xi(x),\xi(y)\big)=b\xi(y)-c\xi(x)$, we have 
\begin{align*}
A(x,\xi)=&1-\sum_{z\in E}\sum_{z'\in E}q(x,z)q(z,z')
\big(b\xi(z')-c\xi(z)\big)=1-\E_x[b\xi(X_2)-c\xi(X_1)],\\
B(y,\xi)=&1-\sum_{z\in E}q(y,z)\big(b\xi(z)-c\xi(y)\big)=1-\E_y[b\xi(X_1)-c\xi(X_0)].
\end{align*}
The foregoing equations give
\begin{align*}
&\sum_{x,y\in E}\pi(x)q(x,y)[\xi(y)-\xi(x)][A(x,\xi)-B(y,\xi)]\\
=&\sum_{x,y\in E}\pi(x)q(x,y)[\xi(y)-\xi(x)]\big(-\E_x[b\xi(X_2)-c\xi(X_1)]+\E_y[b\xi(X_1)-c\xi(X_0)]\big)\\
=&\sum_{x,y\in E}\pi(x)q(x,y)[\xi(y)-\xi(x)]\big(b\E_y[\xi(X_2)+\xi(X_1)]-c\E_y[\xi(X_1)+\xi(X_0)]\big)\\
=&b\big(\E_\pi[\xi(X_0)\xi(X_1)]-\E_\pi[\xi(X_0)\xi(X_3)]\big)
-c\big(\E_\pi[\xi(X_0)]-\E_\pi[\xi(X_0)\xi(X_2)]\big)\\
=&b\big(\E_\pi[\xi(X_0)\dxi(X_3)]-\E_\pi[\xi(X_0)\dxi(X_1)]\big)
-c\E_\pi[\xi(X_0)\dxi(X_2)],
\end{align*}
as required. Notice that we use the reversibility of $q$ in the second equality above. 

The proof of (\ref{eq:Dbar2}) is similar and the reversibility of $q$ is used again. We have
\begin{align*}
&\sum_{x,y\in E}\pi(x)q(x,y)[A(x,\xi)-B(y,\xi)]^2\\
=&\sum_{x,y\in E}\pi(x)q(x,y)\big(-\E_x[b\xi(X_2)-c\xi(X_1)]+\E_y[b\xi(X_1)-c\xi(X_0)]\big)^2\\
=&\E_\pi\big[\big(b\xi(X_2)-c\xi(X_1)\big)(b\xi(Y_2)-c\xi(Y_1)\big)\big]-2\E_\pi\big[\big(b\xi(X_2)-c\xi(X_1)\big)(b\xi(Y_2)-c\xi(Y_1)\big)\big]\\
&+\E_\pi\big[\big(b\xi(X_1)-c\xi(X_0)\big)(b\xi(Y_1)-c\xi(Y_0)\big)\big]\\
=&-\E_\pi\big[\big(b\xi(X_2)-c\xi(X_1)\big)(b\xi(Y_2)-c\xi(Y_1)\big)\big]
+\E_\pi\big[\big(b\xi(X_1)-c\xi(X_0)\big)(b\xi(Y_1)-c\xi(Y_0)\big)\big]\\
=&b^2\E_\pi[\xi(X_0)\dxi(X_4)-\xi(X_0)\dxi(X_2)]-2bc\E_\pi[\xi(X_0)\dxi(X_3)-\xi(X_0)\dxi(X_1)]+c^2\E_\pi[\xi(X_0)\dxi(X_2)].
\end{align*}
This completes the proof. 
\end{proof}

\section{Feynman-Kac duality for voter models}\label{sec:FK}
In this section, we give a brief discussion of the Feynman-Kac duality for voter models, which we use in the earlier sections. 
Although these results are usually thought to be standard, they seem difficult to find in the literature.

We introduce some functions. First we set 
$J_\Sigma(\xi;x)=\1_\Sigma\big(\xi(x)\big)$
for any subset $\Sigma$ of $S$ and $x\in E$. Then for any $m$-tuple $A=(x_1,\cdots,x_m)$ of points of $E$ and $k$-tuple $\bs \Sigma=(\Sigma_1,\cdots,\Sigma_m)$ of subsets of $S$, we define
\begin{align*}
H_{\bs \Sigma}(\xi;A)= &\prod_{i=1}^m[J_{\Sigma_i}(\xi;x_i)-\overline{\mu}(\Sigma_i)],
\end{align*}
where $\overline{\mu}$ is defined by (\ref{def:mubar}).
We write $\mathsf L_{B,m}$ for the generator of a $m$-tuple of coalescing $q$-Markov chains, which allows for the possibility that there are less than $m$ distinct points in the system.   

\begin{prop}\label{prop:dual0}
For any $m$-tuple $A$ with distinct entries and $m$-tuple $\bs \Sigma$ of subsets of $S$, we have 
\begin{align}\label{gen1}
\mathsf L^{0,\mu} H_{\bs \Sigma}(\,\cdot\,;A)(\xi)=\big[\mathsf L_{B,\ell}-m\mu(\1)\big]H_{\bs \Sigma}(\xi;\,\cdot\,)(A).
\end{align}
\end{prop}
\begin{proof}
Write $A=(x_1,\cdots,x_m)$ and $\bs \Sigma=(\Sigma_1,\cdots,\Sigma_m)$.
For any $x\in A$ and $y\in E$, define $A^{x,y}$ to be the $m$-tuple obtained from $A$ by replacing the entry $x$ of $A$ with $y$. 
Then by the definitions of $\mathsf L^{0,0}$ in (\ref{def:Lw}), we have
\begin{align*}
&\mathsf L^{0,0}H_{\bs \Sigma}(\,\cdot\,;A)(\xi)=\sum_{x\in E}\left(\sum_{y\in E}q(x,y)[\xi(x)\dxi(y)+\dxi(x)\xi(y)]\right)\big[H_{\bs \Sigma}(\xi^x;A)-H_{\bs \Sigma}(\xi;A)\big]\\
=&\sum_{i=1}^m\sum_{y\in E}q(x_i,y)[\xi(x_i)\dxi(y)+\dxi(x_i)\xi(y)]\left[\1_{\Sigma_i}\big(\dxi(x)\big)-\1_{\Sigma_i}\big(\xi(x_i)\big)\right] \prod_{j:j\neq i} \big[\1_{\Sigma_j}\big(\xi(x_j)\big)-\overline{\mu}\big(\Sigma_j\big)\big]\\
=&\sum_{i=1}^m\sum_{y\in E}q(x_i,y)\left[\1_{\Sigma_i}\big(\xi(y)\big)-\1_{\Sigma_i}\big(\xi(x_i)\big)\right] \prod_{j:j\neq i} \big[\1_{\Sigma_j}\big(\xi(x_j)\big)-\overline{\mu}\big(\Sigma_j\big)\big]\\
=&\sum_{i=1}^m\sum_{y\in E}q(x_i,y)[H_{\bs \Sigma}(\xi;A^{x_i,y})-H_{\bs \Sigma}(\xi;A)]
=\mathsf L_{B,m}[H_{\bs \Sigma}(\xi;\,\cdot\,)](A),
\end{align*}
where the second equality follows from the assumption that the entries of $A$ are distinct.  
The mutation part of $\mathsf L^{0,\mu}$ is given by
\begin{align*}
&\sum_{x\in E}\int_S \big(H_{\bs \Sigma}(\xi^{x|\sigma};A)-H_{\bs \Sigma}(\xi;A)\big)d\mu(\sigma)\\
=&\sum_{i=1}^m\int_S
\big[\1_{\Sigma_i}\big(\xi^{x_i|\sigma}(x_i)\big)-\1_{\Sigma_i}\big(\xi(x_i)\big)\big] \prod_{j:j\neq i} \left[\1_{\Sigma_j}\big(\xi(x_j)\big)-\overline{\mu}\big(\Sigma_j\big)\right]d\mu(\sigma)\\
=&\sum_{i=1}^m\left[\mu(\Sigma_i) -\mu(\1)\1_{\Sigma_i}\big(\xi(x_i)\big)\right]\prod_{j:j\neq i}\big[\1_{\Sigma_j}\big(\xi(x_j)\big)-\overline{\mu}(\Sigma_j)\big]
=-m\mu(\1)H_{\bs \Sigma}(\xi;A).
\end{align*}
The above two displays give the required equation in  (\ref{gen1}). 
\end{proof}

Proposition~\ref{prop:dual0} is enough to solve for $\E_\xi[H_{\bs \Sigma}(\xi_t;A)]$ by coalescing Markov chains. For example, for $m=1$, Proposition~\ref{prop:dual0} shows that
\begin{align}\label{eq:gh1}
\mathsf L^{0,\mu} H_{\Sigma}(\,\cdot\,;x)(\xi)=&\big[\mathsf L_{B,1}-\mu(\1)\big]H_{ \Sigma}(\xi;\,\cdot\,)(x).
\end{align}
The foregoing equation can be used to find $\mathsf L^{0,\mu} H_{(\Sigma_1,\Sigma_2)}(\cdot;x,x)(\xi)$ as follows. We write (\ref{eq:gh1}) as
\begin{align}\label{eq:gh11}
\mathsf L^{0,\mu} J_{\Sigma}(\,\cdot\,;x)(\xi)=\mathsf L^{0,\mu} H_{\Sigma}(\,\cdot\,;x)(\xi)=&\mathsf L_{B,1}J_\Sigma(\xi;\,\cdot\,)(x)-\mu(\1)J_\Sigma(\xi;x)+\mu(\Sigma)
\end{align}
so that
\begin{align}
&\mathsf L^{0,\mu}H_{(\Sigma_1,\Sigma_2)}(\,\cdot\,;x,x)(\xi)
=\mathsf L^{0,\mu}\big[J_{\Sigma_1\cap \Sigma_2}(\,\cdot\,;x)-
J_{\Sigma_1}(\,\cdot\,;x)\overline{\mu}(\Sigma_2)-J_{\Sigma_2}(\,\cdot\,;x)\overline{\mu}(\Sigma_1)
\big](\xi)\notag\\
=&\mathsf L_{B,1}\big[ J_{\Sigma_1\cap \Sigma_2}(\xi;\,\cdot\,)-
\overline{\mu}(\Sigma_2)J_{\Sigma_1}(\xi;\,\cdot\,)-
\overline{\mu}(\Sigma_1)J_{\Sigma_2}(\xi;\,\cdot\,)+\overline{\mu}(\Sigma_1)\overline{\mu}(\Sigma_2)\big](x)\notag\\
&-\mu(\1)\big[ J_{\Sigma_1\cap \Sigma_2}(\xi;x)-
\overline{\mu}(\Sigma_2)J_{\Sigma_1}(\xi;x)-
\overline{\mu}(\Sigma_1)J_{\Sigma_2}(\xi;x)+\overline{\mu}(\Sigma_1)\overline{\mu}(\Sigma_2)\big]\notag\\
&+\mu(\1)\big[\overline{\mu}(\Sigma_1\cap \Sigma_2)-\overline{\mu}(\Sigma_1)\overline{\mu}(\Sigma_2)\big]\notag\\
\begin{split}\notag 
=&\mathsf L_{B,2}H_{(\Sigma_1,\Sigma_2)}(\xi;\,\cdot\,,\,\cdot\,)(x,x)
-\mu(\1)H_{(\Sigma_1,\Sigma_2)}(\xi;x,x)
+\mu(\1)\big[\overline{\mu}(\Sigma_1\cap \Sigma_2)-\overline{\mu}(\Sigma_1)\overline{\mu}(\Sigma_2)\big].
\end{split}
\end{align}
We can summarize the last equality and (\ref{gen1}) with $m=2$ as the following equation:
\begin{align}\label{eq:gh2}
\begin{split}
\forall\;x,y\in E,\quad \mathsf L^{0,\mu} H_{(\Sigma_1,\Sigma_2)}(\,\cdot\,;x,y)(\xi)=&\big[\mathsf L_{B,2}- \mu(\1) |\{x,y\}|\big]H_{(\Sigma_1,\Sigma_2)}(\xi;\,\cdot\,,\,\cdot\,)(x,y)\\
&+\mu(\1) \big[\overline{\mu}(\Sigma_1\cap \Sigma_2)-\overline{\mu}(\Sigma_1)\overline{\mu}(\Sigma_2)\big]\1_{\{x=y\}}.
\end{split}
\end{align}
This equation is enough for (\ref{MG}) in particular. 

\section*{List of frequent notations}
\mbox{}\smallskip

 $S$: the set of types $\{1,0\}$. \\
\indent  $\Pi=\big(\Pi(\sigma,\tau)\big)$: $2\times 2$ payoff matrices with real entries.\\
\indent $(E,q)$: a kernel on $E$ assumed to have a zero trace and be irreducible and reversible.\\
\indent $R_\ell$: a limiting return probability of voting kernels defined in Assumption~\ref{ass:homo}.\\
\indent $\pi$: the stationary distribution of $q$. \\
\indent $\pi_{\min},\pi_{\max}$: $\pi_{\min}=\min_x\pi(x)$ and $\pi_{\max}=\max_x\pi(x)$. \\
\indent $\nu(x,y)$: the measure on $E\times E$ defined by $\nu(x,y)=\pi(x)^2q(x,y)$ in (\ref{def:nu}).\\
\indent $\mu$: a mutation measure defined on $S$ (Section~\ref{sec:poisson}).\\
\indent $\overline{\mu}(\sigma)$: the ratio $\mu(\sigma)/\mu(\1)$ with the convention that $0/0=0$ defined in (\ref{def:mubar}).\\
\indent   $w$: selection strength (Section~\ref{sec:poisson}).\\
\indent $\overline{w}$: a maximal selection strength defined by $  \big(2+2\max_{\sigma,\tau\in S}|\Pi(\sigma,\tau)|\big)^{-1}$ in (\ref{def:w0}).\\
\indent $\gamma_n$: a constant time change applied to evolutionary games.\\
\indent $\P^w,\P^{(n),w}$: laws of evolutionary games subject to selection strength $w$ (Section~\ref{sec:poisson}).\\
\indent $\ms P(U)$: the set of probability measures defined on a Polish space $U$.\\

\indent \emph{Functions of configurations}
\smallskip

\indent $\xi,\eta$: $\{1,0\}$-valued population configurations .\\
\indent $\dxi(x)$: $1-\xi(x)$ defined in (\ref{def:hat}).\\ 
\indent $p_\sigma(\xi)$: the density of type $\sigma$ in a population configuration $\xi$ defined in (\ref{def:p1}). \\
\indent $W_\ell(\xi)$: the density function defined by $W_\ell(\xi) =\sum_{x,y\in E}\pi(x)q^{\ell}(x,y)\xi(x)\dxi(y)$ in (\ref{def:Pell}).\\
\indent $H(\xi;x,y)$: the dual function defined by  $H(\xi;x,y)=\big[\xi(x)-\overline{\mu}(1)\big]\big[\dxi(y)-\overline{\mu}(0)\big]$ in (\ref{def:H}).\\

\indent \emph{Processes}
\smallskip

\indent $(\Lambda_t(x,y)),(\Lambda^\sigma_t(x))$: the Poisson processes defined in (\ref{rates}).\\
\indent $(Y_t)$ under $\P$ or $\P^{(n)}$: the density process of $1$'s defined in (\ref{eq:Y}).\\ 
\indent $(M_t)$ under $\P$ or $\P^{(n)}$: the martingale part of $(Y_t)$ according to the decomposition in (\ref{eq:M}).\\ 
\indent $(D_t^w)$ under $\P$ or $\P^{(n)}$: the Radon-Nikodym derivative process defined in (\ref{def:D}).\\ 
\indent $Z^{(n)}$ under $\P^{(n)}$: the process $(Y_{\gamma_nt},M_{\gamma_nt},D^{w_n}_{\gamma_nt})$ under $\P^{(n)}$ defined in (\ref{def:Zn}).\\ 
\indent $(Y_t,M_t,D_t)$ under $\P^{(\infty)}$: the limit of $Z^{(n)}$ under $\P^{(n)}$ (Theorem~\ref{thm:main2}).\\ 
\indent $\{(B^x_t);x\in E\}$: coalescing $q$-Markov chains (Section~\ref{sec:weak}).\\
\indent $M_{x,y}$: the first meeting time of $B^x$ and $B^y$ (Section~\ref{sec:weak}).\\
\indent $(X_\ell),(Y_\ell)$: auxiliary discrete-time $q$-Markov chains (Section~\ref{sec:tight}). \\


\begin{thebibliography}{9}
\bibitem{AF_MC} 
Aldous, D.J. and Fill, J.A. (2002). \emph{Reversible Markov Chains and Random Walks on Graphs}. Unfinished monograph, recompiled 2014. Available at \href{https://www.stat.berkeley.edu/~aldous/RWG/book.html}{https://www.stat.berkeley.edu/~aldous/RWG/book.html}.


\bibitem{Aldous_1982}
Aldous, D.J. (1982). Markov chains with almost exponential hitting times. \emph{Stochastic Process. Appl.} {\bf 13} 305--310. \href{https://doi.org/10.1016\%2F0304-4149\%2882\%2990016-3}{doi:10.1016/0304-4149(82)90016-3}

\bibitem{Bordenave_2015}
Bordenave, C. (2015). A new proof of Friedman's second eigenvalue theorem and its extension to random lifts. Preprint. Available at \href{http://arxiv.org/abs/1502.04482}{arXiv:1502.04482}.


\bibitem{Benjamini_2001}
Benjamini, I. and Schramm, O. (2001). Recurrence of distributional limits of finite planar graphs. \emph{Electron. J. Probab.} {\bf 6} no. 23, 13. \href{http://dx.doi.org/10.1214/ejp.v6-96}{doi:10.1214/ejp.v6-96}

\bibitem{CC_16}
Chen, Y.-T. and Cox, J.T. (2016). Weak atomic convergence of finite voter models toward Fleming-Viot processes. Preprint. Available at
\href{https://arxiv.org/abs/1608.05736}{arXiv:1608.05736}.


\bibitem{CCC}
Chen, Y.-T., Choi, J. and Cox, J.T. (2016). On the convergence of densities of finite voter models to the Wright-Fisher diffusion. 
\emph{Ann. Inst. Henri Poincar\'e Probab. Stat.} {\bf 52} 286--322. \href{http://dx.doi.org/10.1214/14-aihp639}{doi:10.1214/14-aihp639}

\bibitem{CDP_2000}
Cox, J.T., Durrett, R. and Perkins, E.A. (2000). Rescaled voter models converge to super-Brownian motion. \emph{Ann. Probab.} {\bf 28} 185--234. \href{http://dx.doi.org/10.1214/aop/1019160117}{doi:10.1214/aop/1019160117}

\bibitem{CDP}
Cox, J.T., Durrett, R. and Perkins, E.A. (2013). Voter model perturbations and reaction diffusion equations. \emph{Ast\'erisque} {\bf 349} \href{http://www.ams.org/mathscinet-getitem?mr=3075759}{MR3075759}

\bibitem{Chen_2013}
Chen, Y.-T. (2013). Sharp benefit-to-cost rules for the evolution of cooperation on regular graphs. \emph{Ann. Appl. Probab.} {\bf 23} 637--664. \href{http://dx.doi.org/10.1214/12-aap849}{doi:10.1214/12-aap849}


\bibitem{Chen_2015}
Chen, Y.-T. , McAvoy, A. and Nowak, M.A. (2016). Fixation probabilities for any configuration of two strategies on regular graphs. \emph{Scientific Reports} {\bf 6} 39181. \href{http://www.nature.com/articles/srep39181}{doi:10.1038/srep39181}

\bibitem{Cox_2004}
Cox, J.T. and Perkins, E.A. (2004). An application of the voter model--super-Brownian motion invariance principle. \emph{Ann. Inst. Henri Poincar\'e Probab. Stat.} {\bf 40} 25--32. \href{https://doi.org/10.1016\%2Fs0246-0203\%2803\%2900046-3}{doi:10.1016/s0246-0203(03)00046-3}

\bibitem{EK:MP}
Ethier, S.N. and Kurtz, T.G. (1986). \emph{Markov Processes. Characterization and Convergence}, 2nd ed. \emph{Wiley Series in Probability and Mathematical Statistics: Probability and Mathematical Statistics}. John Wiley \& Sons, Inc., New York. \href{http://www.ams.org/mathscinet-getitem?mr=838085}{MR0838085}


\bibitem{Friedman_2008}
Friedman, J. (2008). A proof of Alon's second eigenvalue conjecture and related problems. \emph{Mem. Amer. Math. Soc.} {\bf 195}.
\href{http://dx.doi.org/10.1090/memo/0910}{doi:10.1090/memo/0910}


\bibitem{Granovsky_1995}
Granovsky, B.L.  and Madras, N. (1995). The noisy voter model. \emph{Stochastic Process. Appl.} {\bf 55} 23--43. \href{http://dx.doi.org/10.1016/0304-4149(94)00035-r}{doi:10.1016/0304-4149(94)00035-r}


\bibitem{Rani}
Hod, R. (2016). Personal communication.




\bibitem{JS}
Jacod, J. and Shiryaev, A.N. (2003). \emph{Limit Theorems for Stochastic Processes}, 2nd ed. \emph{Grundlehren der Mathematischen Wissenschaften} {\bf 288}. Springer-Verlag, Berlin. \href{http://link.springer.com/book/10.1007/978-3-662-05265-5}{doi:10.1007/978-3-662-05265-5}  


\bibitem{Keilson_1979}
Keilson, J. (1979). \emph{Markov Chain Models --- Rarity and Exponentiality}. Springer, New York. \href{https://doi.org/10.1007\%2F978-1-4612-6200-8}{doi:10.1007/978-1-4612-6200-8}


\bibitem{L:IPS}
Liggett, T.M. (2005). \emph{Interacting Particle Systems}, reprint of the 1985 original. \emph{Classics in Mathematics}. Springer-Verlag, Berlin. \href{http://link.springer.com/book/10.1007/b138374}{doi:10.1007/b138374
}


\bibitem{LWP}
Levin, D.A., Peres, Y. and Wilmer, E.L. (2009). \emph{Markov Chains and Mixing Times}, with a chapter by James G. Propp and David B. Wilson. American Mathematical Society, Providence, RI. \href{https://doi.org/10.1090\%2Fmbk\%2F058}{doi:10.1090/mbk/058}

\bibitem{Mal_cot_1975}
Mal\'ecot, G. (1975). Heterozygosity and relationship in regularly subdivided populations. \emph{Theor. Popul. Biol.} {\bf 8} 212--241. \href{http://dx.doi.org/10.1016/0040-5809(75)90033-7}{doi:10.1016/0040-5809(75)90033-7}


\bibitem{McKay_1981}
McKay, B.D. (1981). The expected eigenvalue distribution of a large regular graph. \emph{Linear Algebra Appl.} {\bf 40} 203--216. \href{http://dx.doi.org/10.1016/0024-3795(81)90150-6}{doi:10.1016/0024-3795(81)90150-6}

\bibitem{Moran}
Moran, P.A.P. (1958). Random processes in genetics. \emph{Mathematical Proceedings of the Cambridge Philosophical Society} {\bf 54} 60--71. \href{https://doi.org/10.1017/S0305004100033193}{doi:10.1017/S0305004100033193}

\bibitem{Matsuda_1992}
Matsuda, H., Ogita,  N., Sasaki, A. and Sat\={o}, K. (1992). Statistical mechanics of population: the lattice Lotka-Volterra model. \emph{Prog. Theor. Phys.} {\bf  88} 1035--1049. \href{http://dx.doi.org/10.1143/ptp.88.1035}{doi:10.1143/ptp.88.1035}

\bibitem{MT}
Mueller, C. and Tribe, R. (1995). Stochastic p.d.e.'s arising from the long range contact and long range voter processes. \emph{Probab. Theory Related Fields} {\bf 102} 519--545.
\href{http://link.springer.com/article/10.1007/BF01198848}{doi:10.1007/BF01198848}



\bibitem{Nowak_2010}
Nowak, M.A., Tarnita, C.E.  and Wilson, E.O. (2010). The evolution of eusociality. \emph{Nature} {\bf 466} 1057--1062. \href{http://dx.doi.org/10.1038/nature09205}{doi:10.1038/nature09205}


\bibitem{Ohtsuki_2006}
Ohtsuki, H., Hauert, C., Lieberman, E. and Nowak, M.A. (2006). A simple rule for the evolution of cooperation on graphs and social networks. \emph{Nature} {\bf 441} 502--505.
\href{http://www.nature.com/nature/journal/v441/n7092/full/nature04605.html}{doi:10.1038/nature04605}


\bibitem{Oliveira_2012}
Oliveira, R.I. (2012). On the coalescence time of reversible random walks. \emph{Trans. Amer. Math. Soc.} {\bf 364} 2109--2128.
\href{http://dx.doi.org/10.1090/s0002-9947-2011-05523-6}{doi:10.1090/s0002-9947-2011-05523-6}

\bibitem{Oliveira_2013}
Oliveira, R.I. (2013). Mean field conditions for coalescing random walks. \emph{Ann. Probab.} {\bf 41} 3420--3461. 
\href{http://dx.doi.org/10.1214/12-aop813}{doi:10.1214/12-aop813}


\bibitem{Rousset}
Rousset, F. (2004). \emph{Genetic Structure and Selection in Subdivided Populations}. \emph{Monographs in Population Biology} {\bf 40}. Princeton University Press, New Jersey.

\bibitem{Revuz_2005}
Revuz, D. and Yor, M. (2005). \emph{Continuous Martingales and Brownian Motion}, 3rd ed. \emph{Grundlehren der Mathematischen Wissenschaften} {\bf 293}. Springer-Verlag, Berlin.
\href{http://dx.doi.org/10.1007/978-3-662-06400-9}{doi:10.1007/978-3-662-06400-9}

\bibitem{Sui_2015}
Sui, X.,  Wu, B. and Wang, L. (2015). Speed of evolution on graphs. \emph{Phys. Rev. E} {\bf 92} 062124.
\href{http://dx.doi.org/10.1103/physreve.92.062124}{doi:10.1103/physreve.92.062124}




\bibitem{Vallender_1974}
Vallender, S.S. (1974). Calculation of the Wasserstein distance between probability distributions on the line. \emph{Theory Probab. Appl.} {\bf 18} 784--786.
\href{https://doi.org/10.1137\%2F1118101}{doi:10.1137/1118101}




\end{thebibliography}
\end{document}